\documentclass[10pt,reqno]{amsart}
\usepackage{fancyhdr}
\usepackage{amsmath}
\usepackage{amsthm}
\usepackage{amssymb}
\usepackage{mathabx}
\usepackage{graphicx}
\usepackage[usenames]{color}
\DeclareMathAlphabet{\mathpzc}{OT1}{pzc}{m}{it}

\newtheorem{theorem}{Theorem}[section]
\newtheorem{proposition}[theorem]{Proposition}
\newtheorem{lemma}[theorem]{Lemma}
\newtheorem{corollary}[theorem]{Corollary}

\newtheorem{definition}[theorem]{Definition}
\newtheorem{remark}[theorem]{Remark}

\def\Holder{{H\"{o}lder}}

\def\bdy #1{{\partial #1\hspace{1pt}}}

\def\bbR{{\mathbb R}}

\def\div{{\operatorname{div}}}
\def\Def{{\operatorname{Def}}}

\def\id{{\text{Id}}}
\def\supp{{\text{spt}}}

\def\cptsubset{\hspace{1pt}{\subset\hspace{-2pt}\subset}\hspace{1pt}}

\def\contsubset{\hspace{1pt}{\hookrightarrow}\hspace{1pt}}

\def\B{{\mathcal B}}
\def\C{{\mathcal C}}

\def\H{{\mathcal H}}

\def\L{{\mathcal L}}

\def\O{{\mathcal O}}
\def\P{{\mathcal P}}
\def\Q{{\mathcal Q}}
\def\R{{\mathcal R}}

\def\U{{\mathcal U}}
\def\V{{\mathcal V}}
\def\W{{\mathcal W}}

\def\rA{{\rm A}}
\def\rB{{\rm B}}
\def\rC{{\rm C}}

\def\rE{{\rm E}}
\def\rF{{\rm F}}
\def\rG{{\rm G}}
\def\rH{{\rm H}}

\def\rJ{{\rm J}}

\def\rL{{\rm L}}
\def\rM{{\rm M}}
\def\rN{{\rm N}}
\def\rO{{\rm O}}
\def\rP{{\rm P}}

\def\rR{{\rm R}}

\def\rT{{\rm T}}

\def\rV{{\rm V}}

\def\rX{{\rm X}}

\def\rb{{\rm b}}

\def\re{{\rm e}}

\def\rg{{\rm g}}

\def\ve{{\varepsilon}}
\def\eps{{\epsilon}}

\def\bJ{{\widebar{\rJ}}}
\def\bA{{\widebar{\rA}}}

\def\bw{{\widebar{w}}}

\def\bh{{\widebar{h}}}
\def\bpsi{{\widebar{\psi}}}

\def\bp{{\overline{\partial}\hspace{1pt}}}
\def\p{{\partial\hspace{1pt}}}

\def\n{{\rm n}}

\def\Forall{\forall\hspace{2pt}}

\def\comm#1#2{{\big[\hspace{-3.2pt}\big[#1,#2\hspace{1pt}\big]\hspace{-3.2pt}\big]}}
\def\bigcomm#1#2{{\Big[\hspace{-4pt}\Big[#1,#2\hspace{1pt}\Big]\hspace{-4pt}\Big]}}
\def\({{(\hspace{-2pt}(}}
\def\){{)\hspace{-2pt})}}

\def\smallexp#1{{\text{\small #1}}}

\def\triplenorm#1{{|\hspace{-1.2pt}|\hspace{-1.2pt}|{#1}|\hspace{-1.2pt}|\hspace{-1.2pt}|}}
\def\triplenorm#1{{\vvvert #1\vvvert}}
\def\pprime{{\hspace{1pt}\prime}}

\def\XXint#1#2#3{{\setbox0=\hbox{$#1{#2#3}{\int}$}
\vcenter{\hbox{$#2#3$}}\kern-.5\wd0}}

\def\bfJ{{\sf J}}
\def\bfA{{\sf A}}
\def\bfa{{\bf a}}
\def\bfg{{\sf g}}
\def\bfv{{\sf v}}
\def\bfw{{\sf w}}
\def\bfh{{\sf h}}
\def\bfq{{\sf q}}
\def\wfJ{{\widetilde{\bfJ}}}
\def\wfA{{\widetilde{\bfA}}}
\def\wfa{{\widetilde{\bfa}}}
\def\wfv{{\widetilde{\bfv}}}

\def\opbh{{1 \hspace{-1pt}+\hspace{-1pt} \rb_0 \bfh_{\ve\ve}}}

\title[The existence of solutions of 2-dimensional incompressible Navier-Stokes equations with surface tension in an optimal Sobolev space]{The existence of solutions of 2-dimensional incompressible Navier-Stokes equations on a moving domain in an optimal Sobolev space}

\author[C.H. A. Cheng]{C.H. Arthur Cheng}
\email{cchsiao@math.ncu.edu.tw}
\address{Department of Mathematics, National Central University, Jhongli City, Taoyuan County, 32001, Taiwan ROC}

\author[Y.C. Lin]{Ying-Chieh Lin}
\email{linyj@math.ncu.edu.tw}
\address{Department of Mathematics, National Central University, Jhongli City, Taoyuan County, 32001, Taiwan ROC}

\author[C.F. Su]{Cheng-Fang Su}
\email{@math.nctu.edu.tw}
\address{Department of Mathematics, National Central University, Jhongli City, Taoyuan County, 32001, Taiwan ROC}


\begin{document}

\begin{abstract}
We establish the existence of a solution to the Navier-Stokes equations on a moving domain with surface tension in an optimal Sobolev space for the case of two space dimension. No compatibility conditions are required to guarantee the existence of a solution.
\end{abstract}

\maketitle
{\small
\tableofcontents}

\section{Introduction}\label{sec:introduction}
\setcounter{equation}{0}
\subsection{The equations}
We are concerned with the 2-dimensional Navier-Stokes equations on a moving domain $\Omega(t)$ with surface tension on the moving boundary $\bdy \Omega(t)$. Let $\Omega$ be a bounded domain of $\bbR^2$ (the regularity of $\Omega$ will be specified later) which denotes the initial fluid domain, and $u$ and $p$ denote the fluid velocity and pressure, respectively. We consider
\begin{subequations}\label{NS}
\begin{alignat}{2}
%\eta_t &= u\circ \eta \qquad&&\text{in}\quad\Omega\times (0,\rT)\,,\\
%\Omega(t) &= \eta(\Omega) &&\text{for}\quad t\in [0,T)\,,\\
u_t + (u\cdot \nabla) u + \nabla p &= \Delta u \qquad&&\text{in}\quad \Omega(t)\,,\\
\div u &= 0 &&\text{in}\quad\Omega(t)\,,\\
(\Def u - p\,\id)n &= \sigma \rH n \qquad&&\text{on}\quad\bdy\Omega(t)\,,\\
u &= u_0 &&\text{on}\quad \Omega\times \{t=0\}\,,\\
\V(\bdy \Omega(t)) &= u\cdot n &&\text{on}\quad\bdy\Omega(t)\,,
\end{alignat}
\end{subequations}
where %$\eta$ is the flow map of the fluid velocity (and is also called the Lagrangian coordinate),
the viscosity of the fluid is assumed to be $1$, $n$ is the outward-pointing unit normal of $\Omega(t)$, $\rH$ is the mean curvature of the boundary of $\Omega(t)$, $\sigma>0$ is the surface tension, $e$ the identity map defined by $e(x) = x$, and $\V(\bdy \Omega(t))$ denotes the normal velocity of the moving boundary $\bdy\Omega(t)$.

\subsection{Some prior results}
Free boundary problems are one of the main sources of highly nonlinear PDEs that require special treatment to establish well-posedness. Due to high nonlinearity, even the local-in-time solution are expected to exist only in spaces with high regularity, and such kind of results usually accompany certain orders of compatibility conditions if the governing equations are of parabolic type. For example, Xinfu Chen \& Fernando Reitich \cite{xfchen1992} established the well-posedness of the Stefan problem with surface tensions provided that the initial data satisfies two compatibility conditions. In the study of the interaction between incompressible viscous fluids and elastic shells, two compatibility conditions also have to be imposed for the purpose of the existence and uniqueness of the solution (see \cite{cchsiao2006} and \cite{cchsiao2010} for the detail).

When considering compressible or incompressible Navier-Stokes with surface tensions, one compatibility condition has to be imposed in order to guarantee the well-posedness in previous literatures. In the compressible case, if $\rho_0$ and $u_0$ denote the initial fluid density and velocity respectively, by imposing the compatibility condition
$$
\big[\mu \Def u_0 + \big(\lambda \div u_0 - p(\rho_0)\big) \id\big] n(0) - \sigma \rH(0) n(0) = - p_e n(0) \ \quad\text{on}\quad \bdy\Omega\,,
$$
Solonnikov \& Tani \cite{Solonnikov1990} showed the existence of a unique solution. In the incompressible case, if $\rP_{\text{tan}}$ denotes the projection map onto the tangent bundle of $\bdy\Omega$ and $u_0$ is the initial velocity, the first order compatibility condition reads
\begin{equation}\label{compatibility_condition}
\rP_{\text{tan}} \big(\Def u_0 n(0)\big) = 0 \ \quad\text{on}\quad \bdy\Omega,
\end{equation}
and under the assumption that $u_0 \in H^2(\Omega)$ satisfies the first order compatibility condition Shkoller \& Coutand \cite{steve2003} established the well-posedness of the equation. We remark that compatibility condition (\ref{compatibility_condition}) does not involve the initial pressure $p_0$.

%In order to obtain a solution with such high regularity, the compatibility condition (\ref{compatibility_condition}) has to be imposed on the initial data $u_0$.
To illustrate the importance of our work in this paper, we emphasize that the compatibility conditions put a lot of constraint on the initial data, especially when considering the numerical simulation in which case the initial data can be given in almost arbitrary fashions. For example, for the incompressible case, if $\Omega = B(0,1)$ and $u_0$ is given by
$$
u_0(x,y) = \big(F(y),G(x)\big)\,.
$$
Then (\ref{compatibility_condition}) holds only when
\begin{equation}\label{FG}
F^\pprime(\sin\theta) + G^\pprime(\cos\theta) = 0 \qquad\Forall \theta\in (0,2\pi)\,,
\end{equation}
while we know that it is easy to find $F$ and $G$ such that (\ref{FG}) does not hold.

\subsection{The difficulties}\label{sec:difficulties}
When looking for the solution of the velocity possessing only $H^2$ spatial regularity (which is the optimal Sobolev space for a strong solution to exist), it is not clear how the moving boundary $\Gamma(t)$ is defined since the Lagrangian flow map $\eta$ satisfying the ODE
\begin{alignat*}{2}
\eta_t(x,t) &= u\big(\eta(x,t),t\big) \qquad&&\Forall x\in \Omega,t> 0\,,\\
\eta(x,0) &= x &&\Forall x\in \Omega\,,
\end{alignat*}
is in general not solvable (uniquely) due to the lack of Lipschitz continuity (on the other hand, $u\in C^{0,\alpha}(\Omega(t))$ for all $\alpha\in (0,1)$ and $t>0$ because of the Sobolev embedding $C^{0,\alpha}(\Omega(t)) \contsubset H^2(\Omega(t))$ if $\n=2$). Therefore, it is not adequate to describe the moving boundary using the Lagrangian flow map $\eta$. Moreover, due to the lack of regularity of the solution, the nonlinearity appears to be much stronger, and some standard ways of constructing solutions fail to work (see Remark \ref{rmk:issues} for the detail), even though the a priori estimates can be easily derived.

%Therefore, a new way of describing the phenomena for which construction of solution is possible is needed.

\subsection{Outlines}
In Section \ref{sec:ALE_formulation}, %we briefly explain the connection between the usage of Lagrangian coordinate in the analysis of the equation and the need of certain order of compatibility conditions. W
we introduce the ALE map which can describe the time-dependent domain $\Omega(t)$ with free boundary moving along with an $H^2$-velocity. However, reasoning in Remark \ref{rmk:issues}, the ALE formulation is still not good enough for the purpose of constructing solutions with optimal regularity, so we slightly modify the ALE formulation in the last part of this section (for the purpose of constructing an approximated solution). The functional framework are then introduced in Section \ref{sec:notation}, and some preliminary results are established in this section as well. The main theorem is stated in Section \ref{sec:main_thm}, and we prove the main theorem from Section \ref{sec:construction} to Section \ref{sec:time_continuation}, including the introduction of an approximated regularized problem (with a smooth parameter $\ve$) as well as the construction of a solution to this particular approximation in Section \ref{sec:construction}, the $\ve$-independent estimates in Section \ref{sec:ve_indep_est}, and the argument of the continuation of time in Section \ref{sec:time_continuation}. %, and finally the uniqueness of the solution in Section \ref{sec:uniqueness}. %Finally, we present some numerical results in Section \ref{sec:numerics}.

\section{The ALE formulation}\label{sec:ALE_formulation}

\subsection{A map $\psi$ that maps from a fixed reference domain to $\Omega(t)$}\label{Gamma}
Let $\Gamma$ be a smooth closed curve in the tubular neighborhood of $\bdy\Omega$, $\rO$ be the region enclosed by $\Gamma$, and $\rN$ is the outward-pointing unit normal of $\rO$ such that each point $y\in \bdy\Omega$ corresponds to a unique $x\in \Gamma$ such that $y = x + h_0(x) \rN(x)$ for some $h_0$; that is, $\bdy\Omega$ is the graph of $h_0$ over $\Gamma$. Let $h:\Gamma \to \bbR$ denote the signed distance function which measure the signed distance from $\bdy\Omega(t)$ to $\Gamma$. In other words, if $x\in \Gamma$, then the point $x + h(x,t) \rN(x)$ belongs to the curve $\bdy\Omega(t)$. Let $\psi$ be the harmonic extension of the map $e + h \rN$ on $\Gamma$; that is, $\psi$ solves
\begin{subequations}\label{psi_eq}
\begin{alignat}{2}
\Delta \psi &= 0 &&\text{in}\quad\rO\,,\\
\psi &= e + h \rN\qquad &&\text{on}\quad\Gamma\,.
\end{alignat}
\end{subequations}
We remark that if $\|h\|_{L^\infty(\Gamma)} \ll 1$, %$\nabla \psi:\rO \to \bbR^2$ is invertible.
$\psi :\rO \to \Omega(t)$ is a diffeomorphism.

%\subsection{The existence of the smooth domain $\rO$ with $\|h_0\|_{H^2(\Gamma)} \ll 1$}
%In this section, we show that for every $\sigma > 0$ small, there is a smooth domain $\rO_\sigma$ such that
%\begin{enumerate}
%\item The difference of the length of the curves $\bdy\rO_\sigma$ and $\bdy\Omega$ is less than $\sigma$;
%\item If $\rX_\sigma$ is the representation of $\bdy\rO_\sigma$; that is, there exists $\ell_\sigma$ (the length of $\rO_\sigma$) such that $\rX_\sigma: \bbR \to \bbR^2$ is smooth, $\ell_\sigma$-periodic, and $\rO_\sigma$ is the image of $\rX_\sigma$, then for all $k \ge 3$,
% $$
% \|\rX_\sigma\|_{H^k(0,\ell_\sigma)} \le M_{k,\sigma}
% $$
% for some constant $M_{k,\sigma}$, and
% $$
% \|\rX_\sigma\|_{H^2(0,\ell_\sigma)} \le M
% $$
% for some constant generic constant $M$.
%\item $\bdy\Omega$ is in the tubular neighborhood of $\bdy\rO_\sigma$; that is, if $\rN_\sigma$ denote the outward-pointing unit normal to $\bdy\rO_\sigma$, then for any $y\in \bdy\Omega$, there exists $x\in \bdy\rO_\sigma$ such that
% $$
% y = x + h_\sigma(x) \rN(x)
% $$
% for some function $h_\sigma$.
%\end{enumerate}

\subsection{The representation of some geometric quantities}
Let $\ell$ be the length of $\Gamma$, and $\Gamma$ be parametrized by the map $\rX:[-\ell/2,\ell/2)\to \bbR^2$ with arc-length $s$. Then the map
$$
\psi(\rX(s),t) = \rX(s) + h(\rX(s),t) \rN(\rX(s)) \quad\text{on}\quad [-\ell/2,\ell/2)
$$
is a parametrization of the moving curve $\bdy\Omega(t)$.
%\begin{align*}
%\psi &= \rX + h \rN \\
%\psi^\pprime &= \rX' + h' \rN + h \rN^\pprime \\
%g &= 1 + 2 \rb_0 h + h^{\prime\hspace{1pt}2} + h^2 \rb_0^2 = (1 + \rH_0 h)^2 + h^{\prime\hspace{1pt}2} \\
%\psi^{\pprime\prime} &= \rX'' + h'' \rN + 2 h' \rN^\pprime + h \rN^{\pprime\prime} \\
%n\circ\psi &= \frac{-h' \rX' + (1 + \rb_0 h) \rN}{\sqrt{(1 + \rb_0 h)^2 + h^{\prime\hspace{1pt}2}}} \\
%H\circ\psi &= \frac{(1 + \rb_0 h) h'' - 2 h^{\prime\hspace{1pt}2} \rb_0 - (1 + \rb_0 h) \rb_0 - h (1 + \rb_0 h) \rb_0^2 - h h' \rb_0^\pprime}{1^{-1} \big[(1 + \rb_0 h)^2 + 1 h^{\prime\hspace{1pt}2}\big]^{3/2}} \\
%%&= \frac{(1+ \rH_0 h) h'' - 2 h^{\prime\hspace{1pt}2} \rH_0 - (1 + \rb_0 h)(\rH_0 + h \rH_0^2) - h' (1-h \rH_0) \frac{1^\pprime}{2 1} - h h' 1^{-1} \rb_0^\pprime }{\big[(1 + \rb_0 h)^2 + 1 h^{\prime\hspace{1pt}2}\big]^{3/2}}
%\end{align*}

\subsubsection{The metric}
For any function $G$ defined on $\Gamma$, we use the notation ${}^\pprime$ to denote the derivative with respect to the arc-length $s$; that is,
$$
G^\pprime\big(\rX(s),t\big) = \frac{\p}{\p s}\, G\big(\rX(s),t\big) %\quad\text{and}\quad g^\pprime(\rX(s)) = \frac{\p}{\p s} g(\rX(s))
\,.
$$
Then $\psi^\pprime = \rX'\circ\rX^{-1} + h' \rN + h \rN^\pprime$. As a consequence, the metric $\rg$ induced by the map $\psi$ is given by
\begin{equation}\label{defn:g}
\rg \equiv \psi^\pprime \cdot \psi^\pprime = 1 + 2 \rb_0 h + h^{\prime\hspace{1pt}2} + h^2 \rb_0^2 = (1 + \rb_0 h)^2 + h^{\prime\hspace{1pt}2} \quad\text{on}\quad \Gamma\,,
\end{equation}
where $\rb_0 = - (\rX''\circ \rX^{-1}) \cdot \rN = (\rX'\circ \rX^{-1}) \cdot \rN^\pprime$ is the curvature of $\Gamma$ (since $|\rX'| = 1$). We remark that since $\Gamma$ is assumed to be smooth, $\rb_0$ is a smooth function (of $s$).

\subsubsection{The curvature}
Since $\psi^\pprime$ is tangent to $\bdy\Omega(t)$, we find that the normal vector $n$ is given by
\begin{equation}\label{defn:n}
n\circ\psi \hspace{-1pt}=\hspace{-1pt}\frac{-h' (\rX'\circ\rX^{-1}) \hspace{-1pt}+\hspace{-1pt}(1 \hspace{-1pt}+\hspace{-1pt}\rb_0 h) \rN}{\sqrt{(1 \hspace{-1pt}+\hspace{-1pt}\rb_0 h)^2 \hspace{-1pt}+\hspace{-1pt}h^{\prime\hspace{1pt}2}}} \hspace{-1pt}=\hspace{-1pt}\frac{-h' (\rX' \circ\rX^{-1}) \hspace{-1pt}+\hspace{-1pt}(1 \hspace{-1pt}+\hspace{-1pt}\rb_0 h) \rN}{\sqrt{\rg}} \quad\text{on}\quad \Gamma\,.
\end{equation}
Therefore, by the formula $\rH\circ\psi = \rg^{-1} \psi^{\pprime\prime}\cdot (n\circ\psi)$ we obtain that
%\begin{align}
%\rH\circ\psi &= \rg^{-1} \psi^{\pprime\prime} \cdot (n\circ\psi) \nonumber\\
%&= \frac{(1 + \rb_0 h) h'' - 2 h^{\prime\hspace{1pt}2} \rb_0 - (1 + \rb_0 h) \rb_0 - h (1 + \rb_0 h) \rb_0^2 - h h' \rb_0^\pprime}{\big[(1 + \rb_0 h)^2 + h^{\prime\hspace{1pt}2}\big]^{3/2}}\,. \label{defn:H}
%\end{align}
\begin{align}
\rH\circ\psi %\hspace{-1pt}=\hspace{-1pt} \frac{(1 \hspace{-1pt}+\hspace{-1pt} \rb_0 h) h'' \hspace{-1pt}-\hspace{-1pt} 2 h^{\prime\hspace{1pt}2} \rb_0 \hspace{-1pt}-\hspace{-1pt} (1 \hspace{-1pt}+\hspace{-1pt} \rb_0 h) \rb_0 \hspace{-1pt}-\hspace{-1pt} h (1 \hspace{-1pt}+\hspace{-1pt} \rb_0 h) \rb_0^2 \hspace{-1pt}-\hspace{-1pt} h h' \rb_0^\pprime}{\big[(1 \hspace{-1pt}+\hspace{-1pt} \rb_0 h)^2 \hspace{-1pt}+\hspace{-1pt} h^{\prime\hspace{1pt}2}\big]^{3/2}}
&= \frac{(1 + \rb_0 h) h'' - \rb_0(1 + 2 \rb_0 h + \rb_0^2 h^2 + 2 h^{\prime\hspace{1pt}2}) - h h' \rb_0^\pprime}{\rg^{3/2}} \,. \label{defn:H} %\\
%&= \frac{1}{1 + \rb_0 h} \big[\rg^{-1/2} h'\big]^\pprime - \rb_0 \rg^{-1/2}
%\quad \text{on}\quad \Gamma \,. \label{defn:H1}
\end{align}
%\begin{align*}
%\big[\rg^{-1/2} h'\big]^\pprime &= - \rg^{-3/2} \big[(1 \hspace{-1pt}+\hspace{-1pt} \rb_0 h) (\rb_0^\pprime h + \rb_0 h') + h' h''\big] h' + \rg^{-1/2} h'' \\
%&= g^{-3/2} \big[- (1 \hspace{-1pt}+\hspace{-1pt} \rb_0 h) (\rb_0^\pprime h h' + \rb_0 h^{\prime\hspace{1pt}2}) - h^{\prime\hspace{1pt}2} h'' + (1 \hspace{-1pt}+\hspace{-1pt} \rb_0 h)^2 h'' + h^{\prime\hspace{1pt}2} h^{\prime\pprime} \big] \\
%&= g^{-3/2} \big[- \rb_0^\pprime h h' - \rb_0 h^{\prime\hspace{1pt}2} - \rb_0 b^\pprime_0 h^2 h' - \rb_0^2 h h^{\prime\hspace{1pt}2} + (1 + \rb_0 h)^2 h'' \big]\,;
%\end{align*}
%thus
%\begin{align*}
%\rH \circ \psi &= \frac{1}{1 + \rb_0 h}\big[\rg^{-1/2} h'\big]^\pprime + \frac{\rb_0^\pprime h h' + \rb_0 h^{\prime\hspace{1pt}2} + \rb_0 b^\pprime_0 h^2 h' + \rb_0^2 h h^{\prime\hspace{1pt}2} }{(1 + \rb_0 h) \rg^{3/2}} \\
%&\quad - \frac{\big[\rb_0(1 + 2 \rb_0 h + \rb_0^2 h^2 + 2 h^{\prime\hspace{1pt}2}) + h h' \rb_0^\pprime\big](1 + \rb_0 h)}{(1 + \rb_0 h)\rg^{3/2}} \\
%&= \frac{1}{1 + \rb_0 h} \big[\rg^{-1/2} h'\big]^\pprime - \frac{\rb_0 (1 \hspace{-1pt}+\hspace{-1pt} \rb_0h)^2}{\rg^{3/2}}
%\end{align*}

\begin{remark}
In the methodology we employ, we need $\rH\circ\psi \in H^{0.5}(\Gamma)$. %{\rm(}see Section {\rm\ref{sec:good_ALE}} for some details{\rm)}.
If $\bdy \Omega$ is an $H^{3.5}$-surface, then we can simply let $\rO = \Omega$ to proceed. However, since we will only assume that $\bdy\Omega$ is an $H^2$-surface, the use of $\rO = \Omega$ will result in that $\rb_0 \in L^2(\Gamma)$ which implies that the curvature $\rH \circ \psi$ at best belongs to $H^{-1}(\Gamma)$ due to the presence of $\rb_0^\pprime$ in {\rm(\ref{defn:H})}. This is the reason why we choose a smooth $\rO$ to start with.
\end{remark}

\subsection{Some basic identities concerning the map $\psi$}
Let $\rJ = \det(\nabla \psi)$, and $\rA = (\nabla \psi)^{-1}$. Writing $x = \psi(y)$, then the divergence theorem and the Piola identity suggest that
$$
\int_{\bdy\psi(\rO)} w\cdot n dS_x = \int_{\psi(\rO)} \div w dx = \int_\rO \rJ \rA^j_i (w^i\circ\psi)_{,j} dy = \int_{\bdy\rO} \hspace{-2pt}\rJ \rA^j_i (w^i\circ\psi) \rN_j dS_y\,.
$$
Since $\psi: \bdy \rO \to \bdy\psi(\rO)$ is also a diffeomorphism, the change of variable formula implies that
$$
\int_{\bdy\rO} (w^i \circ\psi) (n^i\circ\psi) \sqrt{\rg} dS_y = \int_{\bdy\psi(\rO)} w\cdot n dS_x = \int_{\bdy\rO} \hspace{-2pt}\rJ \rA^j_i (w^i\circ\psi) \rN_j dS_y\,.
$$
The identity above holds for all smooth $w$; thus we obtain that
\begin{equation}\label{JAtN}
\rJ \rA^{\hspace{-1pt}\rT} \rN = \sqrt{\rg} (n\circ \psi) \ \quad\text{on}\quad \Gamma.
\end{equation}
In other words, the direction of the exterior normal $n$ is parallel to the vector $\rA^{\hspace{-1pt}\rT} \rN$, and the length of $\rJ \rA^{\hspace{-1pt}\rT} \rN$ is $\sqrt{\rg}$, the square root of the metric.

\subsection{The equations in ALE coordinate}
Let $v=u\circ\psi$ and $q=p\circ\psi$ be the velocity and pressure in ALE coordinate. Taking the composition of (\ref{NS}) and the map $\psi$, we find that the equation (\ref{NS}) is transformed to
\begin{subequations}\label{NSALE}
\begin{alignat}{2}
v^i_t + \rA^k_\ell (v^\ell - \psi^\ell_t) v^i_{,k} + \rA^k_i q_{,k} &= \rA^k_\ell \big(\rA^j_\ell v^i_{,j} + \rA^j_i v^\ell_{,j}\big)_{,k} \qquad &&\text{in}\quad\rO\times(0,\rT)\,,\\
\rA^j_i v^i_{,j} &= 0 &&\text{in}\quad\rO\times (0,\rT)\,,\\
\big[\rA^j_\ell v^i_{,j} + \rA^j_i v^\ell_{,j} - q \delta^\ell_i\big] \rA^k_\ell \rN_k &= \sigma (\rH \circ \psi) \rA^k_i \rN_k &&\text{on}\quad \Gamma\times (0,\rT)\,,\\
\Delta \psi &= 0 &&\text{in}\quad\rO\times(0,\rT)\,,\\
\psi &= (e + h \rN) &&\text{on}\quad \Gamma\times (0,\rT)\,,\\
\psi_t \cdot (n\circ \psi) &= (u\cdot n)\circ\psi &&\text{on}\quad \Gamma\times (0,\rT)\,,\\
v &= v_0 \equiv u_0\circ \psi_0 &&\text{on}\quad \rO\times \{t=0\},\\
h &= h_0 &&\text{on}\quad \Gamma\times \{t=0\},
\end{alignat}
\end{subequations}
where $\rJ$, $\rA$ are defined in previous sub-section, and $\psi_0 = \psi(0)$. The boundary condition (\ref{NSALE}c) is obtained by taking the composition of (\ref{NS}c) and the map $\psi$, then applying identity (\ref{JAtN}). Furthermore, the curvature $\rH \circ\psi$ in boundary condition (\ref{NSALE}c) will be represented by (\ref{defn:H}).

We also remark that boundary condition (\ref{NSALE}f) reads that the speed of $\bdy\Omega(t)$ in the direction of exterior normal is the same as $u\cdot n$. In other words, the boundary of $\Omega(t)$ moves along with the fluid velocity.

\subsection{The evolution equation of $h$}
The evolution equation of $h$ is a direct consequence of the boundary condition (\ref{NSALE}e,f). Differentiating (\ref{NSALE}e) in time, then taking the inner product of the resulting equation and $n\circ\psi$, by (\ref{NSALE}f) we find that
$$
h_t (\rN \cdot (n\circ\psi)\big) = (u\cdot n)\circ \psi \qquad\text{on}\quad \Gamma\times (0,\rT)\,.
$$
By (\ref{defn:n}) and (\ref{JAtN}), the equation above reads
\begin{equation}
h_t = \frac{v\cdot (n\circ\psi)}{\rN \cdot (n\circ\psi)} %= \frac{v^i \rA^k_i \rN_k}{\rN_i \rA^k_i \rN_k}
= \frac{\rJ v^i \rA^k_i \rN_k}{\sqrt{\rg}\, \rN \cdot (n\circ\psi)} = \frac{\rJ \rA^{\hspace{-1pt}\rT} \rN}{1 + \rb_0 h} \cdot v \qquad\text{on}\quad \Gamma\times (0,\rT)\,. \label{h_eq}
\end{equation}
We remark here that the denominator does not vanish for a short period of time if $|h|\ll 1$. Equation (\ref{h_eq}) is the evolution equation of $h$.

\subsection{A modification of the ALE formulation}\label{sec:new_formulation}
For the purpose of constructing solutions, we modify the ALE formulation such that the new formulation keeps the structure of the divergence-free ``velocity'' field and the ``pure'' pressure gradient.

Let $w^i = \rJ \rA^i_j v^j$ or equivalently $v^i = \rJ^{-1} \psi^i_{,r} w^r$. Then (\ref{NSALE}b) together with the Piola identity implies that $w$ is divergence-free. Moreover, since
\begin{align*}
& \psi^i_{,s} \rA^k_\ell \big(\rA^j_\ell v^i_{,j} +\rA^j_i v^\ell_{,j} \big)_{,k} \hspace{-1pt}= \psi^i_{,s} \rA^k_\ell \big[\rA^j_\ell (\rJ^{-1} \psi^i_{,r} w^r)_{,j} +\rA^j_i (\rJ^{-1} \psi^\ell_{,r} w^r)_{,j} \big]_{,k} \\
&\qquad = \big[\psi^i_{,s} \rA^k_\ell \rA^j_\ell (\rJ^{-1} \psi^i_{,r} w^r)_{,j} + \rA^k_\ell(\rJ^{-} \psi^\ell_{,r} w^r)_{,s})\big]_{,k} \\
&\qquad\quad - (\psi^i_{,s} \rA^k_\ell)_{,k} \big[\rA^j_\ell (\rJ^{-1} \psi^i_{,r} w^r)_{,j} + \rA^j_i (\rJ^{-} \psi^\ell_{,r} w^r)_{,j})\big]
\end{align*}
and
\begin{align*}
&\psi^i_{,s} \big[\rA^j_\ell v^i_{,j} + \rA^j_i v^\ell_{,j} - q \delta^\ell_i\big] \rA^k_\ell \rN_k \\
&\qquad\quad = \big[\psi^i_{,s} \rA^k_\ell \rA^j_\ell (\rJ^{-1} \psi^i_{,r} w^r)_{,j} + \rA^k_\ell (\rJ^{-1} \psi^\ell_{,r} w^r)_{,s} - q \delta^k_s \big] \rN_k\,,
\end{align*}
if $\rL_\psi$ denotes the second order differential operator given by
\begin{equation}\label{defn:Lpsi}
\displaystyle{} \big[\rL_\psi(w)\big]^s = \big[\psi^i_{,s} \rA^k_\ell \rA^j_\ell (\rJ^{-1} \psi^i_{,r} w^r)_{,j} + \rA^k_\ell (\rJ^{-1} \psi^\ell_{,r} w^r)_{,s}\big]_{,k}
\end{equation}
and $\ell_\psi(w,q)$ denotes the boundary operator given by
\begin{equation}\label{defn:ellpsi}
\begin{array}{l}
\displaystyle{} \big[\ell_\psi(w,q)\big]^s = \big[\psi^i_{,s} \rA^k_\ell \rA^j_\ell (\rJ^{-1} \psi^i_{,r} w^r)_{,j} + \rA^k_\ell (\rJ^{-1} \psi^\ell_{,r} w^r)_{,s} - q \delta^k_s \big] \rN_k \,,
\end{array}
\end{equation}
we find that $(w,q)$ satisfies the following equation
\begin{subequations}\label{NSnew}
\begin{alignat}{2}
\rJ^{-1} \psi^i_{,s} \psi^i_{,r} w^r_t - \big[\rL_\psi(w)\big]^s + q_{,s} &= F^s \qquad&&\text{in}\quad \rO\times (0,T)\,,\\
\div w &= 0\qquad &&\text{in}\quad \rO\times (0,T)\,,\\
\ell_\psi(w,q) &= \sigma (\rH\circ\psi) \rN \qquad&&\text{on}\quad \Gamma\times (0,T)\,,\\
\Delta \psi &= 0 &&\text{in}\quad \rO\times (0,T)\,,\\
\psi &= e + h \rN &&\text{on}\quad \Gamma\times (0,T)\,,\\
h_t &= \frac{w\cdot \rN}{1 \hspace{-1pt}+\hspace{-1pt} \rb_0 h} \qquad&&\text{on}\quad \Gamma\times (0,T)\,,\\
w &= w_0 &&\text{on}\quad \rO\times \{t=0\}\,,\\
h &= h_0 &&\text{on}\quad \Gamma\times \{t=0\}\,,
\end{alignat}
\end{subequations}
where
\begin{equation}\label{defn:F}
\begin{array}{l}
\displaystyle{} F^s = - \,\psi^i_{,s} \rA^j_\ell (\rJ^{-1} \psi^\ell_{,r} w^r - \psi^\ell_t) (\rJ^{-1} \psi^i_{,s} w^s)_{,j} - \psi^i_{,s} (\rJ^{-1} \psi^i_{,r})_t w^r \vspace{.1cm}\\
\displaystyle{}\hspace{24pt} + (\psi^i_{,s} \rA^k_\ell)_{,k} \big[\rA^j_\ell (\rJ^{-1} \psi^i_{,r} w^r)_{,j} + \rA^j_i (\rJ^{-} \psi^\ell_{,r} w^r)_{,j})\big]\,.
\end{array}
\end{equation}
Letting $\rF = F - \big[\rJ^{-1} (\nabla \psi)^\rT (\nabla \psi) - \id\big] w_t$, (\ref{NSnew}a) can be rewritten as
$$
w_t - \rL_\psi(w) + \nabla q = \rF \qquad\text{in}\quad \rO\times (0,T)\,. \eqno{\rm(\ref{NSnew}a')}
$$

%Before proceeding to the next section, we remark that since $\psi^j_{,i} \approx \delta^j_i$, $\rA^j_i \approx \delta^j_i$ and $\rJ \approx 1$ under the assumption that $\|h_0\|_{H^{1.7}(\Gamma)} \ll 1$, the differential operator $\rL_\psi$ is approximately the differential operator $\div \Def$, and the boundary operator $\ell_\psi(w,q)$ is approximately the normal traction of $(w,q)$. In other words,
%\begin{equation}\label{approximation}
%\rL_\psi(w) \approx \div \Def w \qquad\text{and}\qquad \ell_\psi(w,q) \approx (\Def w - q\id) \rN\,.
%\end{equation}

\section{Notation and preliminary results}\label{sec:notation}
\subsection{The energy spaces $\V(\rT)$, $\H(\rT)$, $\W(\rT)$, $\H_1(\rT)$}
Let $\V(\rT)$ denote the space (of solutions $v$)
$$
\V(\rT) \equiv \Big\{v \in L^2(0,\rT;H^2(\rO))\, \Big|\, v_t \in L^2(0,\rT;L^2(\rO))\Big\}
$$
equipped with norm
$$
\|v\|_{\V(\rT)} = \|v\|_{L^2(0,\rT;H^2(\rO))} + \|v_t\|_{L^2(0,\rT;L^2(\rO))}\,,
$$
$\Q(\rT)$ denote the space (of solutions $q$) $L^2(0,\rT;H^1(\rO))$, and $\H(\rT)$ denote the space (of solutions $h$)
$$
\H(\rT) \equiv \Big\{ h\in L^2(0,\rT;H^{2.5}(\Gamma))\,\Big|\, h_t \in L^2(0,\rT;H^{1.4}(\Gamma)) %\, h_{tt} \in L^2(0,\rT;H^{-0.5}(\Gamma))
\Big\}
$$
(in which the number $1.4$ can be replaced by any number closed to but less than $1.5$) equipped with norm
$$
\|h\|_{\H(\rT)} \equiv \|h\|_{L^2(0,\rT;H^{2.5}(\Gamma))} + \|h_t\|_{L^2(0,\rT;H^{1.4}(\Gamma))} \,.
$$
We also define two spaces $\W(\rT)$ and $\H_1(\rT)$ for the purpose of constructing approximated solutions. The space $\W(\rT)$ is the collection of all $w\in \V(\rT)$ such that $w\in L^2(0,\rT;H^2(\Gamma))$; that is
$$
\W(\rT) = \Big\{ w \in \V(\rT) \,\Big|\, w \in L^2(0,\rT;H^2(\Gamma))\Big\},
$$
and the norm $\|\cdot\|_{\W(\rT)}$ is given by
$$
\|w\|_{\W(\rT)} = \|w\|_{\V(\rT)} + \|w\|_{L^2(0,\rT;H^2(\Gamma))}\,.
$$
The space $\H_1(\rT)$, on the other hand, is not a subspace of $\H(\rT)$. It is given by
$$
\H_1(\rT) \equiv \Big\{ h \in L^2(0,\rT;H^2(\Gamma))\,\Big|\, h_t\in L^2(0,\rT;H^2(\Gamma)) \,,\ h_{tt}\in L^2(0,\rT;H^{-0.5}(\Gamma))
\Big\}
$$
equipped with norm
$$
\|h\|_{\H_1(\rT)} \equiv \|h\|_{L^2(0,\rT;H^2(\Gamma))} + \|h_t\|_{L^2(0,\rT;H^2(\Gamma))} + \|h_{tt}\|_{L^2(0,\rT;H^{-0.5}(\Gamma))}
\,.
$$
By the fundamental theorem of Calculus,
\begin{subequations}\label{sup_in_time_ineq}
\begin{align}
\sup_{t\in [0,\rT]} \|v(t)\|_{H^1(\rO)} &\le \|v(0)\|_{H^1(\rO)} + C \|v\|_{\V(\rT)} \,, \label{v_H1_est}\\
%\end{equation}
%and
%\begin{equation}
%\sup_{t\in [0,\rT]} \|h(t)\|_{H^2(\Gamma)} &\le \|h(0)\|_{H^2(\Gamma)} + \|h\|_{\H(\rT)} \,, \label{h_H2_est}\\
\sup_{t\in [0,\rT]} \|h(t)\|_{H^2(\Gamma)} &\le \|h(0)\|_{H^2(\Gamma)} + \|h\|_{\H_1(\rT)} \,. \label{h_H2_est1} %\\
%\sup_{t\in [0,\rT]} \|h_t(t)\|_{H^{0.75}(\Gamma)} &\le C \Big[\|h_t(0)\|_{H^{0.75}(\Gamma)} +
\end{align}
\end{subequations}

\subsection{A useful lemma}
%In the process of doing the energy estimates, we need the following
By interpolation, we can derive the following useful
\begin{lemma}
Suppose that $f \in H^s(\Gamma)$ for some $s>1/2$, and $g \in H^{0.5}(\Gamma)$. Then $fg \in H^{0.5}(\Gamma)$ and
\begin{equation}\label{duality_ineq}
\|fg\|_{H^{0.5}(\Gamma)} \le C_s \|f\|_{H^s(\Gamma)} \|g\|_{H^{0.5}(\Gamma)}
\end{equation}
for some generic constant $C_s>0$.
\end{lemma}

\subsection{The horizontal convolution-by-layers operator and a commutator type estimate}\label{sec:horizontal_convolution}
Let $\eta:\bbR \to \bbR$ be a non-negative smooth function supported in the interval $\big(\smallexp{$\displaystyle{}-\hspace{-1pt}\frac{\ell}{4}$},\smallexp{$\displaystyle{}\frac{\ell}{4}$}\big)$ and satisfying
$$
\int_{-\ell/4}^{\ell/4} \eta(s)\, ds = 1.
$$
Let $\eta_\eps(s) = \smallexp{$\displaystyle{}\frac{1}{\eps} \eta\big(\frac{s}{\eps}\big)$}$ for all $\eps>0$. Given a function $f$ defined on $\Gamma$, we define the horizontal convolution $\eta_\eps \star f$ by
\begin{align*}
f_\eps(x) &= (\eta_\eps\star f)(x) \equiv \big(\eta_\eps \star (f\circ\rX)\big)(\rX^{-1}(x)) \\
&= \int_{-\ell/2}^{\ell/2} \eta_\eps(\rX^{-1}(x)-\tilde{s}) f(\rX(\tilde{s}))\, d\tilde{s} \qquad\Forall x\in \Gamma\,.
\end{align*}
Since every point $x$ can be identified as $\rX(s)$ for a unique $s\in [-\ell/2,\ell/2)$, we also write the equation above as
$$
f_\eps(s) = (\eta_\eps \star f)(s) = \int_{-\ell/2}^{\ell/2} \eta_\eps(s-\tilde{s}) f(\tilde{s})\, d\tilde{s} \qquad\text{if \ $x = \rX(s)$}\,.
$$
We note that in the equation above, $f(\tilde{s})$ is understood as $f(\rX(\tilde{s}))$. Moreover, given two parameters $\eps$ and $\ve$, we have
\begin{align}
\big[\eta_\eps \star (\eta_\ve \star f)\big] (s) &= \int_{-\ell/2}^{\ell/2} \frac{1}{\eps}\, \eta\big(\frac{s'}{\eps}\big) \int_{-\ell/2}^{\ell/2} \frac{1}{\ve}\, \eta\big(\frac{s''}{\ve}\big) f(s''+s'-s) ds''
ds' \nonumber\\
&= \big[\eta_\ve \star (\eta_\eps \star f)\big] (s)\,; \label{commute_convol}
\end{align}
thus two horizontal convolution commute.

Having the convolution on $\Gamma$ defined, we define the commutator of the horizontal convolution operator $\eta_\eps\star $ and a function $f$ by
$$
\big(\comm{\eta_\eps\star }{f} g\big)(s) = \big(\eta_\eps\star (fg) \hspace{-1pt}-\hspace{-1pt} f (\eta_\eps\star g)\big)(s) = \int_{-\ell/2}^{\ell/2} \hspace{-2pt} \eta_\eps(s-\tilde{s}) \big[f(\tilde{s}) \hspace{-1pt}-\hspace{-1pt} f(s)\big] g(\tilde{s})\, d\tilde{s}\,.
$$
Then the mean value theorem and Young's inequality imply that
\begin{equation}\label{comm_est_temp1}
\big\|\comm{\eta_\eps\star }{f} g\big\|_{L^2(\Gamma)} \le \eps \|f^\pprime\|_{L^\infty(\Gamma)} \|g\|_{L^2(\Gamma)}\,.
\end{equation}
Moreover,
\begin{align*}
& \big(\comm{\eta_\eps\star }{f} g\big)'(s) \\
&\qquad = \int_{-\ell/2}^{\ell/2} \frac{1}{\eps}\,\eta'_\eps(s - s^\pprime) \big[f(s^\pprime) - f(s)\big] g(s^\pprime)\, ds^\pprime - \int_{-\ell/2}^{\ell/2} \eta_\eps(x-y) f^\pprime(s) g(s^\pprime)\, ds^\pprime;
\end{align*}
thus
\begin{align}
\big\|\big(\comm{\eta_\eps\star }{f} g \big)'\big\|_{L^2(\Gamma)} &\le \|\eta^\pprime\|_{L^1(\Gamma)} \|f^\pprime\|_{L^\infty(\Gamma)} \|g\|_{L^2(\Gamma)} + \|f^\pprime\|_{L^\infty(\Gamma)} \|g\|_{L^2(\Gamma)} \nonumber\\
&\le C \|f^\pprime\|_{L^\infty(\Gamma)} \|g\|_{L^2(\Gamma)}\,. \label{comm_est_temp2}
\end{align}

For any given $f$ defined on $\bbR^2_+$, we can also define the {\it horizontal convolution-by-layers} operator $\Lambda_\eps$ by
\begin{equation*}
\Lambda_\eps f (y_1, y_2) = \int_\bbR \eta_\eps (y_1 - z_1) f(z_1, y_1)\, dz_1 \qquad\Forall f ( \cdot , y_2) \in L^1(\bbR) \,.
\end{equation*}
It should be clear that $\Lambda_\eps$ smooths functions defined on $\bbR^2$ along horizontal $y_1$-direction, but does not smooth these functions in the vertical $y_2$-direction. In addition, we can restrict the operator $ \Lambda_\eps$ to act on functions $f: \bbR^{\n-1} \to \bbR $ as well, in which case $ \Lambda_\eps$ becomes the usual mollification operator $\eta_\eps\star$.

By standard properties of convolution, there exists a constant $C$ which is independent of $\eps$, such that for
$s \ge 0$,
\begin{equation*}
\|\Lambda_\eps F\|_{H^s(\bbR^2_+)} \le C \|F\|_{H^s(\bbR^2_+)} \qquad \Forall F \in H^s(\bbR^2_+) \,,
\end{equation*}
and
\begin{equation*}
\|\Lambda_\epsilon F\|_{H^s(\bdy \bbR^2_+)} \le C \|F\|_{H^s(\bdy\bbR^2_+)} \qquad \Forall F \in H^s(\bdy\bbR^2_+) \,.
\end{equation*}
Furthermore,
\begin{equation}\label{F1}
\epsilon \|(\Lambda_\epsilon F)_{,1}\|_{L^2(\bbR^2_+)}
\le C \|F\|_{L^2(\bbR^2_+)} \qquad \Forall F \in L^2(\bbR^2_+) \,.
\end{equation}
Similar to (\ref{comm_est_temp1}) and (\ref{comm_est_temp2}), %with $\bp$ denoting the tangential gradient,
we also have
\begin{subequations}\label{comm_est}
\begin{align}
\big\|\comm{\Lambda_\eps}{f} g\big\|_{L^2(\bbR^2_+)} &\le C \eps \|f_{,1}\|_{L^\infty(\bbR^2_+)} \|g\|_{L^2(\bbR^2_+)}\,,\\
\big\|\big(\comm{\Lambda_\eps}{f}g\big)_{,1}\big\|_{L^2(\bbR^2_+)} &\le C \|f_{,1}\|_{L^\infty(\bbR^2_+)} \|g\|_{L^2(\bbR^2_+)}\,.
\end{align}
\end{subequations}

%Interpolating between (\ref{comm_est_temp1}) and (\ref{comm_est_temp2}), by \Poincare's inequality we obtain that for $s\in [0,1]$,
%$$
%\Big\|\big[\comm{\eta_\eps\star }{f} g - \frac{1}{L} \int_{-\ell/2}^{\ell/2} \hspace{-2pt}\big[\comm{\eta_\eps\star }{f} g\, ds \Big\|_{H^s(\Gamma)} \le C \eps^{1-s} \|f^\pprime\|_{L^\infty(\Gamma)} \|g\|_{L^2(\Gamma)}
%$$
%or
%\begin{equation}
%\big\|\big[\comm{\eta_\eps\star }{f} g\big\|_{H^s(\Gamma)} \le C \|f\|_{L^2(\Gamma)} \|g\|_{L^2(\Gamma)} + C \eps^{1-s} \|f^\pprime\|_{L^\infty(\Gamma)} \|g\|_{L^2(\Gamma)} \,. \label{comm_est}
%\end{equation}

\subsection{The generalized Gronwall inequality}
In the process of performing the nonlinear estimates, we need the following Gronwall type inequality.
\begin{theorem}\label{thm:Gronwall}
Let $X$ be a non-negative continuous function of $t$, and satisfy that for some positive constants $C$, $M$, $T_1$ and polynomial $\P$,
$$
X(t) \le M + C t \P\big(X(t)\big) \qquad\Forall t\in [0,T_1]\,.
$$
Then there is a $T \in (0,T_1]$ such that $X(t) \le 2 M$ for all $t\in [0,T]$.
\end{theorem}

%\subsection{The normal trace estimates}
%By the divergence theorem,
%$$
%\int_\Gamma \bJ \bA^k_i \rN_k w^i \varphi\, dS = \int_\rO \bJ \bA^k_i w^i_{,k} \varphi dx + \int_\rO \bJ \bA^k_i w^i \varphi_{,k} dx \qquad\Forall w, \varphi \in H^1(\rO)\,.
%$$
%Taking the supreme of the left-hand side over all $\varphi \in H^{0.5}(\Gamma)$ with $\|\varphi\|_{H^{0.5}(\Gamma)}$, by (\ref{J-1_est1}), (\ref{hH175_bound}) and the density argument we conclude that
%\begin{align}
%\|\bJ \bA^k_i \rN_k w^i\|_{H^{-0.5}(\Gamma)} &\le C \big[\|w\|_{L^2(\rO)} + \|\bJ \bA^k_i w^i_{,k}\|_{L^2(\rO)} \big] \nonumber\\
%&\le C \big[\|w\|_{L^2(\rO)} + \|\smallexp{$\displaystyle{}\sqrt{\bJ}$} \bA^k_i w^i_{,k}\|_{L^2(\rO)}\Big]\,. \label{normal_trace}
%\end{align}

\subsection{The Lagrangian multiplier lemma}
Let
$\displaystyle{}
\rV \equiv \big\{u\in H^1(\rO)\,\big|\, u \in H^1(\Gamma) %\text{ \ and } \smallexp{$\displaystyle{}\int_\Gamma$} u\cdot \rN\, dS = 0
\big\}
$
equipped with norm
\begin{equation}
\|u\|_\rV = \Big[\|u\|^2_{L^2(\rO)} + \|\Def u\|^2_{L^2(\rO)} + \|u'\|^2_{L^2(\Gamma)}\Big]^{1/2} \label{defn:norm_of_U}
\end{equation}
which is induced by the inner product
$$
(u,v)_\rV = (u,v)_{L^2(\rO)} + (\Def u,\Def v)_{L^2(\rO)} + (u',v^\pprime)_{L^2(\Gamma)}\,.
$$
We note that by Korn's inequality, the norm defined by {\rm(\ref{defn:norm_of_U})} is equivalent to the norm
$$
\triplenorm{u} \equiv \|u\|_{H^1(\rO)} + \|u\|_{H^1(\Gamma)}\,.
$$
\begin{lemma}\label{lem:Lagrange_multiplier}
Let $T: \rV \to \bbR$ be a bounded linear functional satisfying that $T(\varphi) = 0$ whenever $\div \varphi = 0$. Then there exists a unique $q \in L^2(\rO)$ such that
$$
T(\varphi) = \big(q,\div \varphi\big)_{L^2(\rO)} \qquad \Forall \varphi \in \rV\,.
$$
Moreover, for some constant $c>0$,
\begin{equation}\label{Lagrange_multiplier_pressure_estimate}
\frac{1}{c}\hspace{.5pt}\|q\|_{L^2(\rO)} \le \|T\|_{\B(\rV,\bbR)} \equiv \sup_{\|\varphi\|_\rV = 1} T(\varphi)\,.
\end{equation}
\end{lemma}
\begin{proof}
For any given $p\in L^2(\rO)$, define $L_p(\varphi) = \big(p,\div \varphi\big)_{L^2(\rO)}$. Then $L_p: \rV \to \bbR$ is a bounded linear functional. By the Riesz representation theorem, there exists $Q p \in \rV$ such that
$$
L_p(\varphi) = (Q p,\varphi)_\rV \qquad\Forall \varphi\in \rV
$$
and $\|Q p\|_\rV = \|L_p\|_{\B(\rV,\bbR)} \le c_1 \|p\|_{L^2(\rO)}$. On the other hand, for any $p\in L^2(\rO)$, there exists $u_p\in \rV$ such that
\begin{equation}\label{div_prob}
\div u_p = p \qquad\text{in}\quad \rO
\end{equation}
and satisfies $\|u_p\|_\rV \le c_2 \|p\|_{L^2(\rO)}$. In fact, with $\bar{p}$ denoting the average of $p$ over $\rO$; that is, $\bar{p} \equiv \smallexp{$\displaystyle{}\frac{1}{|\rO|}\int_\rO$} p\, dx$, if $v$ solves
\begin{alignat*}{2}
\div v &= p - \bar{p} \qquad&&\text{in}\quad\rO\,,\\
v &= 0 &&\text{on}\quad\Gamma\,,
\end{alignat*}
and satisfies $\|v\|_{H^1(\rO)} \le C \|p - \bar{p}\|_{L^2(\rO)}$, then $u_p(x,y) \equiv v(x,y) + \smallexp{$\displaystyle{}\frac{\bar{p}}{2}$}(x,y)$ belongs to $\rV$ and satisfies (\ref{div_prob}) and
$$
\|u_p\|_{H^1(\rO)} \le \|v\|_{H^1(\rO)} + C |\bar{p}| \le C \|p\|_{L^2(\rO)}\,.
$$
Therefore,
$$
\|p\|^2_{L^2(\rO)} = (p, \div u_p)_{L^2(\rO)} = (Q p,u_p)_\rV \le \|Q p\|_\rV \|u_p\|_\rV \le c_2 \|Qp\|_\rV \|p\|_{L^2(\rO)}
$$
which implies that $\|p\|_{L^2(\rO)} \le c_2 \|Q p\|_\rV$ for all $p\in L^2(\rO)$. As a consequence, %there exists positive constants $c_1$ and $c_2$ such that
$$
\frac{1}{c_2} \|p\|_{L^2(\rO)} \le \|Q p\|_\rV \le c_1 \|p\|_{L^2(\rO)} \qquad \Forall p\in L^2(\rO)\,,
$$
so $Q:L^2(\rO) \to \rV$ is one-to-one, and $\rR(Q)$, the range of $Q$, is closed. Therefore, $\rV = \rR(Q)\oplus_\rV \rR(Q)^\perp$. Let $\rP:\rV \to \rR(Q)$ be the projection. Then for any $u\in \rV$,
$$u = \rP u + (\id - \rP) u\,,$$
where $(\id - \rP)u \in \rR(Q)^\perp$. Moreover,
\begin{align*}
v\in \rR(Q)^\perp &\Leftrightarrow (Q q, v)_\rV = 0 \ \ \Forall q\in L^2(\rO) \Leftrightarrow (q, \div v)_{L^2(\rO)} = 0 \ \ \Forall q\in L^2(\rO) \\
&\Leftrightarrow \div v = 0\,.
\end{align*}
Let $w\in \rV$ be the representation of $T$; that is, $T(\varphi) = (w,\varphi)_\rV$ for all $\varphi\in \rV$. Then $\rP w = Q q$ for some $q\in L^2(\rO)$, and for any given $\varphi \in \rV$,
\begin{align*}
T(\varphi) &= T(\rP \varphi) = (w, \rP \varphi)_\rV = (\rP w, \rP \varphi)_\rV = (Q q, \rP \varphi)_\rV \\
& = (Q q, \varphi)_\rV = L_q(\varphi) = (q, \div \varphi)_{L^2(\rO)}\,.
\end{align*}
Finally, by the solvability of (\ref{div_prob}),
$$
\|q\|_{L^2(\rO)} = \sup_{\|p\|_{L^2(\rO)} = 1} (q,p)_{L^2(\rO)} \le \sup_{\|\varphi\|_\rV \le c_2} (q, \div \varphi)_{L^2(\rO)} = \sup_{\|\varphi\|_\rV \le c_2} T(\varphi)\,,
$$
and (\ref{Lagrange_multiplier_pressure_estimate}) follows immediately.
\end{proof}

%\subsection{The Stokes regularity}
%We need the following regularity theory.
%\begin{theorem}\label{theorem:elliptic_regularity}
%Suppose that $(v,q)\in H^2(\rO) \times H^1(\rO)$ is a solution to the following elliptic equation
%\begin{alignat*}{2}
%- \Delta v + \nabla q &= f &&\text{in}\quad\rO\,,\\
%\div w &= g &&\text{in}\quad\rO\,,\\
%w &= h \qquad&&\text{on}\quad\Gamma\,;
%\end{alignat*}
%for some $f\in L^2(\rO)$ and $g \in H^{1.5}(\Gamma)$ with $\smallexp{$\displaystyle{}\int_\Gamma$}\,h\cdot \rN\, dS = \smallexp{$\displaystyle{}\int_\rO$} g\, dx$. Then $(v,q) \in H^2(\rO) \times H^1(\rO)$, and satisfies
%\begin{equation}\label{elliptic_est3}
%\|v\|^2_{H^2(\rO)} + \|q\|^2_{H^1(\rO)} \le C \Big[\|f\|^2_{L^2(\rO)} + \|g\|^2_{H^1(\rO)} + \|h\|^2_{H^{1.5}(\Gamma)} \Big]
%\end{equation}
%for some generic constant $C$ depending only on $\rO$.
%\end{theorem}

\subsection{Elliptic regularity}
In this sub-section we study an important regularity theory which is the foundation of our result and is very similar to the regularity of the steady Stokes problem. The proof is provided in Appendix \ref{app:Stokes}.
\begin{theorem}\label{thm:elliptic_regularity}
Let $a^{jk}_{rs}$ be a $(2,2)$-tensor such that $a^{jk}_{rs} = a^{kj}_{rs} = a^{jk}_{sr}$, and satisfy
\begin{equation}\label{smallness}
\|a^{jk}_{rs} - \lambda_1 \delta^j_k \delta^r_s - \lambda_2 \delta^k_r \delta^j_s\|_{L^\infty(\rO)} \ll 1
\end{equation}
for some positive constants $\lambda_1$ and $\lambda_2$.
\begin{enumerate}
\item
    Suppose that $(w,q)\in \rV \times L^2(\rO)$ is a weak solution to the following elliptic equation
    \begin{subequations}\label{elliptic}
    \begin{alignat}{2}
    - \big[a^{jk}_{rs} w^r_{,j}\big]_{,k} + q_{,s} &= f^s &&\text{in}\quad\rO\,,\\
    \div w &= 0 &&\text{in}\quad\rO\,,\\
    a^{jk}_{rs} w^r_{,j} \rN_k - q\rN_s &= \ve \Delta_0 w^s + g^s \qquad&&\text{on}\quad\Gamma\,;
    \end{alignat}
    \end{subequations}
    that is, $(w,q)$ satisfies the variational formulation
    \begin{equation}\label{elliptic_weak}
    \begin{array}{l}
    \displaystyle{} \big(a^{jk}_{rs} w^r_{,j}, \varphi^s_{,k})_{L^2(\rO)} - (q,\div \varphi)_{L^2(\rO)} + \ve (w^\pprime, \varphi^\pprime)_{L^2(\Gamma)} \vspace{.2cm}\\
    \displaystyle{} \hspace{60pt} = (f, \varphi)_{L^2(\rO)} + (g,\varphi)_{L^2(\Gamma)} \qquad \Forall \varphi \in \rV\,.
    \end{array}
    \end{equation}
    Then $(w,q) \in H^2(\rO) \times H^1(\rO)$, and there are constants $C$ and $C_\ve$ such that
    \begin{equation}\label{elliptic_est}
    \begin{array}{l}
    \displaystyle{} \hspace{35pt}\|w\|^2_{H^2(\rO)} \hspace{-1pt}+\hspace{-1pt} \ve \|w\|^2_{H^2(\Gamma)} \hspace{-1pt}+\hspace{-1pt} \|q\|^2_{H^1(\rO)} \hspace{-1pt}\le\hspace{-1pt} C_\ve \|g\|^2_{L^2(\Gamma)} \hspace{-1pt}+\hspace{-1pt} C \big(1 \hspace{-1pt}+\hspace{-1pt} \|a\|^2_{L^\infty(\rO)}\big) \times \vspace{.2cm}\\
    \displaystyle{} \hspace{60pt} \times \Big[\big(1 \hspace{-1pt}+\hspace{-1pt} \|a\|^2_{W^{1,\infty}(\rO)} \big)\|w\|^2_{H^1(\rO)} \hspace{-1pt}+\hspace{-1pt} \|f\|^2_{L^2(\rO)} \hspace{-1pt}+\hspace{-1pt} \|g\|^2_{H^{-0.5}(\Gamma)} \Big] \,.
    \end{array}
    \end{equation}
\item Suppose that $a \in W^{1,4}(\rO)$, and $(w,q)\in H^1_0(\rO) \times L^2(\rO)$ is a weak solution to the following elliptic equation
    \begin{subequations}\label{elliptic_Dirichlet}
    \begin{alignat}{2}
    - \big[a^{jk}_{rs} w^r_{,j}\big]_{,k} + q_{,s} &= f^s &&\text{in}\quad\rO\,,\\
    \div w &= 0 &&\text{in}\quad\rO\,,\\
    w &= 0 \qquad&&\text{on}\quad\Gamma\,;
    \end{alignat}
    \end{subequations}
    that is,
    \begin{equation}\label{elliptic_weak_Dirichlet}
    \ \ \big(a^{jk}_{rs} w^r_{,j}, \varphi^s_{,k})_{L^2(\rO)} - (q,\div \varphi)_{L^2(\rO)} = (f, \varphi)_{L^2(\rO)} \quad\ \Forall \varphi \in H^1_0(\rO) \,.
    \end{equation}
    Then $(w,q) \in H^2(\rO) \times H^1(\rO)$ satisfies
    \begin{equation}\label{elliptic_est2}
    \qquad\qquad\ \ \|w\|^2_{H^2(\rO)} + \|q\|^2_{H^1(\rO)} \le C \Big[1 + \|f\|^2_{L^2(\rO)} + \|\nabla a\|^4_{L^4(\rO)} \|w\|^2_{H^1(\rO)} \Big] \,.
    \end{equation}
\end{enumerate}
\end{theorem}
The following corollary is a direct consequence of the second part of the theorem above.
\begin{corollary}\label{cor:elliptic_regularity}
Suppose that $a\in W^{1,4}(\rO)$, and $(w,q)\in H^2(\rO) \times H^1(\rO)$ is a solution to the following elliptic equation
\begin{alignat*}{2}
- \big[a^{jk}_{rs} w^r_{,j}\big]_{,k} + q_{,s} &= f^s &&\text{in}\quad\rO\,,\\
\div w &= g &&\text{in}\quad\rO\,,\\
w &= h \qquad&&\text{on}\quad\Gamma\,;
\end{alignat*}
for some $f\in L^2(\rO)$, $g\in H^1(\rO)$ and $h \in H^{1.5}(\Gamma)$ with $\smallexp{$\displaystyle{}\int_\Gamma$}\, h\cdot \rN\, dS = \smallexp{$\displaystyle{}\int_\rO$} g\, dx$. Then $(w,q) \in H^2(\rO) \times H^1(\rO)$, and satisfies
\begin{equation}\label{elliptic_est3}
\begin{array}{l}
\|w\|^2_{H^2(\rO)} + \|q\|^2_{H^1(\rO)} \le C \Big[1 + \|f\|^2_{L^2(\rO)} + \|g\|^2_{H^1(\rO)} + \|\nabla a\|^4_{L^4(\rO)} \|\nabla w\|^2_{L^2(\rO)} \vspace{.2cm}\\
\hspace{112pt} +\,\|h\|^2_{H^{1.5}(\Gamma)} + \|\nabla a\|^2_{L^4(\rO)} \|h\|^2_{H^1(\Gamma)} \Big]\,.
\end{array}
\end{equation}
\end{corollary}

\begin{remark}
When $a^{jk}_{rs} = \delta^j_k \delta^r_s$, Part {\rm(2)} of Theorem {\rm\ref{thm:elliptic_regularity}} is the {\rm(}lowest order{\rm)} regularity result for the steady Stokes equations. We also note that condition {\rm(\ref{smallness})} can also be replaced by the usual ellipticity assumption
$$
a^{jk}_{rs} \xi^{j}_{r} \xi^{k}_{s} \ge \lambda_1 \xi^i_j \xi^i_j + \lambda_2 \xi^i_j \xi^j_i
$$
without any modifications of the proof. Moreover, the condition $a\in W^{1,\infty}(\rO)$ in part {\rm(1)} can be relaxed to $a\in W^{1,4}(\rO)$; however, it is not used in our paper, so we omit the detail.
\end{remark}

\section{Main Theorem}\label{sec:main_thm}
Now we state the main result that we establish in this paper.
\begin{theorem}\label{thm:main}
For any bounded smooth domain $\rO\subseteq \bbR^2$, there is a number $\varsigma > 0$ such that for all $u_0 \in H^1(\Omega)$ and $h_0 \in H^2(\Gamma)$ satisfying $\|h_0\|_{H^{1.7}(\Gamma)} < \varsigma$, there exists a unique solution $(v,h,q)\in \V(\rT)\times \H(\rT) \times \Q(\rT)$ to equation {\rm(\ref{NSALE})} for some $\rT>0$. Moreover, $(v,h,q)$ satisfies
\begin{equation}
\|v\|_{\V(\rT)} + \|q\|_{\Q(\rT)} + \|h\|_{\H(\rT)} \le C \big[1 + \|u_0\|_{H^1(\rO)} + \|h_0\|_{H^2(\Gamma)}\big] \,. \label{main_energy_estimate}
\end{equation}
\end{theorem}

\begin{remark}
If the initial domain $\Omega$ has smooth boundary, we simply let $\rO = \Omega$ and choose $h_0 = 0$ which certainly satisfies the condition $\|h_0\|_{H^{1.7}(\Gamma)} < \varsigma$.
\end{remark}

\begin{remark}
The fluid velocity in Eulerian coordinate is given by $u = v\circ \psi^{-1}$. Since $(v,\psi)\in L^2(0,\rT;H^2(\rO))\times L^\infty(0,\rT;H^{2.5}(\rO))$, the fluid velocity $u$ belongs to $H^2(\Omega(t))$ for almost all $t\in [0,\rT]$, where $\Omega(t)$ is given by $\Omega(t) \equiv \psi(\Omega,t)$. However, as mentioned in Section {\rm\ref{sec:difficulties}} the existence of the flow map $\eta$ is not guaranteed since $u$ is not Lipschitz in spatial variable. Therefore, our result does not implies the existence of flow map, but implies that the existence of fluid velocity $u$ and fluid domain $\Omega(t) = \psi(\rO, t)$ such that {\rm(\ref{NS}c-f)} holds.
\end{remark}

\begin{remark}
Our theorem can also by applied to the case that $\rO$ is separated by an interface $\Gamma(t)$ such that $\rO = \rO^+(t) \cup \Gamma(t) \cup \rO^-(t)$, where $\Gamma(t) = \bdy \rO^-(t)$.
\end{remark}

\section{The construction of solutions}\label{sec:construction}
\subsection{The regularized problem}
Without loss of generality, we assume that the surface tension $\sigma = 1$. For any fixed $\ve>0$, we consider an approximation of (\ref{NSnew}) given by the following equations:
\begin{subequations}\label{NSreg}
\begin{alignat}{2}
\rJ^{-1} \psi^i_{,s} \psi^i_{,r} w^r_t - [\rL_\psi(w)]^s + q_{,s} &= F^s \qquad&&\text{in}\quad \rO\times (0,T)\,,\\
\div w &= 0\qquad &&\text{in}\quad \rO\times (0,T)\,,\\
\ell_\psi(w,q) &= \L_\ve(h) \rN + \ve^2 \Delta_0 w\qquad&&\text{on}\quad \Gamma\times (0,T)\,,\\
\Delta \psi &= 0 &&\text{in}\quad \rO\times (0,T)\,,\\
\psi &= e + h_{\ve\ve} \rN &&\text{on}\quad \Gamma\times (0,T)\,,\\
h_t &= \frac{w\cdot \rN}{1 \hspace{-1pt}+\hspace{-1pt} \rb_0 h_{\ve\ve}} \qquad&&\text{on}\quad \Gamma\times (0,T)\,,\\
w &= w_{0\ve} &&\text{on}\quad \rO\times \{t=0\}\,,\\
h &= h_{0\ve} &&\text{on}\quad \Gamma\times \{t=0\}\,,
\end{alignat}
\end{subequations}
where as before $\rJ = \det(\nabla \psi)$ and $\rA = (\nabla \psi)^{-1}$, $F$ is given by (\ref{defn:F}), $e$ is the identity map, $h_{\ve \ve} \equiv \eta_\ve \star (\eta_\ve \star h)$ is the double horizontal convolution of $h$, $\Delta_0$ is the Laplace-Beltrami operator defined by
\begin{equation}\label{Laplace-Beltrami_op}
\Delta_0 w = (w\circ \rX)''\circ \rX^{-1} \qquad\text{on}\quad \Gamma\,,
\end{equation}
and $\L_\ve$ is a differential operator defined by
\begin{equation}\label{defn:Le}
\L_\ve(h) \equiv \frac{\big(1 + \rb_0 h_{\ve\ve}\big) h'' - \rb_0 \big[ (1 \hspace{-1pt}+\hspace{-1pt} \rb_0h_{\ve\ve})^2+ 2 h_{\ve\ve}^{\prime\hspace{1pt}2}\big] - h_{\ve\ve} h_{\ve\ve}' \rb_0^\pprime}{\big[\big(1 \hspace{-1pt}+\hspace{-1pt} \rb_0h_{\ve\ve}\big)^2 + h_{\ve\ve}^{\prime\hspace{1pt}2} \big]^{3/2}}\,.
\end{equation}
We note that the solution $(w,q,\psi,h)$ depends on $\ve$.

\subsection{The linearization of the regularized problem}
Define
\begin{align*}
C_\rT(\rM) = \Big\{ (w,h)\in \W(\rT)\times \H_1(\rT)\,\Big|&\, \|w\|_{\W(\rT)} + \|h\|_{\H_1(\rT)} \le \rM,\ w(0) = w_{0\ve} \\
&\ h(0) = h_{0\ve}, \quad h_t(0) = \frac{w_{0\ve} \cdot \rN}{1 + \rb_0 h_{0\ve\ve\ve}} \Big\}
\end{align*}
for some $\rM > 1$ and $\rT<1$ to be determined later. We note that if $\rM \gg 1$ and $\rT>0$ small enough, then $C_\rT(\rM)$ is non-empty since
$$
\Big(w_{0\ve}, h_{0\ve} + \frac{t w_{0\ve} \cdot \rN}{1 + \rb_0 h_{0\ve\ve\ve}}\Big) \in C_\rT(\rM)\,.
$$
Let $(\bw,\bh)\in C_\rT(\rM)$ be given. For a fixed $\ve > 0$, let $\bpsi$ be the solution to
\begin{subequations}\label{bpsi_eq}
\begin{alignat}{2}
\Delta \bpsi &= 0 &&\text{in}\quad\rO\,,\\
\bpsi &= e + \bh_{\ve\ve} \rN \qquad&&\text{on}\quad\Gamma\,,
\end{alignat}
\end{subequations}
where we recall that $\bh_{\ve\ve} \equiv \eta_\ve\star (\eta_\ve\star \bh)$. We note that since $\rO$ is smooth and $\bh_{\ve\ve}$ is smooth, $\bpsi$ is a smooth diffeomorphism if $h \ll 1$.

Let $\bA = (\nabla \bpsi)^{-1}$ and $\bJ = \det(\nabla \bpsi)$ be defined according. We consider the following linearization of (\ref{NSreg}):
\begin{subequations}\label{NSregL}
\begin{alignat}{2}
w_t - \rL_\bpsi(w) + \nabla q &= \widebar{\rF} \qquad&&\text{in}\quad \rO\times (0,T)\,,\\
\div w &= 0\qquad &&\text{in}\quad \rO\times (0,T)\,,\\
\ell_\bpsi(w,q) &= \ve^2 \Delta_0 w + \widebar{\rG} \qquad&&\text{on}\quad \Gamma\times (0,T)\,,
\end{alignat}
\end{subequations}
where
\begin{align*}
\widebar{\rF}^s &= (\delta^r_s - \bJ^{-1} \bpsi^i_{,s} \bpsi^i_{,r})\, \bw^r_t - \bpsi^i_{,s} \bA^j_\ell (\bJ^{-1} \bpsi^\ell_{,r} \bw^r - \bpsi^\ell_t) (\bJ^{-1} \bpsi^i_{,s} \bw^s)_{,j} \\
&\quad - \bpsi^i_{,s} (\bJ^{-1} \bpsi^i_{,r})_t \bw^r + (\bpsi^i_{,s} \bA^k_\ell)_{,k} \big[\bA^j_\ell (\bJ^{-1} \bpsi^i_{,r} \bw^r)_{,j} + \bA^j_i (\bJ^{-1} \bpsi^\ell_{,r} \bw^r)_{,j} \big] \,,\\
\widebar{\rG} &= \frac{\big(1 + \rb_0 \bh_{\ve\ve}\big) \bh'' - \rb_0 \big[ (1 \hspace{-1pt}+\hspace{-1pt} \rb_0\bh_{\ve\ve})^2+ 2 \bh_{\ve\ve}^{\prime\hspace{1pt}2}\big] - \bh_{\ve\ve} \bh_{\ve\ve}' \rb_0^\pprime}{\big[\big(1 \hspace{-1pt}+\hspace{-1pt} \rb_0\bh_{\ve\ve}\big)^2 + \bh_{\ve\ve}^{\prime\hspace{1pt}2} \big]^{3/2}} \,.
\end{align*}

\begin{definition}\label{defn:weak_soln_reg}
A vector-valued function $w \in L^2(0,\rT;H^1(\rO))\cap L^2(0,\rT;H^1(\Gamma))$ satisfying $w_t \in L^2(0,\rT;L^2(\rO))$ is said to be a weak solution to {\rm(\ref{NSregL})} if
\begin{equation}\label{weak_reg}
\begin{array}{l}
\displaystyle{} (w_t,\varphi)_{L^2(\rO)} + \rB_\bpsi(w,\varphi) + \ve^2 (w^\pprime,\varphi^\pprime)_{L^2(\Gamma)} \vspace{.2cm}\\
\displaystyle{}\hspace{35pt} = (\widebar{\rF},\varphi)_{L^2(\rO)} + (\widebar{\rG},\varphi)_{L^2(\Gamma)} \quad \ \Forall \varphi\in H^1_\div(\rO)\,,\text{ \ a.e. } t\in[0,T]\,,
\end{array}
\end{equation}
and
$$
w = w_0 \qquad\text{on}\quad \rO\times \{t=0\}\,,
$$
where $\rB_\bpsi:H^1(\rO) \times H^1(\rO) \to \bbR$ is a bilinear form given by
\begin{equation}\label{defn:Bpsi}
\rB_\bpsi(w,\varphi) = \int_\rO \Big[\bA^j_\ell \bpsi^i_{,s} \bA^k_\ell (\bJ^{-1} \bpsi^i_{,r} w^r)_{,j} + \bA^k_\ell (\bJ^{-1} \bpsi^\ell_{,r} w^r)_{,s}\Big]\, \varphi^s_{,k} dx\,,
\end{equation}
and $H^1_\div(\rO) \equiv \big\{v\in H^1(\rO)\,\big|\, \div v = 0\big\}$ is the collection of divergence-free $H^1$ vector-valued functions.
\end{definition}

\subsection{Some a priori estimates}\label{sec:apriori_estimate}
In this sub-section, we establish some estimates concerning the smallness of certain important quantities which will be used throughout the paper. We remark that even though we have better regularity for the input $\bh$, we perform the following estimates under the assumption that $(\bh,\bh_t) \in L^\infty(0,\rT;H^2(\Gamma)) \times L^2(0,\rT;H^{1.4}(\Gamma))$.
\begin{proposition}
For any $\varsigma > 0$ {\rm(}which is an upper bound of $\|h_0\|_{H^{1.7}(\Gamma)}${\rm)}, there exists $\rT_\rM > 0$ such that
\begin{equation}
\|\bh(t)\|_{H^{1.7}(\Gamma)} < \varsigma \qquad\Forall t \in [0,\rT_\rM]\,, \label{smallness_of_bfh}
\end{equation}
\end{proposition}\label{prop:small_h}
\begin{proof}
By $H^{0.25}$-$H^{-0.25}$ duality,
\begin{align*}
\frac{1}{2} \frac{d}{dt} \|\bh\|^2_{H^{1.7}(\Gamma)} \le \|\bh\|_{H^2(\Gamma)} \|\bh_t\|_{H^{1.4}(\Gamma)} \le \rM \|\bh_t\|_{H^{1.4}(\Gamma)}
\end{align*}
which suggests that
\begin{align*}
\|\bh(t)\|^2_{H^{1.7}(\Gamma)} &\le \|h_{0\ve}\|^2_{H^{1.7}(\Gamma)} + \int_0^t \rM \|\bh_t\|_{H^{1.4}(\Gamma)} d\tilde{t} \\
&\le \|h_0\|^2_{H^{1.7}(\Gamma)} + \sqrt{t} \rM \|\bh_t\|_{L^2(0,t;H^{1.4}(\Gamma))} \le \|h_0\|^2_{H^{1.7}(\Gamma)} + \sqrt{t}\, \rM^2\,.
\end{align*}
Since $\|h_0\|_{H^{1.7}(\Gamma)} < \varsigma$, the inequality above suggests that
$$
\|\bh(t)\|^2_{H^{1.7}(\Gamma)} < \varsigma^2 \qquad\Forall t\in [0,\rT_\rM]\,.
$$
provided that $\rT_\rM>0 $ is chosen small enough.
\end{proof}

\begin{corollary}
There exists $\varsigma \ll 1$ {\rm(}with corresponding $\rT_\rM$ given in Proposition {\rm\ref{prop:small_h}}{\rm)} such that
\begin{equation}\label{smallness_of_JA_m_1}
\|\nabla \bpsi(t) - {\rm Id}\|_{L^\infty(\rO)} + \|\bA(t) - {\rm Id}\|_{L^\infty(\rO)} + \|\bJ(t) - 1\|_{L^\infty(\rO)} \le C \varsigma \quad\Forall t \in [0,\rT_\rM]
\end{equation}
for some generic constant $C$ independent of $\varsigma$. In particular,
\begin{align}
\frac{1}{2} \le \|\bJ(t)\|_{L^\infty(\rO)} \le \frac{3}{2} \qquad\Forall t\in [0,\rT_\rM]\,. \label{upper_lower_bound_for_J}
\end{align}
\end{corollary}
\begin{proof}
By the Sobolev embedding and elliptic regularity, for $t \in [0,\rT_\rM]$,
$$
\|\nabla \bpsi(t) - \id\|_{L^\infty(\Omega)} \le C \|\nabla \bpsi - \id\|_{H^{1.2}(\rO)} \le C \|\bh\|_{H^{1.7}(\Gamma)} \le C \varsigma\,.
$$
Therefore, since $\bA - \id = \bA (\id - \nabla \bpsi)$,
\begin{equation}
\|\bA(t) - \id\|_{L^\infty(\Omega)} \le C \|\bA\|_{L^\infty(\rO)} \varsigma \qquad\Forall t\in [0,\rT_\rM]. \label{boundedness_of_A}
\end{equation}
On the other hand, $\|\bA\|_{L^\infty(\rO)} - 1 \le \|\bA - \id\|_{L^\infty(\rO)}$; thus
$$
\|\bA\|_{L^\infty(\rO)} \le 1 + C \|\bA\|_{L^\infty(\rO)}\,.
$$
By choosing $\varsigma$ small enough,
\begin{align*}
\|\bA(t)\|_{L^\infty(\Omega)} \le \frac{3}{2}\,.
\end{align*}
The estimate above, together with (\ref{boundedness_of_A}), in turn implies that
\begin{equation}
\|\bA(t) - \id\|_{L^\infty(\Omega)} \le C \varsigma \qquad\Forall t\in [0,\rT_\rM]\,.
\end{equation}
The estimate for $\|\bJ(t) - 1\|_{L^\infty(\rO)}$ is similar since $\bJ = \det(\nabla \bpsi)$.
\end{proof}

\subsection{The construction of solutions to the regularized problem}
In this sub-section, our goal is to establish a map $\Phi: C_\rT(\rM) \to C_\rT(\rM)$ by choosing $\rM \gg 1$ and $\rT\ll 1$, and then show the existence of a fixed-point of the map. This fixed-point then is a strong solution to the regularized problem (\ref{NSreg}).

Since the velocity $w$ we are looking for is divergence-free, there are lots of ways of constructing the solution to the linear problem (\ref{NSregL}). Two typical ways are:
\begin{enumerate}
\item The Galerkin method: let $\{\re_k\}_{k=1}^\infty$ be an orthonomal divergence-free basis in $L^2(\rO)$ which is orthogonal in $H^1(\rO)$, and approximate (\ref{NSregL}) by projecting (\ref{NSregL}a) onto the subspace $\text{span}(\re_1,\cdots,\re_n)$. The projection of (\ref{NSregL}a) then becomes an ODE so the projected problem is solvable. The remaining thing to do is to guarantee the solution $w_n$ to the projected problem has an $n$-independent estimate in appropriate spaces.
\item The penalty method: approximate (\ref{NSregL}) by introducing a penalized parameter $\theta > 0$ and approximate the pressure $q$ in (\ref{NSregL}) by $-\smallexp{$\displaystyle{}\frac{1}{\theta}$}\, \div w_\theta$. Once a $\theta$-independent estimate is obtained, the limit of $w_\theta$ (as $\theta \to 0$) satisfies the divergence-free constraint automatically.
\end{enumerate}
Both methods mentioned above can be used to construct a solution to (\ref{NSregL}); however, it should be clear to the readers that the standard Galerkin method fails to work in constructing solutions to (\ref{NSALE}) since there is no basis in $\rV$ which preserves condition (\ref{NSALE}b) for all $t>0$. In order to explain why the penalty method does not work for the purpose of constructing a solution to (\ref{NSALE}) in the space $\V(\rT)\times \H(\rT)$ either, in the following we construct a solution to (\ref{NSregL}) using the penalty method. We remark here that it is much easier to use the Galerkin method to construct a weak solution to (\ref{NSregL}).

\subsubsection{The penalized problem}
Let $\theta>0$ be a given positive constant. We consider
\begin{subequations}\label{NSpenalty}
\begin{alignat}{2}
w_{\theta t} - \rL_\bpsi(w_\theta) + \nabla q_\theta &= \widebar{\rF} &&\text{in}\quad\rO\times (0,T)\,,\\
\ell_\bpsi(w_\theta, q_\theta) &= \ve^2 \Delta_0 w_\theta + \widebar{\rG} \qquad&&\text{on}\quad\Gamma\times (0,T)\,,\\
w_\theta &= w_0 &&\text{on}\quad\rO\times \{t=0\}\,,
\end{alignat}
\end{subequations}
where $q_\theta = - \smallexp{$\displaystyle{}\frac{1}{\theta}$}\, \div w_\theta$ is the penalized pressure. Similar to Definition \ref{defn:weak_soln_reg}, we have
\begin{definition}
A function $w_\theta \in L^2(0,\rT;\rV)$ satisfying $w_{\theta t} \in L^2(0,\rT;L^2(\rO))$ is said to be a weak solution to {\rm(\ref{NSpenalty})} if
\begin{equation}\label{weak_penalty}
\begin{array}{l}
\displaystyle{} (w_{\theta t},\varphi)_{L^2(\rO)} + \rB_\bpsi(w_\theta,\varphi) - (q_\theta,\div \varphi) + \ve^2 (w_\theta^\pprime,\varphi^\pprime)_{L^2(\Gamma)} \vspace{.2cm}\\
\displaystyle{}\hspace{55pt} = (\widebar{\rF},\varphi)_{L^2(\rO)} + (\widebar{\rG},\varphi)_{L^2(\Gamma)} \qquad \Forall \varphi\in \rV\,,\text{ \ a.e. } t\in[0,T]\,,
\end{array}
\end{equation}
and
$$
w_\theta = w_0 \qquad\text{on}\quad \rO\times \{t=0\}\,.
$$
%where we recall that $\rB_\bpsi:H^1(\rO) \times H^1(\rO) \to \bbR$ is a bilinear form given by
%$$
%\rB_\bpsi(w,\varphi) = \int_\rO \Big[\bA^j_\ell \bpsi^i_{,s} \bA^k_\ell (\bJ^{-1} \bpsi^i_{,r} w^r)_{,j} + \bA^k_\ell (\rJ^{-1} \bpsi^\ell_{,r} w^r)_{,s}\Big]\, \varphi^s_{,k} dx\,.
%$$
\end{definition}

We recall that $\rV = \big\{u\in H^1(\rO)\,\big|\, u\in H^1(\Gamma) %\text{ and } \smallexp{$\displaystyle{}\int_\Gamma$} u\cdot \rN \,dS = 0
\big\}$, and the space $L^2(0,T;\rV)$ consists of $u:[0,\rT]\to \rV$ such that
$$
\int_0^\rT \|u(t)\|^2_\rV\, dt < \infty\,.
$$
Our goal is to obtain a weak solution to (\ref{NSregL}) by passing the penalized parameter $\theta$ to the weak limit for the weak solution $v_\theta$ to (\ref{NSpenalty}).

\subsubsection{The existence of the unique weak solution to the penalized problem}\label{sec:Galerkin}
%\subsubsection{The Galerkin method}
The construction of a weak solution to (\ref{NSpenalty}) can be done using the Galerkin method. Let $\{\re_k\}_{k=1}^\infty$ be an orthonormal basis in $L^2(\rO)$ which is orthogonal in $\rV$, and let $w_n = \sum\limits_{k=1}^n d_k(t)\, \re_k(x)$ solve
\begin{equation}\label{ODE}
\begin{array}{l}
\displaystyle{} (w_{nt}, \varphi)_{L^2(\rO)} + \rB_\bpsi(w_n,\varphi) - (q_n, \div \varphi)_{L^2(\rO)} + \ve^2 (w_n^\pprime, \varphi^\pprime)_{L^2(\Gamma)} \vspace{.2cm}\\
\displaystyle{} \hspace{55pt} = (\widebar{\rF},\varphi)_{L^2(\rO)} + (\widebar{\rG},\varphi)_{L^2(\Gamma)} \qquad \ \Forall \varphi\in\text{span}(\re_1,\cdots,\re_n)
\end{array}
\end{equation}
with the initial condition
\begin{equation}\label{ODE_initial}
w_n(0) = \sum_{k=1}^n (w_0,\re_k)_{L^2(\rO)} \re_k\,,
\end{equation}
where $q_n \equiv -\smallexp{$\displaystyle{}\frac{1}{\theta}$}\, \div w_n$. %, and $\rB_\bpsi:H^1(\rO) \times H^1(\rO) \to \bbR$ is a bilinear form given by
%\begin{align*}
%\rB_\bpsi(w,\varphi) &= \int_\rO \bJ^{-1} \big(\bA^j_\ell \bpsi^i_{,s} \bA^k_\ell \bpsi^i_{,r} w^r_{,j} + w^k_{,s}\big)\, \varphi^s_{,k} dx \\
%&\quad + \int_\rO \big[\bpsi^i_{,s} \bA^k_\ell \bA^j_\ell (\bJ^{-1} \bpsi^i_{,r})_{,j} w^r + \bA^k_\ell (\bJ^{-1} \bpsi^\ell_{,r})_{,s} w^r \big] \varphi^s_{,k} dx\,.
%\end{align*}
We note that (\ref{ODE}) is an ODE; thus the fundamental theorem of ODE implies that there exists $\rT_n>0$ such that $w_n$ exists in the time interval $[0,\rT_n)$.

\begin{remark}
A basis of $\rV$ can be obtained by the following eigenvalue problem
\begin{alignat*}{2}
u - \Delta u &= \lambda u \qquad&&\text{in}\quad\rO\,,\\
\frac{\p u}{\p \rN} &= \Delta_0 u \qquad&&\text{on}\quad\Gamma\,.
\end{alignat*}
The study of this eigenvalue problem relies on the solvability of the elliptic problem
\begin{alignat*}{2}
u - \Delta u &= f \qquad&&\text{in}\quad\rO\,,\\
%\div u &= 0 &&\text{in}\quad\rO\,,\\
\frac{\p u}{\p \rN} &= \Delta_0 u \qquad&&\text{on}\quad\Gamma\,,
\end{alignat*}
while the solvability of the equation above is trivial because of the Lax-Milgram theorem.
\end{remark}

%\subsubsection{The estimates of $w_n$}
The next step is to obtain $n$-independent estimates for the finite dimensional approximation $w_n$ in certain Banach spaces. Before doing so, we note that if assuming that $\|h_0\|_{H^{1.7}(\Gamma)} < \varsigma$, by (\ref{smallness_of_JA_m_1}), for all $w\in H^1(\rO)$ we find that the bilinear form $\rB_\bpsi$ satisfies
\begin{align*}
& \rB_\bpsi(w,w) %= \int_\rO \bJ^{-1} \big(\bA^j_\ell \bpsi^i_{,s} \bA^k_\ell \bpsi^i_{,r} w^r_{,j} w^s_{,k} + w^k_{,s} w^s_{,k}\big) dx \nonumber\\
= \int_\rO \big(w^s_{,k} w^s_{,k} + w^k_{,s} w^s_{,k}\big) dx - \int_\rO \big[\bJ^{-1} \bA^j_\ell \bpsi^i_{,s} \bA^k_\ell \bpsi^i_{,r} - \delta^s_r \delta^j_k \big] w^r_{,j} w^s_{,k} dx \nonumber\\
&\qquad + \int_\rO \big[\bpsi^i_{,s} \bA^k_\ell \bA^j_\ell (\bJ^{-1} \bpsi^i_{,r})_{,j} w^r + \bA^k_\ell (\bJ^{-1} \bpsi^\ell_{,r})_{,s} w^r \big] w^s_{,k} dx + \int_\rO (\bJ^{-1} \hspace{-1pt}-\hspace{-1pt} 1)w^k_{,s} w^s_{,k} dx\nonumber\\
&\quad \ge \frac{1}{2} \|\Def w\|^2_{L^2(\rO)} \hspace{-1pt}-\hspace{-1pt} C \varsigma \|w\|^2_{H^1(\rO)} \hspace{-1pt}-\hspace{-1pt} C \|\bh\|_{H^2(\Gamma)} \|w\|^{1/2}_{L^2(\rO)}\|w\|^{3/2}_{H^1(\rO)} \nonumber\\
&\quad \ge \frac{1}{2} \|\Def w\|^2_{L^2(\rO)} \hspace{-1pt}-\hspace{-1pt} \big(C \varsigma \hspace{-1pt}+\hspace{-1pt} \delta_1\big) \|w\|^2_{H^1(\rO)} \hspace{-1pt}-\hspace{-1pt} C_{\delta_1} \big(\|h_0\|^2_{H^2(\Gamma)} \hspace{-1pt}+\hspace{-1pt} \rM^2\big) \|w\|^2_{L^2(\rO)}\,. \nonumber
\end{align*}
Now let $\varphi = w_n \in \text{span}(\re_1,\cdots,\re_n)$ in (\ref{ODE}), by the inequality above we have
\begin{align*}
& \frac{1}{2} \frac{d}{dt} \|w_n(t)\|^2_{L^2(\rO)} \hspace{-2pt}+\hspace{-1pt} \frac{1}{2} \|\Def w_n\|^2_{L^2(\rO)} \hspace{-2pt}+\hspace{-1pt} \theta \|q_n\|^2_{L^2(\rO)} \hspace{-2pt}+\hspace{-1pt} \ve^2 \|w_n^\pprime\|^2_{L^2(\Gamma)} \\
&\ \le\hspace{-1pt} \big( C \varsigma \hspace{-1pt}+\hspace{-1pt} \delta_1\big) \|w_n\|^2_{H^1(\rO)} \hspace{-2pt}+\hspace{-1pt} C_{\delta_1} \big(\|h_0\|^2_{H^2(\Gamma)} \hspace{-2pt}+\hspace{-1pt} \rM^2\big) \|w_n\|^2_{L^2(\rO)} \hspace{-2pt}+\hspace{-1pt} C \big[\|\widebar{\rF}\|^2_{L^2(\rO)} \hspace{-2pt}+\hspace{-1pt} \|\widebar{\rG}\|^2_{L^2(\Gamma)}\big].
\end{align*}
Integrating the inequality above in time over the time interval $(0,t)$, by Korn's inequality we further obtain that
\begin{align*}
& \|w_n(t)\|^2_{L^2(\rO)} + \int_0^t \Big[\|w_n\|^2_{H^1(\rO)} + \frac{1}{\theta} \|\div w_n\|^2_{L^2(\rO)} + \ve^2 \|w_n^\pprime\|^2_{L^2(\Gamma)}\Big] d\tilde{t} \\
&\quad \le C \|w_0\|^2_{L^2(\rO)} + \big(C \varsigma + 2 \delta_1\big) \int_0^t \|w_n\|^2_{H^1(\rO)} d\tilde{t} \\
&\qquad + C_{\delta_1} \big(\|h_0\|^2_{H^2(\Gamma)} \hspace{-2pt}+ \rM^2\big) \int_0^t \|w_n\|^2_{L^2(\rO)} d\tilde{t} + C \int_0^t \Big[\|\widebar{\rF}\|^2_{L^2(\rO)} + \|\widebar{\rG}\|^2_{L^2(\Gamma)} \Big] d\tilde{t}\,.
\end{align*}
Since
\begin{align*}
\|\widebar{\rF}\|_{L^2(\rO)} &\le C \varsigma \|\bw_t\|_{L^2(\rO)} + C \|D^2 \bpsi\|_{L^4(\rO)} \|D\bw\|_{L^4(\rO)} + C \|D^2 \bpsi\|^2_{L^4(\rO)} \|\bw\|_{L^\infty(\rO)} \\
&\quad + C \big(\|\bw\|_{L^4(\rO)} + \|\bpsi_t\|_{L^4(\rO)}\big) \big(\|D^2 \bpsi\|_{L^4(\rO)} \|\bw\|_{L^\infty(\rO)} + \|\bw\|_{W^{1,4}(\rO)}\big) \\
&\quad + C \|\nabla \bpsi_t\|_{L^4(\rO)} \|\bw\|_{L^4(\rO)} \,,\\
\|\widebar{\rG}\|_{L^2(\Gamma)} &\le C \Big[\|\bh\|_{H^2(\Gamma)} + \|\bh'\|^2_{L^4(\Gamma)}\Big] \le C \Big[\|\bh\|_{H^2(\Gamma)} + 1\Big] \,,
\end{align*}
by (\ref{sup_in_time_ineq}) and Young's inequality we obtain that
\begin{align*}
& \|\widebar{\rF}\|^2_{L^2(\rO)} + \|\widebar{\rG}\|^2_{L^2(\Gamma)} \le C \varsigma^2 \|\bw_t\|^2_{L^2(\rO)} + C \|\bh_t\|^2_{H^1(\Gamma)} \\
&\qquad\quad + C \rM^2 \big(\|w_0\|^2_{H^1(\rO)} + \|h_0\|^2_{H^2(\Gamma)} + \rM^2\big)^6 + \frac{1}{\rM^2} \|\bw\|^2_{H^2(\rO)}\,.
\end{align*}
Therefore, by choosing $\varsigma>0$, $\delta_1>0$ small enough, and $0 < \rT^*_\rM < \rT_\rM$ small enough so that $C \rM^2 \big(\|w_0\|^2_{H^1(\rO)} + \|h_0\|^2_{H^2(\Gamma)} + \rM^2\big)^6 \rT^*_\rM \le 1$,
\begin{align*}
& \|w_n(t)\|^2_{L^2(\rO)} \hspace{-1pt}+\hspace{-1pt} \int_0^t \Big[\|w_n\|^2_{H^1(\rO)} \hspace{-1pt}+\hspace{-1pt} \theta \|q_n\|^2_{L^2(\rO)} \hspace{-1pt}+\hspace{-1pt} \ve^2 \|w_n^\pprime\|^2_{L^2(\Gamma)}\Big] d\tilde{t} \\
&\quad \le C \Big[1 \hspace{-1pt}+\hspace{-1pt} \|w_0\|^2_{L^2(\rO)}\Big] \hspace{-1pt}+\hspace{-1pt} C \big(\|h_0\|^2_{H^2(\Gamma)} \hspace{-2pt}+ \rM^2\big) \int_0^t \hspace{-1pt} \|w_n\|^2_{L^2(\rO)} d\tilde{t} \hspace{-1pt}+\hspace{-1pt} C \varsigma^2 \int_0^t \hspace{-1pt}\|\bw_t\|^2_{L^2(\rO)} d\tilde{t} %\\
%&\qquad\quad + C \rM^2 \big(\|w_0\|^2_{H^1(\rO)} + \|h_0\|^2_{H^2(\Gamma)} + \rM^2\big)^6 t + \frac{1}{\rM^2} \int_0^t \|\bw\|^2_{H^2(\rO)} d\tilde{t}
\end{align*}
for all $t\in [0,\rT^*_\rM]$. The Gronwall inequality further suggests that
\begin{equation}\label{wn_L2H1_est}
\begin{array}{l}
\displaystyle{} \|w_n(t)\|^2_{L^2(\rO)} + \int_0^t \Big[\|w_n\|^2_{H^1(\rO)} + \theta \|q_n\|^2_{L^2(\rO)} + \ve^2 \|w_n^\pprime\|^2_{L^2(\Gamma)}\Big] d\tilde{t} \vspace{.1cm} \\
\displaystyle{} \qquad\quad \le C \Big[1 + \|w_0\|^2_{L^2(\rO)} \Big] + C \varsigma^2 \int_0^t \|\bw_t\|^2_{L^2(\rO)} d\tilde{t}
\end{array}
\end{equation}
if $\rT^*_\rM$ is chosen even smaller.

%\subsubsection{The estimates of $w_{nt}$}
We also need an estimate of $w_{nt}$ in order to pass $n$ to the limit. Since $w_{nt}$ belongs to the span of $\re_1,\cdots,\re_n$, it can be used as a test function in (\ref{ODE}). By doing so we obtain that
\begin{equation}\label{wt_eq}
\begin{array}{l}
\displaystyle{} \|w_{nt}\|^2_{L^2(\rO)} + \int_\rO \Big[\bpsi^i_{,s} \bA^k_\ell \bA^j_\ell (\bJ^{-1} \bpsi^i_{,r} w^r_{n})_{,j} + \bA^k_\ell (\bJ^{-1} \bpsi^i_{,r} w^r_{n})_{,s} \Big] w^s_{nt,k} dx \vspace{.2cm}\\
\displaystyle{}\qquad + \frac{1}{2} \frac{d}{dt} \Big[\theta \|q_n\|^2_{L^2(\rO)} \hspace{-1pt}+ \ve^2 \|w_n^\pprime\|^2_{L^2(\Gamma)}\Big] = (\widebar{\rF},w_{nt})_{L^2(\rO)} \hspace{-1pt}+ \int_\Gamma \hspace{-1pt}\widebar{\rG} \cdot w_{nt} dS\,.
\end{array}
\end{equation}
Since
\begin{align*}
& \int_\rO \Big[\bpsi^i_{,s} \bA^k_\ell \bA^j_\ell (\bJ^{-1} \bpsi^i_{,r} w^r_n)_{,j} + \bA^k_\ell (\bJ^{-1} \bpsi^i_{,r} w^r_{n})_{,s} \Big] w^s_{nt,k} dx \\
&\qquad = \int_\rO \bJ \Big[\bA^j_\ell (\bJ^{-1} \bpsi^i_{,r} w^r_n)_{,j} + \bA^j_i (\bJ^{-1} \bpsi^\ell_{,r} w^r_n)_{,j} \Big] \bA^k_\ell (\bJ^{-1} \bpsi^i_{,s} w^s_{nt,k})\, dx \\
&\qquad = \int_\rO \bJ \Big[\bA^j_\ell (\bJ^{-1} \bpsi^i_{,r} w^r_n)_{,j} + \bA^j_i (\bJ^{-1} \bpsi^\ell_{,r} w^r_n)_{,j} \Big] \times \Big[\bA^k_\ell (\bJ^{-1} \bpsi^i_{,s} w^s_n)_{t,k} \\
&\qquad\qquad\quad - \bA^k_\ell (\bJ^{-1} \bpsi^i_{,s})_t w^s_{n,k} - \bA^k_\ell (\bJ^{-1} \bpsi^i_{,s})_{,k} w^s_{nt} - \bA^k_\ell (\bJ^{-1} \bpsi^i_{,s})_{t,k} w^s_n \Big] dx \,,
%&\qquad = \int_\rO \bJ \Big[\bA^j_\ell (\bJ^{-1} \bpsi^i_{,r} w^r_n)_{,j} + \bA^j_i (\bJ^{-1} \bpsi^\ell_{,r} w^r_n)_{,j} \Big] \bA^k_\ell (\bJ^{-1} \bpsi^i_{,s} w^s_n)_{t,k} dx \\
%&\qquad\quad - \int_\rO \bJ \Big[\bA^j_\ell (\bJ^{-1} \bpsi^i_{,r} w^r_n)_{,j} + \bA^j_i (\bJ^{-1} \bpsi^\ell_{,r} w^r_n)_{,j} \Big] \bA^k_\ell (\bJ^{-1} \bpsi^i_{,s})_t w^s_{n,k} dx \\
%&\qquad\quad - \int_\rO \bJ \Big[\bA^j_\ell (\bJ^{-1} \bpsi^i_{,r} w^r_n)_{,j} + \bA^j_i (\bJ^{-1} \bpsi^\ell_{,r} w^r_n)_{,j} \Big] \bA^k_\ell (\bJ^{-1} \bpsi^i_{,s})_{,k} w^s_{nt} dx \\
%&\qquad\quad - \int_\rO \bJ \Big[\bA^j_\ell (\bJ^{-1} \bpsi^i_{,r} w^r_n)_{,j} + \bA^j_i (\bJ^{-1} \bpsi^\ell_{,r} w^r_n)_{,j} \Big] \bA^k_\ell (\bJ^{-1} \bpsi^i_{,s})_{t,k} w^s_n dx \,,
\end{align*}
by the symmetry of $\bA^j_\ell (\bJ^{-1} \bpsi^i_{,r} w^r_n)_{,j} + \bA^j_i (\bJ^{-1} \bpsi^\ell_{,r} w^r_n)_{,j}$ in $(i,\ell)$ we find that
\begin{equation}\label{extra_regularity1}
\begin{array}{ll}
&\displaystyle{} \int_\rO \bJ \Big[\bA^j_\ell (\bJ^{-1} \bpsi^i_{,r} w^r_n)_{,j} + \bA^j_i (\bJ^{-1} \bpsi^\ell_{,r} w^r_n)_{,j} \Big] \bA^k_\ell (\bJ^{-1} \bpsi^i_{,s} w^s_n)_{t,k} dx \vspace{.1cm}\\
&\displaystyle{} \qquad \ge \frac{1}{2} \frac{d}{dt} \int_\rO \bJ \Big[\bA^j_\ell (\bJ^{-1} \bpsi^i_{,r} w^r_n)_{,j} + \bA^j_i (\bJ^{-1} \bpsi^\ell_{,r} w^r_n)_{,j} \Big] \bA^k_\ell (\bJ^{-1} \bpsi^i_{,s} w^s_n)_{,k} dx \vspace{.2cm}\\
&\displaystyle{} \qquad\quad - C \Big[\|\nabla^2 \bpsi\|_{L^\infty(\rO)} \|w_{nt}\|_{L^2(\rO)} + \|\nabla \bpsi_t\|_{L^\infty(\rO)} \|\nabla w_n\|_{L^2(\rO)} \vspace{.2cm}\\
&\displaystyle{} \hspace{50pt} + \|\nabla^2 \bpsi_t\|_{L^4(\rO)} \|w_n\|_{L^4(\rO)} \Big] \big\|\nabla [\bJ^{-1} (\nabla \bpsi)^\rT w_n]\big\|_{L^2(\rO)} \,;
\end{array}
\end{equation}
thus
\begin{align*}
& \int_\rO \Big[\bpsi^i_{,s} \bA^k_\ell \bA^j_\ell (\bJ^{-1} \bpsi^i_{,r} w^r_n)_{,j} + \bA^k_\ell (\bJ^{-1} \bpsi^i_{,r} w^r_{n})_{,s} \Big] w^s_{nt,k} dx \\
&\qquad \ge \frac{1}{2} \frac{d}{dt} \int_\rO \bJ \Big[\bA^j_\ell (\bJ^{-1} \bpsi^i_{,r} w^r_n)_{,j} + \bA^j_i (\bJ^{-1} \bpsi^\ell_{,r} w^r_n)_{,j} \Big] \bA^k_\ell (\bJ^{-1} \bpsi^i_{,s} w^s_n)_{,k} dx \\
%&\qquad\quad - C_\delta \Big[\|\nabla^2 \bpsi w_n\|^2_{L^2(\rO)} + \|\bJ \nabla \bpsi - \id\|^2_{L^\infty(\rO)} \|\nabla w_n\|^2_{L^2(\rO)} \Big] - \delta \|w_{nt}\|^2_{L^2(\rO)} \\
%&\qquad\quad - C \Big[\|\bh_{\ve}\|_{H^{2.75}(\Gamma)} \|w_{nt}\|_{L^2(\rO)} + \|\bh_{\ve t}\|_{H^{1.75}(\Gamma)} \|w_n\|_{H^1(\rO)} \Big] \times \big\|\nabla [\bJ^{-1} (\nabla \bpsi)^\rT w_n]\big\|_{L^2(\rO)} \\
&\qquad\quad - C_\delta \Big[\|\bh_\ve\|^2_{H^{2.7}(\Gamma)} + \|\bh_t\|_{H^2(\Gamma)} \Big] \|w_n\|^2_{H^1(\rO)} - \delta \|w_{nt}\|^2_{L^2(\rO)} \,;
\end{align*}
Moreover, by the embedding $H^1(\rO) \contsubset H^{0.5}(\Gamma)$,
\begin{align*}
&\int_0^t \int_\Gamma \widebar{\rG} \cdot w_{nt} dS d\tilde{t} = \int_\Gamma \widebar{\rG} \cdot w_n dS \Big|_{\,\tilde{t}=0}^{\,\tilde{t} = t} - \int_0^t \int_\Gamma \widebar{\rG}_t \cdot w_n dS d\tilde{t} \\
&\qquad\quad\le C_{\delta_1} %\sup_{\tilde{t}\in [0,t]}
\|\widebar{\rG}(t)\|^2_{H^{-0.5}(\Gamma)} + \delta_1 \|w_n(t)\|^2_{H^1(\rO)} + C \|w_n(0)\|^2_{H^1(\rO)} \\
&\qquad\qquad + \frac{1}{\rM^2} \int_0^t \|\widebar{\rG}_t\|^2_{H^{-0.5}(\Gamma)} d\tilde{t} + C \rM^2 \int_0^t \|w_n\|^2_{H^1(\rO)} d\tilde{t}\,.
\end{align*}
As a consequence, integrating (\ref{wt_eq}) in time over the time interval $(0,t)$, by (\ref{smallness_of_JA_m_1}) and
choosing $\delta>0$ small enough we conclude that
\begin{align}
& \|\Def w_n(t)\|^2_{L^2(\rO)} \hspace{-1pt}+\hspace{-1pt} \theta \|q_n(t)\|^2_{L^2(\rO)} \hspace{-1pt}+\hspace{-1pt} \ve^2 \|w_n^\pprime(t)\|^2_{L^2(\Gamma)} \hspace{-1pt}+\hspace{-1pt} \int_0^t \|w_{nt}\|^2_{L^2(\rO)} d\tilde{t} \label{wn_t_L2L2_est}\\
&\quad \le C \varsigma \|\nabla w_n\|^2_{L^2(\rO)} \hspace{-1pt}+\hspace{-1pt} C \|D^2 \bpsi w_n\|^2_{L^2(\rO)} \hspace{-1pt}+\hspace{-1pt} C \|w_n(0)\|^2_{H^1(\rO)} \nonumber\\
&\qquad \hspace{-1pt}+\hspace{-1pt} C_{\delta_1} \|\widebar{\rG}(t)\|^2_{H^{-0.5}(\Gamma)} \hspace{-1pt}+\hspace{-1pt} C_{\delta_2} \int_0^t \hspace{-1pt}\|\widebar{\rF}\|^2_{L^2(\rO)} d\tilde{t} \hspace{-1pt}+\hspace{-1pt} \delta_1 \|w_n(t)\|^2_{H^1(\rO)} \hspace{-1pt}+\hspace{-1pt} \delta_2 \int_0^t \hspace{-1pt} \|w_{nt}\|^2_{L^2(\rO)} d\tilde{t} \nonumber\\
&\qquad \hspace{-1pt}+\hspace{-1pt} \frac{1}{\rM^2} \int_0^t \|\widebar{\rG}_t\|^2_{H^{-0.5}(\Gamma)} d\tilde{t} \hspace{-1pt}+\hspace{-1pt} C \int_0^t \Big[\|\bh_\ve\|^2_{H^{2.7}(\Gamma)} \hspace{-1pt}+\hspace{-1pt} \|\bh_t\|_{H^2(\Gamma)} \hspace{-1pt}+\hspace{-1pt} \rM^2 \Big]\|w_n\|^2_{H^1(\rO)} d\tilde{t} \,. \nonumber
\end{align}
Since
\begin{align*}
%\|\widebar{\rF}\|_{L^2(\rO)} &\le C \big[\varsigma \hspace{-1pt}+\hspace{-1pt} \sqrt{t}(\|h_0\|_{H^2(\Gamma)} \hspace{-1pt}+\hspace{-1pt} \rM)\rM \big] \|\widebar{w}_t\|_{L^2(\rO)} + \\
\|\widebar{\rG}\|_{H^{-0.5}(\Gamma)} &\le C (\|b_0\|_{H^1(\Gamma)}) \|\bh\|_{H^{1.5}(\Gamma)} \le C \Big[ \|h_0\|_{H^{1.5}(\Gamma)} + \sqrt{t} \|\bh_t\|_{L^2(0,t;H^{1.5}(\Gamma))}\Big] \\
\|\widebar{\rG}_t\|_{H^{-0.5}(\Gamma)} &\le C \Big[\|\bh_t\|_{H^{1.5}(\Gamma)} + \|\bh_t' \bh''\|_{H^{-0.5}(\Gamma)} \Big] \le C \|\bh_t\|_{H^{1.5}(\Gamma)}\,,
\end{align*}
choosing $\varsigma \ll 1$, $\delta_1, \delta_2 > 0$ small enough, by Korn's inequality the combination of (\ref{wn_L2H1_est}) and (\ref{wn_t_L2L2_est}) suggests that if $t \in [0,\rT^*_\rM]$,
\begin{align}
& \|w_n(t)\|^2_{H^1(\rO)} + \theta \|q_n(t)\|^2_{L^2(\rO)} + \ve^2 \|w_n(t)\|^2_{H^1(\Gamma)} + \int_0^t \|w_{nt}\|^2_{L^2(\rO)} d\tilde{t} \nonumber\\
&\qquad \le C \Big[1 + \|w_0\|^2_{H^1(\rO)} + \|h_0\|^2_{H^{1.5}(\Gamma)} \Big] + C \varsigma^2 \rM^2 \label{wnt_estimate_temp1}\\
&\qquad\quad + C \int_0^t \Big[\|\bh_\ve\|^2_{H^{2.7}(\Gamma)} \hspace{-1pt}+\hspace{-1pt} \|\bh_t\|_{H^2(\Gamma)} \hspace{-1pt}+\hspace{-1pt} \rM^2 \Big]\|w_n\|^2_{H^1(\rO)} d\tilde{t} \nonumber
\end{align}
for some constant $C$ independent of the initial data. Since
$$
\|\bh_\ve\|^2_{H^{2.7}(\Gamma)} \le C \ve^{-3/2} \|\bh\|_{H^2(\Gamma)}\,,
$$
the Gronwall inequality further implies that we may choose $\rT_\ve < \rT^*_\rM$ such that for all $t\in [0,\rT_\ve]$,
\begin{equation}\label{wn_LinfH1_est}
\begin{array}{l}
\displaystyle{} \|w_n(t)\|^2_{H^1(\rO)} + \theta \|q_n(t)\|^2_{L^2(\rO)} + \ve^2 \|w_n^\pprime(t)\|^2_{L^2(\Gamma)} + \int_0^t \|w_{nt}(\tilde{t})\|^2_{L^2(\rO)} d\tilde{t} \vspace{.15cm}\\
\displaystyle{} \qquad\qquad\quad \le C \Big[1 + \|w_0\|^2_{H^1(\rO)} + \|h_0\|^2_{H^{1.5}(\Gamma)} \Big] + C \varsigma^2 \rM^2 \,.
\end{array}
\end{equation}

%\subsubsection{The existence of a weak solution to {\rm(\ref{NSregL})}}
Estimate (\ref{wn_LinfH1_est}) provides an $n$-independent upper bound for $w_{nt} \in L^2(0,\rT_\ve; L^2(\rO))$ and $w_n \in L^\infty(0,\rT_\ve;\rV)$. Therefore, there exists a subsequence $n_k$ of $n$ such that
\begin{alignat*}{2}
w_{n_k} &\rightharpoonup w_\theta &&\text{in}\quad L^p(0,\rT_\ve;\rV)\qquad\quad\Forall p\in (1,\infty)\,, \\
w_{n_k t} &\rightharpoonup w_{\theta t} \qquad&&\text{in}\quad L^2(0,\rT_\ve;L^2(\rO))\,.
\end{alignat*}
The weak limit $w_\theta$ satisfies
\begin{equation}\label{wtheta_LinfH1_est_temp}
\begin{array}{l}
\displaystyle{} \Big[\int_0^t \hspace{-2pt}\Big(\|w_\theta(\tilde{t})\|^{2p}_{H^1(\rO)} + \|\sqrt{\theta} q_\theta(\tilde{t})\big\|^{2p}_{L^2(\rO)} + \|\ve w_\theta^\pprime(\tilde{t})\|^{2p}_{L^2(\Gamma)}\Big) d\tilde{t}\Big]^{1/p} \vspace{.15cm}\\
\displaystyle{} \qquad + \int_0^t \|w_{\theta t}(\tilde{t})\|^2_{L^2(\rO)} d\tilde{t} \le C \Big[\|w_0\|^2_{H^1(\rO)} + \|h_0\|^2_{H^{1.5}(\Gamma)} + 1 \Big] + C \varsigma^2 \rM^2
\end{array}
\end{equation}
and the variational form
\begin{equation}\label{weak_penalty1}
\begin{array}{l}
\displaystyle{} \int_0^{\rT_\ve} \Big[(w_{\theta t},\varphi)_{L^2(\rO)} + \rB_\bpsi(w_\theta,\varphi) - (q_\theta,\div \varphi) + \ve^2 (w_\theta^\pprime,\varphi^\pprime)_{L^2(\Gamma)}\Big] dt \vspace{.1cm}\\
\displaystyle{}\hspace{45pt} = \int_0^{\rT_\ve} \Big[(\widebar{\rF},\varphi)_{L^2(\rO)} + (\widebar{\rG},\varphi)_{L^2(\Gamma)}\Big] dt \qquad \Forall \varphi\in L^2(0,\rT_\ve;\rV)\,.
\end{array}
\end{equation}
Since the right-hand side of (\ref{wtheta_LinfH1_est_temp}) is independent of the exponent $p$, we let $p\to \infty$ and obtain that
\begin{equation}\label{wtheta_LinfH1_est}
\begin{array}{l}
\displaystyle{} \sup_{t\in[0,\rT_\ve]} \Big[\|w_\theta(t)\|^2_{H^1(\rO)} + \theta \|q_\theta(t)\|^2_{L^2(\rO)} + \ve^2 \|w_\theta^\pprime(t)\|^2_{L^2(\Gamma)}\Big] \vspace{.15cm}\\
\displaystyle{} \quad\ + \int_0^{\rT_\ve} \hspace{-2pt} \|w_{\theta t}(t)\|^2_{L^2(\rO)} dt \le C \Big[\|w_0\|^2_{H^1(\rO)} + \|h_0\|^2_{H^{1.5}(\Gamma)} + 1 \Big] + C \varsigma^2 \rM^2 \,.
\end{array}
\end{equation}
The same argument suggests that there exists $\theta_k \to 0$ such that
\begin{alignat*}{2}
w_{\theta_k} &\rightharpoonup w &&\text{in}\quad L^p(0,\rT_\ve;\rV)\qquad\quad\Forall p\in (1,\infty)\,, \\
w_{\theta_k t} &\rightharpoonup w_t \qquad&&\text{in}\quad L^2(0,\rT_\ve;L^2(\rO))\,,
\end{alignat*}
and the weak limit $w$ satisfies
\begin{equation}\label{w_LinfH1_est}
\begin{array}{l}
\displaystyle{} \sup_{t\in [0,\rT_\ve]} \Big[\|w(t)\|^2_{H^1(\rO)} + \ve^2 \|w^\pprime(t)\|^2_{L^2(\Gamma)}\Big] + \int_0^{\rT_\ve} \|w_t(t)\|^2_{L^2(\rO)} dt \vspace{.15cm}\\
\displaystyle{} \qquad\qquad \le C_1 \Big[\|w_0\|^2_{H^1(\rO)} + \|h_0\|^2_{H^{1.5}(\Gamma)} + 1 \Big] + C_2 \varsigma^2 \rM^2 \,,
\end{array}
\end{equation}
and $\div w = 0$ since $\sqrt{\theta} q_\theta = \smallexp{$\displaystyle{}\frac{1}{\sqrt{\theta}}$}\, \div w_\theta$ is uniformly bounded in $L^\infty(0,\rT_\ve;L^2(\rO))$. Moreover, (\ref{weak_penalty1}) implies that for all $a,b\in (0,\rT_\ve)$,
\begin{align*}
& \int_a^b \Big[(w_t,\varphi)_{L^2(\rO)} + \rB_\bpsi(w,\varphi) + \ve^2 (w^\pprime,\varphi^\pprime)_{L^2(\Gamma)}\Big] dt \\
&\hspace{35pt} = \int_a^b \Big[(\widebar{\rF},\varphi)_{L^2(\rO)} + (\widebar{\rG},\varphi)_{L^2(\Gamma)}\Big] dt \qquad\Forall \varphi\in \rV\cap H^1_\div(\rO)\,,
\end{align*}
and Lebesgue's differentiation theorem further suggests that $w$ satisfy (\ref{weak_reg}).

To finish the process of construction a weak solution to (\ref{NSregL}), it remains to show that $w(0) = w_0$. Let $\zeta:[0,\rT_\ve] \to \bbR$ be a non-negative smooth function (of $t$) such that $\zeta(0) = 1$ and $\zeta({\rT_\ve}) = 0$, and $\varphi \in \text{span}(\re_1,\cdots,\re_n)$. The use of $\zeta \varphi$ as a test function in (\ref{ODE}) and then integrating in time over the time interval $(0,\rT_\ve)$, integrating by parts in time we obtain that
\begin{align*}
& (w_n(0), \varphi)_{L^2(\rO)} - \int_0^{\rT_\ve} (w_n, \zeta^\pprime \varphi)_{L^2(\rO)} dt + \int_0^{\rT_\ve} \hspace{-2pt}\Big[\rB_\bpsi(w_n, \zeta \varphi) - (q_n, \zeta \div\varphi)_{L^2(\rO)} \\
&\qquad\qquad + \ve^2 (w^\pprime_n, \zeta \varphi^\pprime)_{L^2(\Gamma)} \Big] dt = \int_0^{\rT_\ve} \hspace{-2pt}\Big[(\widebar{\rF},\zeta \varphi)_{L^2(\rO)} + (\widebar{\rG},\zeta \varphi)_{L^2(\Gamma)}\Big] dt\,.
\end{align*}
By (\ref{ODE_initial}), passing $n\to \infty$ we find that
\begin{align*}
& (w_0, \varphi)_{L^2(\rO)} - \int_0^{\rT_\ve} \hspace{-2pt} (w_\theta, \zeta^\pprime \varphi)_{L^2(\rO)} dt + \int_0^{\rT_\ve} \hspace{-2pt}\Big[\rB_\bpsi(w_\theta, \zeta \varphi) - (q_\theta, \zeta \div\varphi)_{L^2(\rO)} \\
&\qquad + \ve^2 (w^\pprime_\theta, \zeta \varphi^\pprime)_{L^2(\Gamma)} \Big] dt = \int_0^{\rT_\ve} \hspace{-2pt}\Big[(\widebar{\rF},\zeta \varphi)_{L^2(\rO)} + (\widebar{\rG},\zeta \varphi)_{L^2(\Gamma)}\Big] dt \quad\ \Forall \varphi \in H^1(\rO)\,.
\end{align*}
On the other hand, the use of $\zeta \varphi$ as a test function in (\ref{weak_penalty1}) suggests that
\begin{align*}
& (w_\theta(0), \varphi)_{L^2(\rO)} - \int_0^{\rT_\ve} \hspace{-2pt}(w_\theta, \zeta^\pprime \varphi)_{L^2(\rO)} dt + \int_0^{\rT_\ve} \hspace{-2pt}\Big[\rB_\bpsi(w_\theta, \zeta \varphi) - (q_\theta, \zeta \div\varphi)_{L^2(\rO)} \\
&\quad + \ve^2 (w^\pprime_\theta, \zeta \varphi^\pprime)_{L^2(\Gamma)} \Big] dt = \int_0^{\rT_\ve} \hspace{-2pt}\Big[(\widebar{\rF},\zeta \varphi)_{L^2(\rO)} + (\widebar{\rG},\zeta \varphi)_{L^2(\Gamma)}\Big] dt \quad\ \Forall \varphi \in H^1(\rO)\,.
\end{align*}
The comparison between the two identities above enable us to conclude the identity $w_\theta(0) = w_0$. Similar argument can be used to conclude that $w(0) = w_0$, and is left to the readers. The uniqueness of the weak solution should also be clear to the readers.

Finally, let $T:\rV \to \bbR$ be given by
$$
T(\varphi) = (w_t,\varphi)_{L^2(\rO)} + \rB_\bpsi(w,\varphi) + \ve^2 (w^\pprime,\varphi^\pprime)_{L^2(\Gamma)} - (\widebar{\rF},\varphi)_{L^2(\rO)} - (\widebar{\rG},\varphi)_{L^2(\Gamma)}\,,
$$
where $w$ is the weak solution to (\ref{NSregL}). Since $T:\rV \cap H^1_\div(\rO) \to \{0\}$, the Lagrange multiplier lemma %Lemma \ref{lem:Lagrange_multiplier}
implies that there exists a unique $q \in L^2(\rO)$ such that
$$
T(\varphi) = (q,\div \varphi)_{L^2(\rO)} \qquad\Forall \varphi\in \rV\,,
$$
and by (\ref{smallness_of_JA_m_1}) we also conclude that $q$ satisfies
\begin{equation}\label{q_est}
\|q\|_{L^2(\rO)} \le C \Big[\|w_t\|_{L^2(\rO)} \hspace{-1pt}+\hspace{-1pt} \|\nabla w\|_{L^2(\rO)} \hspace{-1pt}+\hspace{-1pt} \ve^2 \|w\|_{H^1(\Gamma)} \hspace{-1pt}+\hspace{-1pt} \|\widebar{\rF}\|_{L^2(\rO)} \hspace{-1pt}+\hspace{-1pt} \|\widebar{\rG}\|_{H^{-0.5}(\Gamma)}\Big]\,.
\end{equation}

\begin{remark}\label{rmk:issues}
Now we explain briefly why the usual ALE formulation {\rm(\ref{NSALE})} is not a good choice in obtaining a solution $(v,q) \in \V(\rT) \times \Q(\rT)$. Similar to {\rm(\ref{NSreg})} and {\rm(\ref{NSregL})}, we introduce the linear penalized problem of {\rm(\ref{NSALE})} as the following equation
\begin{subequations}\label{NSALEpenalty}
\begin{alignat}{2}
v^i_{\theta t} + \bA^k_\ell (\bA^j_\ell v^i_{\theta,j} + \bA^j_i v^\ell_{\theta,j}\big)_{,k} + \bA^k_i q_{\theta,k} &= \widebar{\rF}_1 \qquad&&\text{in}\quad\rO\times (0,T)\,,\\
(\bA^j_\ell v^i_{\theta,j} + \bA^j_i v^\ell_{\theta,j} - q_\theta \delta^\ell_i) \bA^k_\ell \rN_k &= \ve^2 \Delta_0 v_\theta + \widebar{\rG}_1 &&\text{on}\quad \Gamma\times (0,T)\,,\\
%\Delta \psi &= 0 &&\text{in}\quad \rO\times (0,T)\,,\\
%\psi_\theta &= e + h_{\theta\ve\ve} \rN &&\text{on}\quad \Gamma\times (0,T)\,,\\
%h_{\theta t} &= \frac{\bJ \bA^{\hspace{-1pt}\rT} \rN}{1 \hspace{-1pt}+\hspace{-1pt} \rb_0 h_{\ve\ve}}\cdot v_\theta \quad\ \ &&\text{on}\quad \Gamma\times (0,T)\,,\\
v_\theta &= v_0 \equiv u_0\circ \psi_0 \qquad&&\text{on}\quad \rO\times\{t=0\},%\\
%h_\theta &= h_{0\ve} &&\text{on}\quad \Gamma\times \{t=0\}\,,
\end{alignat}
\end{subequations}
for some functions $\widebar{\rF}_1$ and $\widebar{\rG}_1$, where $\bA$ could be obtained from the ALE map $\bpsi$ or Lagrangian coordinate $\widebar{\eta}$ and $q_\theta = - \smallexp{$\displaystyle{}\frac{1}{\theta}$}\, \bA^j_i v^i_{\theta,j}$.
The same as before, let $v_n(x,t) = \sum\limits_{k=1}^n d_k(t)\, \re_k(x)$ be the finite dimensional approximation of the solution to {\rm(\ref{NSALEpenalty})} satisfying the ODE
\begin{equation}\label{weak_soln_ALE_penalty}
\begin{array}{l}
\displaystyle{} (v_{nt},\varphi)_{L^2(\rO)} \hspace{-2pt}+\hspace{-1pt} \frac{1}{2} \big((\bA^j_\ell v^i_{n,j} \hspace{-2pt}+\hspace{-1pt} \bA^j_i v^\ell_{n,j}), (\bA^k_\ell \varphi^i_{,k} \hspace{-2pt}+\hspace{-1pt} \bA^k_i \varphi^\ell_{,j})\big)_{L^2(\rO)} \hspace{-2pt}+\hspace{-1pt} \ve^2 (v_n^\pprime, \varphi^\pprime)_{L^2(\Gamma)} \vspace{.2cm}\\
\displaystyle{}\ \ - (q_n, \bA^j_i \varphi^i_{,j})_{L^2(\rO)} = (\widebar{\rF}_1,\varphi)_{L^2(\rO)} \hspace{-1pt}+\hspace{-1pt} (\widebar{\rG}_1,\varphi)_{L^2(\Gamma)} \quad \Forall \varphi \in \text{span}(\re_1,\cdots,\re_n)
\end{array}
\end{equation}
and the initial condition
$$
v_n(x,0) = \sum_{k=1}^n (v_0,\re_k)_{L^2(\rO)} \, \re_k(x) \,.
$$
Then for all fixed $\theta >0$,
\begin{alignat*}{2}
v_{n_j} &\to v_\theta &&\text{in}\quad L^2(0,\rT;H^1(\rO))\,,\\
v_{n_j t} &\to v_{\theta t} \qquad&&\text{in}\quad L^2(0,\rT;H^1(\rO)')\,.
\end{alignat*}
However, the difficulties here is due to a $\theta$-independent estimate of $v_{\theta t}$ {\rm(}which is required to pass $\theta \to 0${\rm)}. We remind the readers that the way to obtain an estimate of $v_{\theta t}$ is to use $v_{n t}$ as a test function in {\rm(\ref{weak_soln_ALE_penalty})}, while in this case, for the last term on the left-hand side of {\rm(\ref{weak_soln_ALE_penalty})} we have
\begin{equation}
- \big(q_n, \bA^j_\ell v^\ell_{n t,j}\big)_{L^2(\rO)} = \frac{\theta}{2} \frac{d}{dt} \|q_n\|^2_{L^2(\rO)} + \big(q_n, \bA^j_{\ell t} v^\ell_{n,j}\big)_{L^2(\rO)}\,, \label{penalty_pressure_estimate}
\end{equation}
where $q_n = -\smallexp{$\displaystyle{}\frac{1}{\theta}$} \bA^j_i v^i_{n,j}$.
The appearance of the second term suggests that we are not able to obtain $\theta$-independent estimate of $v_{n t}$ and the penalty method is not applicable.

On the other hand, with the satisfaction of the first order compatibility condition {\rm(\ref{compatibility_condition})}, we use the Lagrangian formulation and look for a solution $(v,q)$ in the space $L^2(0,\rT;H^3(\rO))\times L^2(0,\rT;H^2(\rO))$ with $v_t \in L^2(0,\rT;H^1(\rO))$. In this case, we first use the Galerkin method to obtain a solution to the following integral equality {\rm(}which is the time derivative of the weak formulation {\rm(\ref{weak_soln_Lag_penalty1})} of the solution to {\rm(\ref{weak_soln_ALE_penalty}))}
\begin{equation}\label{weak_soln_Lag_penalty}
\begin{array}{l}
\displaystyle{} \langle w_{\theta t}, \varphi\rangle + \big(\big[\bA^k_\ell (\bA^j_\ell v^i_{\theta,j} + \bA^j_i v^\ell_{\theta,j})\big]_t,\varphi^i_{,k}\big)_{L^2(\rO)} + \ve^2 (w^\pprime_\theta, \varphi^\pprime)_{L^2(\Gamma)} \vspace{.2cm}\\
\displaystyle{} \qquad - \big((\bA^j_i q_{\theta})_t, \varphi^i_{,j}\big)_{L^2(\rO)} = (\widebar{\rF}_{1t},\varphi)_{L^2(\rO)} + (\widebar{\rG}_{1t},\varphi)_{L^2(\Gamma)} \quad\ \ \Forall \varphi\in H^1(\rO)
\end{array}
\end{equation}
with $w_\theta(0) = \widebar{\rF}_1(0) - \nabla q_0 + \Delta u_0$, where $v_\theta = u_0 + \smallexp{$\displaystyle{}\int_0^t$} w_\theta ds$, and
$$
q_\theta = q_0 - \frac{1}{\theta} \bA^j_i v^i_{\theta,j}\,.
$$
Integrating {\rm(\ref{weak_soln_Lag_penalty})} in time we find that $v_\theta$ satisfies
\begin{equation}\label{weak_soln_Lag_penalty1}
\begin{array}{l}
\displaystyle{} \langle v_{\theta t}, \varphi\rangle + \big(\bA^k_\ell (\bA^j_\ell v^i_{\theta,j} + \bA^j_i v^\ell_{\theta,j}),\varphi^i_{,k}\big)_{L^2(\rO)} + \ve^2 (v^\pprime_\theta, \varphi^\pprime)_{L^2(\Gamma)} \vspace{.2cm}\\
\displaystyle{} \qquad - \big(q_{\theta}, \bA^j_i \varphi^i_{,j}\big)_{L^2(\rO)} = (\widebar{\rF}_1,\varphi)_{L^2(\rO)} + (\widebar{\rG}_1,\varphi)_{L^2(\Gamma)} \quad\ \ \Forall \varphi\in \rV\,.
\end{array}
\end{equation}
We remark that to obtain {\rm(\ref{weak_soln_Lag_penalty1})} by integrating {\rm(\ref{weak_soln_Lag_penalty})} in time, the first order compatibility condition {\rm(\ref{compatibility_condition})} is required.

Since the solution $w_\theta = v_{\theta t}$ to {\rm(\ref{weak_soln_Lag_penalty})} belongs to $L^2(0,T;H^1(\rO))$, we may use it as a test function in {\rm(\ref{weak_soln_Lag_penalty1})}. In this case, %we need to look at the term $-(q_\theta, \bA^j_\ell v^\ell_{\theta t,j})_{L^2(\rO)}$, and
for the last term on the left-hand side of {\rm(\ref{weak_soln_Lag_penalty1})} we have
$$
- \big(q_\theta, \bA^j_\ell v^\ell_{\theta t,j}\big)_{L^2(\rO)} = \frac{\theta}{2} \frac{d}{dt} \|q_\theta\|^2_{L^2(\rO)} + \big(q_\theta, \bA^j_{\ell t} v^\ell_{\theta,j}\big)_{L^2(\rO)}\,.
$$
Even though the equality above looks similar to the finite dimensional approximation {\rm(\ref{penalty_pressure_estimate})}, the situation now is different since $q_\theta$ is the Lagrange multiplier for the functional $T:\rV \to \bbR$ given by
\begin{align*}
T(\varphi) &= (v_{\theta t},\varphi)_{L^2(\rO)} + \frac{1}{2} \big((\bA^j_\ell v^i_{\theta,j} + \bA^j_i v^\ell_{\theta,j}) (\bA^k_\ell \varphi^i_{,k} + \bA^k_i \varphi^\ell_{,k})\big)_{L^2(\rO)} \\
&\quad + \ve^2 (v^\pprime_\theta \cdot \varphi^\pprime)_{L^2(\Gamma)} - (\widebar{\rF}_1,\varphi)_{L^2(\rO)} - (\widebar{\rG}_1,\varphi)_{L^2(\Gamma)}\,.
\end{align*}
Therefore, {\rm(}another version of\,{\rm)} the Lagrange multiplier lemma provides an $L^2$-estimate for $q_\theta$:
$$
\|q_\theta\|_{L^2(\rO)} \le C \Big[\|v_{\theta t}\|_{L^2(\rO)} + \|v_\theta\|_{H^1(\rO)} + \ve^2 \|v_\theta\|_{H^1(\Gamma)} + \|\widebar{\rF}_1\|_{L^2(\rO)} + \|\widebar{\rG}_1\|_{L^2(\Gamma)}\Big]\,,
$$
and the estimate above can be used to deduce an $\theta$-independent estimate for $v_{\theta t}$ in $L^2(0,T;L^2(\rO))$ {\rm(}using Young's inequality{\rm)}.
\end{remark}
%\vspace{.1in}

\subsubsection{The existence of a solution to the regularized problem {\rm(\ref{NSreg})}}\label{sec:existence_to_NSreg}
In the previous sub-section we have already established the existence of a unique $(w,q) \in \V(\rT)\times \Q(\rT)$ satisfying
$$
\begin{array}{l}
\displaystyle{} \rB_\bpsi(w,\varphi) - (q,\div \varphi)_{L^2(\rO)} + \ve^2 (w^\prime,\varphi^\pprime)_{L^2(\Gamma)} \vspace{.2cm}\\
\displaystyle{} \qquad\qquad = (\widebar{\rF} \hspace{-1pt}-\hspace{-1pt} w_t,\varphi)_{L^2(\rO)} + (\widebar{\rG},\varphi)_{L^2(\Gamma)}
\end{array}
\qquad \Forall \varphi\in \rV,\ \ \text{a.e. $t\in [0,\rT]$}.
$$
In other words, $w(t)$ is the weak solution to the elliptic equation
\begin{alignat*}{2}
- (a^{jk}_{rs} w^r_{,j})_{,k} + q_{,s} &= f^s \qquad&&\text{in}\quad\rO\,,\\
\div w &= 0 &&\text{in}\quad\rO\,,\\
a^{jk}_{rs} w^r \rN_k - q \rN_s &= \ve^2 \Delta_0 w^s + g^s \qquad&&\text{on}\quad\Gamma\,.
\end{alignat*}
where
\begin{equation}
a^{ik}_{rs} = \bJ^{-1} \bA^j_\ell \bpsi^i_{,s} \bA^k_\ell \bpsi^i_{,r} + \bJ^{-1} \delta^k_r \delta^j_s \label{defn:ajk_rs}
\end{equation}
and
\begin{align*}
f^s &\equiv \widebar{\rF}^s - w^s_t + \big[\bpsi^i_{,s} \bA^k_\ell \bA^j_\ell (\bJ^{-1} \bpsi^i_{,r})_{,j} w^r + \bA^k_\ell (\bJ^{-1} \bpsi^\ell_{,r})_{,s} w^r \big]_{,k} \,,\\
g^s &\equiv \widebar{\rG}^s - \big[\bpsi^i_{,s} \bA^k_\ell \bA^j_\ell (\bJ^{-1} \bpsi^i_{,r})_{,j} w^r + \bA^k_\ell (\bJ^{-1} \bpsi^\ell_{,r})_{,s} w^r \big] \rN_k \,.
\end{align*}
Thanks to the convolution, $\bpsi, \bpsi_t \in L^\infty(0,T;H^k(\rO))$ for all $k\in \rN$. Therefore,
\begin{align*}
\|a\|_{W^{1,\infty}(\rO)} &\le C \Big[1 + \|\nabla^2 \bpsi\|_{L^\infty(\rO)} \Big] \le C \Big[1 + \|\nabla^2 \bpsi\|_{H^{1.5}(\rO)}\Big] \\
&\le C \Big[1 + \|\bh_{\ve\ve}\|_{H^{2.75}(\Gamma)}\Big] \le C \Big[1 + \ve^{-1} \|\bh\|_{H^{1.75}(\Gamma)}\Big]\,.
\end{align*}
Moreover, by (\ref{w_LinfH1_est}) we find that
\begin{align*}
\|g\|_{L^2(\Gamma)} &\le C \Big[1+\|\bh\|_{H^2(\Gamma)} + \|\nabla^2 \bpsi\|_{L^4(\Gamma)} \|w\|_{L^4(\Gamma)}\Big] \\
&\le C \Big[1 + \|\bh\|_{H^2(\Gamma)} + \ve^{-1} \|\bh\|_{H^{1.25}(\Gamma)} \|w\|_{H^1(\rO)}\Big] \\
&\le C \Big[1 + \|h_0\|_{H^2(\Gamma)} + \rM \sqrt{t} + \ve^{-1} \big(\|w_0\|_{H^1(\rO)} + \|h_0\|_{H^{1.5}(\Gamma)} + 1\big) \Big]\,,
\end{align*}
and by (\ref{sup_in_time_ineq}a) we obtain that
\begin{align*}
& \|f\|_{L^2(\rO)} \le C \Big[\big\|\id - \bJ^{-1} (\nabla \bpsi)^\rT (\nabla \bpsi)\big\|_{L^\infty(\rO)} \|\bw_t\|_{L^2(\rO)} + \|w_t\|_{L^2(\rO)} \\
&\qquad\quad + \big(\|\bw\|_{L^4(\rO)} + \|\bpsi_t\|_{L^4(\rO)}\big) \big(\|\nabla \bw\|_{L^4(\rO)} + \|\nabla^2 \bpsi\|_{L^4(\rO)} \|\bw\|_{L^\infty(\rO)} \big) \\
&\qquad\quad + \|\nabla \bpsi_t\|_{L^2(\rO)} \|\bw\|_{L^\infty(\rO)} + \|\nabla^2 \bpsi\|_{L^4(\rO)} \|\nabla \bw\|_{L^4(\rO)} + \|\nabla^2 \bpsi\|^2_{L^4(\rO)} \|\bw\|_{L^\infty(\rO)} \\
&\qquad\quad + \big(\|\nabla^2 \bpsi\|^2_{L^4(\rO)} + \|\nabla^3 \bpsi\|_{L^4(\rO)} \big) \|w\|_{L^2(\rO)} + \|\nabla^2\bpsi\|_{L^2(\rO)} \|\nabla w\|_{L^2(\rO)} \Big] \\
&\qquad\le C \Big[\varsigma \|\bw_t\|_{L^2(\rO)} + \big(1 + \|w_0\|_{H^1(\rO)} + \rM \sqrt{t} \big) \|\bw\|_{H^{1.25}(\rO)} \Big] \\
&\qquad\quad + C \|w_t\|_{L^2(\rO)} + C_\ve \|w\|_{H^1(\rO)}.
\end{align*}
Therefore, Theorem \ref{thm:elliptic_regularity} together with (\ref{w_LinfH1_est}) implies that
\begin{align}
& \int_0^{\rT_\ve} \hspace{-2pt}\Big[\|w\|^2_{H^2(\rO)} + \ve^2 \|w\|^2_{H^2(\Gamma)} + \|q\|^2_{H^1(\rO)}\Big] dt \le C_\ve \int_0^{\rT_\ve} \hspace{-2pt}\|g\|^2_{L^2(\Gamma)} dt \nonumber\\
&\qquad\quad + C \int_0^{\rT_\ve} \hspace{-2pt}\Big[\big(1 + \|a\|^2_{W^{1,\infty}(\rO)} \big)\|w\|^2_{H^1(\rO)} + \|f\|^2_{L^2(\rO)} + \|g\|^2_{H^{-0.5}(\Gamma)} \Big] dt \nonumber\\
&\qquad \le C_\ve \big(1 + \|w_0\|^2_{H^1(\rO)} + \|h_0\|^2_{H^2(\Gamma)} + \rM^2 {\rT_\ve}\big) {\rT_\ve} + C_\ve \int_0^{\rT_\ve} \hspace{-2pt}\|w\|^2_{H^1(\rO)} dt + C \varsigma^2 \rM^2 \nonumber\\
&\qquad\quad + C \big(1 + \|w_0\|^2_{H^1(\rO)} + \rM^2 {\rT_\ve} \big) \int_0^{\rT_\ve} \hspace{-2pt}\|\bw\|^2_{H^{1.25}(\rO)} dt + C \int_0^{\rT_\ve} \hspace{-2pt}\|w_t\|^2_{L^2(\rO)} dt \nonumber\\
&\qquad \le C_3 \Big[\|w_0\|^2_{H^1(\rO)} + \|h_0\|^2_{H^2(\Gamma)} + 1\Big] + C_4 \varsigma^2 \rM^2 \label{temp_main_estimate}
\end{align}
if $\rT_\ve > 0$ is chosen small enough, where $C_3$ and $C_4$ are generic constants.

Once $(w,q)$ is obtained, we define
\begin{equation}\label{h_evol_eq}
h(t) = h_{0\ve} + \int_0^t \frac{w\cdot \rN}{1+b_0 \bh}\, d\tilde{t}\,.
\end{equation}
Then \Holder's inequality and the normal trace estimate imply that
\begin{align*}
& \|h\|^2_{\H_1(\rT_\ve)} \equiv \int_0^{\rT_\ve} \Big[\|h\|^2_{H^2(\Gamma)} + \|h_t\|^2_{H^2(\Gamma)} + \|h_{tt}\|^2_{H^{-0.5}(\Gamma)} \Big] dt \\
&\quad \le \int_0^{\rT_\ve} \hspace{-2pt} \|w_t\|^2_{L^2(\rO)} dt + C {\rT_\ve} \Big[\|h_0\|^2_{H^2(\Gamma)} + \int_0^{\rT_\ve} \hspace{-2pt}\big(\|w\|^2_{H^2(\Gamma)} + \|w\|^2_{L^\infty(\Gamma)} \|\bh\|^2_{H^2(\Gamma)} \big) dt \Big] \\
&\quad \le C_1 \Big[\|w_0\|^2_{H^1(\rO)} + \|h_0\|^2_{H^{1.5}(\Gamma)} + 1 \Big] + C_2 \varsigma^2 \rM^2 \\
&\qquad + C {\rT_\ve} \Big[\|h_0\|^2_{H^2(\Gamma)} + \frac{\rM^2 + 1}{\ve^2} \Big(C_3 \big[\|w_0\|^2_{H^1(\rO)} + \|h_0\|^2_{H^2(\Gamma)} + 1\big] + C_4 \varsigma^2 \rM^2 \Big) \Big]\,.
\end{align*}
%which is bounded by a generic constant if $\rT_\ve >0$ is chosen even smaller.
%
As a consequence, the combination of (\ref{w_LinfH1_est}), (\ref{temp_main_estimate}) and an upper bound for $\|h\|_{\H_1(\rT_\ve)}$ (stated in the estimate above) suggests that
\begin{equation}\label{main_estimate1}
\|w\|^2_{\W(\rT_\ve)} + \|h\|^2_{\H_1(\rT_\ve)} \le C_5 \Big[\|w_0\|^2_{H^1(\rO)} \hspace{-1pt}+\hspace{-1pt} \|h_0\|^2_{H^{1.5}(\Gamma)} \hspace{-1pt}+\hspace{-1pt} 1\Big] \hspace{-1pt}+\hspace{-1pt} C_6 \varsigma^2 \rM^2
\end{equation}
for some constants $C_5 \ge \max\{C_1,C_3\}$ and $C_6 \ge \max\{C_2,C_4\}$, if $\rT_\ve >0$ is chosen even smaller. Let $\varsigma$ be small enough so that $C_6 \varsigma^2 \le \smallexp{$\displaystyle{}\frac{1}{2}$}$, and
$$
\rM^2 = 2 C_5 \Big[\|w_0\|^2_{H^1(\rO)} + \|h_0\|^2_{H^2(\Gamma)} + 1\Big]\,.
$$
Then
\begin{equation}\label{linear_estimate}
\|w\|^2_{\W(\rT_\ve)} + \|h\|^2_{\H_1(\rT_\ve)} \le \rM^2 = 2 C_5 \Big[\|w_0\|^2_{H^1(\rO)} + \|h_0\|^2_{H^2(\Gamma)} + 1\Big]\,.
\end{equation}

At this point we have established a map
$$
\Phi:\left\{\begin{array}{ccc}
C_{\rT_\ve}(\rM) &\to& C_{\rT_\ve}(\rM) \vspace{.1cm}\\
(\bw,\bh) &\mapsto& (w,h)
\end{array}
\right.\,.
$$
Suppose that $(w_n, h_n) \in C_{\rT_\ve}(\rM)$ and $(w_n, h_n) \rightharpoonup (w,h)$ in $\W(\rT_\ve)\times \H_1(\rT_\ve)$. Let $(\psi_n, \psi)$ be the corresponding ALE map with $\rJ_n$, $\rA_n$, $\rJ$, $\rA$ defined accordingly, and $(w^n,h^n) \equiv \Phi(w_n,h_n) \in C_{\rT_\ve}(\rM)$. Moreover, let $(\widebar{\rF}_n,\widebar{\rG}_n)$ and $(\widebar{\rF},\widebar{\rG})$ be the correspond forcing. Due to the boundedness of $(w^n, h^n)$ in $\W(\rT_\ve)\times \H_1(\rT_\ve)$ and the convolution we must have
\begin{subequations}\label{convergence}
\begin{alignat}{2}
(w^{n_\ell}, h^{n_\ell}) &\rightharpoonup (\widetilde{\bfw}, \widetilde{h}) \qquad&&\text{in}\quad \W(\rT_\ve)\times \H_1(\rT_\ve) \,,\\
\psi_{n_\ell t} &\rightharpoonup \psi_t &&\text{in}\quad L^2(0,\rT_\ve;H^2(\rO))\,, \\
(w_{n_\ell},h_{n_\ell}) &\to (w,h) &&\text{in}\quad L^2(0,\rT_\ve;H^{1.75}(\rO)) \times \C([0,\rT_\ve];H^{1.75}(\Gamma)) \,,\\
%h_{n_\ell} &\to h &&\text{in}\quad \C([0,\rT_\ve];H^{1.75}(\Gamma)) \,,\\
(\psi_{n_\ell},\psi_{n_\ell t}) &\to (\psi,\psi_t) &&\text{in}\quad L^\infty(0,\rT_\ve;W^{2,\infty}(\rO))\times  L^\infty(0,\rT_\ve;H^1(\rO))\,,
\end{alignat}
\end{subequations}
for some subsequence $n_\ell$ of $n$. This implies that
$$
(a_{n_\ell})^{jk}_{rs} \equiv (\rJ_{n_\ell})^{-1} (\rA_{n_\ell})^j_\ell (\psi_{n_\ell})^i_{,s} (\rA_{n_\ell})^k_\ell (\psi_{n_\ell})^i_{,r} + (\rJ_{n_\ell})^{-1} \delta^k_r \delta^j_s
$$
converges to $a^{jk}_{rs} \equiv \rJ^{-1} \rA^j_\ell \psi^i_{,s} \rA^k_\ell \psi^i_{,r} + \rJ^{-1} \delta^k_r \delta^j_s$ in $L^\infty(0,\rT_\ve;W^{1,\infty}(\rO))$, and
\begin{subequations}\label{convergence1}
\begin{align}
\widebar{\rF}_{n_\ell} &\rightharpoonup \widebar{\rF} \qquad\text{in}\quad L^2(0,\rT_\ve;L^2(\rO))\,, \\
\widebar{\rG}_{n_\ell} &\rightharpoonup \widebar{\rG} \qquad\text{in}\quad L^2(0,\rT_\ve;L^2(\Gamma))\,.
\end{align}
\end{subequations}
%We note that in the weak convergence of $\widebar{\rF}_{n_\ell}$ we need to take care of the weak convergence of the term $$
%(\psi_{n_\ell})^i_{,s} (\rA_{n_\ell})^j_\ell \big[(\rJ_{n_\ell})^{-1} (\psi_{n_\ell})^\ell_{,r} w_{n_\ell}^r - (\psi_{n_\ell})^\ell_t\big] \big[(\rJ_{n_\ell})^{-1} (\psi_{n_\ell})^i_{,s} w_{n_\ell}^s\big]_{,j}
%$$
%and it is concluded by (\ref{convergence}b) and (\ref{convergence}d).
Since the definition of $\Phi$ suggests that the function $(w^{n_\ell},h^{n_\ell}) \in C_{\rT_\ve}(\rM)$ satisfy the variational identity
\begin{align*}
& \big(w^{n_\ell}_t,\varphi\big)_{L^2(\rO)} + %\big((a_{n_\ell})^{jk}_{rs} (w^{n_\ell})^r_{,j}, \varphi^s_{,k}\big)_{L^2(\rO)}
\rB_{\psi_{n_\ell}}(w^{n_\ell}, \varphi) + \ve^2 \big(w^{n_\ell \prime}, \varphi^\pprime)_{L^2(\Gamma)} \\
&\qquad\quad = (\widebar{\rF}_{n_\ell}, \varphi)_{L^2(\rO)} + (\widebar{\rG}_{n_\ell},\varphi)_{L^2(\Gamma)} \qquad \Forall \varphi \in \rV \cap H^1_\div(\rO)\,,
\end{align*}
by (\ref{convergence}) and (\ref{convergence1}) we find that $(\widetilde{\bfw},\widetilde{h})$ satisfies that
\begin{align*}
& (\widetilde{\bfw}_t,\varphi)_{L^2(\rO)} + \rB_\psi(\widetilde{\bfw},\varphi) + \ve^2 (\widetilde{\bfw}^\pprime,\varphi^\pprime)_{L^2(\Gamma)} \\
&\qquad\quad = (\widebar{\rF}, \varphi)_{L^2(\rO)} + (\widebar{\rG},\varphi)_{L^2(\Gamma)} \qquad \Forall \varphi \in \rV \cap H^1_\div(\rO)\,.
\end{align*}
However, by the uniqueness of the weak solution (to the linearized problem) we find that $(\widetilde{\bfw},\widetilde{h}) = (w,h)$. This implies that $\Phi:C_{\rT_\ve}(\rM) \to C_{\rT_\ve}(\rM)$ is weakly continuous. Therefore, the Tychonoff fixed-point theorem suggests that there exists a fixed-point of $\Phi$ in $C_{\rT_\ve}(\rM)$. Every fixed-point $(w^\ve,h^\ve)$ with associated Lagrange multiplies $q^\ve$ then is a solution to the regularized equation (\ref{NSreg}).

\section{The $\ve$-independent estimates}\label{sec:ve_indep_est}
To avoid confusion and simplify the notation, throughout this section we omit the super-script $\ve$ and denote the strong solution $(v^\ve, w^\ve, q^\ve, h^\ve)$ by $(\bfv,\bfw,\bfq,\bfh)$. Let $\Psi$ be the solution to
\begin{alignat*}{2}
\Delta \Psi &= 0 &&\text{in}\quad\rO\,,\\
\Psi &= e + (\eta_\ve\star \eta_\ve \star \bfh) \rN \qquad&&\text{on}\quad\Gamma\,,
\end{alignat*}
and $\bfJ$, $\rA$ be defined accordingly. If $\bfv = \bfJ^{-1} (\nabla \Psi) \bfw$, then $(\bfv,\bfq,\bfh)$ satisfies
\begin{subequations}\label{NSreg1}
\begin{alignat}{2}
\bfv^i_t \hspace{-1pt}-\hspace{-1pt} \bfA^k_\ell \big[\big(\bfA^j_\ell \bfv^i_{,j} \hspace{-1pt}+\hspace{-1pt} \bfA^j_i \bfv^\ell_{,j} \big)\big]_{,k} \hspace{-1pt}+\hspace{-1pt} \bfA^j_i \bfq_{,j} \hspace{-1pt}&=\hspace{-1pt} \bfA^\ell_j (\Psi^j_t \hspace{-1pt}-\hspace{-1pt} \bfv^j) \bfv^i_{,\ell} \qquad&&\text{in}\ \ \rO \hspace{-1.5pt}\times\hspace{-1.5pt}  (0,\hspace{-1pt}\rT_\ve\hspace{-1pt}),\\
\bfA^j_i \bfv^i_{,j} \hspace{-1pt}&=\hspace{-1pt} 0 &&\text{in}\ \ \rO \hspace{-1.5pt}\times\hspace{-1.5pt}  (0,\hspace{-1pt}\rT_\ve\hspace{-1pt}),\\
\big[\bfA^j_\ell \bfv^i_{,j} \hspace{-1pt}+\hspace{-1pt} \bfA^j_i \bfv^\ell_{,j} \hspace{-1pt}-\hspace{-1pt} \bfq \delta^\ell_i \big] \bfA^k_\ell \rN_k \hspace{-1pt}&=\hspace{-1pt} \L_\ve(\bfh) \bfA^j_i \rN_i \hspace{-1pt}+\hspace{-1pt} \ve^2 \bfA^s_i \Delta_0 \bfw^s %\big(\bfJ \bfA^s_j \bfv^j \big)
\hspace{9.5pt} &&\text{on}\ \ \Gamma \hspace{-1.5pt}\times\hspace{-1.5pt}  (0,\hspace{-1pt}\rT_\ve\hspace{-1pt}),\\
\bfh_t \hspace{-1pt}&=\hspace{-1pt} \frac{\bfJ \bfA^\rT \rN}{1 \hspace{-1pt}+\hspace{-1pt} \rb_0 \bfh_{\ve\ve}} \cdot \bfv &&\text{on}\ \ \Gamma \hspace{-1.5pt}\times\hspace{-1.5pt}  (0,\hspace{-1pt}\rT_\ve\hspace{-1pt}), \label{he_eq}\\
(\bfv,\bfh) \hspace{-1pt}&=\hspace{-1pt} (u_{0\ve}, h_{0\ve}) &&\text{on}\ \ \rO \hspace{-1.5pt}\times\hspace{-1.5pt} \{t \hspace{-1.5pt}=\hspace{-1.5pt} 0\},
\end{alignat}
\end{subequations}
here we recall that %with $\bfg^{\ve} = \big[(1 + \rb_0 \bfh_{\ve\ve})^2 + \bfh^{\prime\hspace{1pt}2}_{\ve\ve}\big]$,
\begin{align*}
\L_\ve(\bfh) &= \frac{(1 + \rb_0 \bfh_{\ve\ve}) \bfh'' - \rb_0(1 + 2 \rb_0 \bfh_{\ve\ve} + \rb_0^2 \bfh^2_{\ve\ve} + 2 \bfh^{\prime\hspace{1pt}2}_{\ve\ve}) - \bfh_{\ve\ve} \bfh'_{\ve\ve} \rb_0^\pprime}{\big[(1 + \rb_0 \bfh_{\ve\ve})^2 + \bfh^{\prime\hspace{1pt}2}_{\ve\ve}\big]^{3/2}} \,.
\end{align*}
If $\varphi\in \rV$ is a test function, then
%\begin{equation}\label{weak_v}
%\begin{array}{l}
%\displaystyle{} (\bfJ \bfv_t,\varphi)_{L^2(\rO)} + \int_\rO \bfJ \big(\bfA^j_\ell \bfv^i_{,j} + \bfA^j_i \bfv^\ell_{,j}\big) \rA^k_\ell \varphi^i_{,k} dx + \ve^2 \int_\Gamma (\bfJ \bfA^j_i \bfv^i)' (\bfJ \bfA^j_k \varphi^k)' dS \vspace{.1cm}\\
%\displaystyle{} \quad = \int_\Gamma \L_\ve(\bfh) \bfJ \bfA^j_i \rN_j \varphi^i dS + (\bfA^\ell_j (\Psi^j_t \hspace{-1pt}-\hspace{-1pt} \bfv^j) \bfv^i_{,\ell}, \varphi^i)_{L^2(\rO)} + \big(\bfq, \bfJ \bfA^j_i \varphi^i_{,j}\big)_{L^2(\rO)} \qquad\text{a.e. $t\in (0,\rT_\ve)$}\,.
%\end{array}
%\end{equation}
\begin{align}
& (\bfJ \bfv_t,\varphi)_{L^2(\rO)} + \int_\rO \bfJ \big(\bfA^j_\ell \bfv^i_{,j} + \bfA^j_i \bfv^\ell_{,j}\big) \rA^k_\ell \varphi^i_{,k} dx \nonumber\\
&\qquad + \ve^2 \int_\Gamma (\bfJ \bfA^j_i \bfv^i)' (\bfJ \bfA^j_k \varphi^k)' dS = \int_\Gamma \L_\ve(\bfh) \bfJ \bfA^j_i \rN_j \varphi^i dS \label{weak_v}\\
&\qquad\qquad + \big(\bfq, \bfJ \bfA^j_i \varphi^i_{,j}\big)_{L^2(\rO)} + (\bfA^\ell_j (\Psi^j_t \hspace{-1pt}-\hspace{-1pt} \bfv^j) \bfv^i_{,\ell}, \varphi^i)_{L^2(\rO)} \qquad\text{a.e. $t\in (0,\rT_\ve)$}\,. \nonumber
\end{align}
%In the following, we use $E(t)$ to denote the energy function
%\begin{align*}
%E(t) &\equiv \max_{s\in [0,t]} \Big[\|\bfv(s)\|^2_{H^1(\rO)} + \|\bfh(s)\|^2_{H^2(\Gamma)} \Big] + \int_0^t \Big[\|\bfv(\tilde{t})\|^2_{H^2(\rO)} + \|\bfv_t(\tilde{t})\|^2_{L^2(\rO)} \\
%&\quad + \|\bfq(\tilde{t})\|^2_{H^1(\rO)} + \|\bfh(\tilde{t})\|^2_{H^{2.5}(\Gamma)} + \|\bfh_t(\tilde{t})\|^2_{H^{1.5}(\Gamma)} \Big] d\tilde{t} \,.
%\end{align*}
In the following, we let $\rM_1$ denote a positive constant (to be determined later) such that
$$
\rM_1^2 \ge \rM^2 = 2 C_5 \Big[\|w_0\|^2_{H^1(\rO)} + \|h_0\|^2_{H^2(\Gamma)} + 1\Big]\,,
$$
and $[0,\rT]$ be the (maximal) time interval in which a solution exists, and will be determined later as well.

\subsection{Key elliptic estimate}
The fundamental reason for that the time of existence $\rT_\ve$ depends on the smoothing parameter $\ve$ is the requirement for extra regularity for $\bpsi$. For example, it requires that $\nabla^2 \bpsi \in L^\infty(0,\rT;L^\infty(\rO))$ in the process of estimating $w_{nt}$ (see estimate (\ref{extra_regularity1})), and we need $a^{jk}_{rs} \in W^{1,\infty}(\rO)$ and $f\in L^2(0,\rT;L^2(\rO))$ to apply Theorem \ref{thm:elliptic_regularity}, while the boundedness of $\nabla^3 \bpsi$ (in some Sobolev spaces) is required in both cases. However, our functional framework (for the linearized problem) only provides us $\bh\in L^\infty(0,\rT;H^2(\Gamma))$ which suggests that $\bpsi \in L^\infty(0,\rT;H^{2.5}(\rO))$, so the requirement of extra regularity has to be provided by the convolution. Therefore, an improvement of the regularity of $\bfh$ is important for obtaining $\ve$-independent estimates.

Before proceeding to the derivation of $\ve$-independent estimates, we state the following lemma, while the proof of this Lemma will be provided in Appendix \ref{app:elliptic}.

\begin{lemma}\label{lem:key_elliptic_estimate}
Let $(\bfv,\bfq,\bfh) \in \W(\rT) \times \Q(\rT) \times \H_1(\rT)$ be a strong solution to {\rm(\ref{NSreg1})}. Then for some generic constant $C$ {\rm(}independent of $\ve${\rm)}
\begin{equation}\label{hH25_estimate}
\int_0^\rT \hspace{-1pt} \|\bfh(t)\|^2_{H^{2.5}(\Gamma)} dt \le C \Big[1 + \|h_0\|^2_{H^{1.5}(\Gamma)} + \|\bfv\|^2_{\V(\rT)} + \|\bfq\|^2_{\Q(\rT)} \Big]\,.
\end{equation}
In particular, the corresponding $\bfJ$, $\bfA$ and $\Psi$ satisfy
\begin{equation}\label{JApsi_H2_estimate}
\begin{array}{l}
\displaystyle{} \int_0^\rT \hspace{-1pt}\Big[\|\bfA\|^2_{H^2(\rO)} + \|\bfJ\|^2_{H^2(\rO)} + \|\nabla \psi\|^2_{H^2(\rO)} \Big] dt \vspace{.1cm}\\
\displaystyle{} \hspace{50pt}\le C \Big[1 + \|h_0\|^2_{H^{1.5}(\Gamma)} + \|\bfv\|^2_{\V(\rT)} + \|\bfq\|^2_{\Q(\rT)}\Big]\,.
\end{array}
\end{equation}
\end{lemma}

\begin{remark}
We emphasize that with the regularity of $\Psi$ given by {\rm(\ref{JApsi_H2_estimate})}, %in Lemma {\rm\ref{lem:key_elliptic_estimate}},
$\bfw$ and $\bfv$ cannot belong to $L^2(0,\rT;H^2(\rO))$ simultaneously: if $\bfv\in L^2(0,\rT;H^2(\rO))$, the identity
$$
D^2 \bfw = D^2 (\bfJ \bfA) \bfv + 2 D (\bfJ \bfA) D \bfv + \bfJ \bfA D^2 \bfv
$$
suggests that $\bfw \not\in L^2(0,\rT;H^2(\rO))$ since $D^2 (\bfJ \bfA) \bfv \not\in L^2(0,\rT;L^2(\rO))$ {\rm(}for it would require that $\bfv \in L^\infty(0,\rT;L^\infty(\rO))${\rm)}. %Similarly, $\bfv$ belongs to $L^2(0,\rT;H^2(\rO))$ if $\bfw \in L^\infty(0,\rT;L^\infty(\rO))$. This observation implies that we cannot close the (energy) estimates at the level $\bfw\in L^2(0,\rT;H^2(\rO))$ if we aim for closing the estimates at the level of $\bfv \in L^2(0,\rT;H^2(\rO))$. Moreover, we expect that the forcing $\rF \in L^2(0,\rT;L^2(\rO))$ to conclude that $w\in L^2(0,\rT;L^2(\rO))$, while $\rF$ fails to belong to this space because of a term of the form $\bfw\nabla^3 \Psi$. These facts are the main reason why we need to change back to the ALE formulation {\rm(\ref{NSreg1})} instead of directly looking for $\ve$-independent estimates using {\rm(\ref{NSreg})}, even though we have improved regularity for $\bfh$.
This is the main reason that we switch back to the ALE formulation {\rm(\ref{NSreg1})} instead of directly looking for $\ve$-independent estimates using {\rm(\ref{NSreg})}, even though we have improved regularity for $\bfh$.
\end{remark}

%\subsection{The estimate for $\bfv$ in $L^2(0,\rT;H^1(\rO))$}
%Let us start from the most basic estimate. Using $\bfv$ as a test function in (\ref{weak_v}), we find that
%\begin{align*}
%& \frac{1}{2} \frac{d}{dt} \|\sqrt{\bfJ} \bfv\|^2_{L^2(\rO)} + \frac{1}{2} \big\|\sqrt{\bfJ}\big[\nabla \bfv \bfA + \bfA^\rT (\nabla \bfv)^{\rT}\big]\big\|^2_{L^2(\rO)} \vspace{.1cm}\\
%& \qquad \le \frac{1}{2} \|\bfJ_t\|_{L^2(\rO)} \|\bfv\|^2_{L^4(\rO)} + C \big(\|\bfh\|_{H^2(\Gamma)} + 1\big) \|\bfv\|_{L^2(\Gamma)} \\
%& \qquad \le C \|\bfv\|^{3/2}_{H^1(\rO)} \|\bfv\|^{1/2}_{L^2(\rO)} + C \big(\|\bfh\|_{H^2(\Gamma)} + 1\big) \|\bfv\|_{H^{0.5}(\Gamma)} \\
%& \qquad \le C_\delta \Big[\|\bfv\|^2_{L^2(\rO)} + \|\bfh\|^2_{H^2(\Gamma)} + 1\Big] + \delta \|\nabla \bfv\|^2_{H^1(\rO)} \,,
%\end{align*}
%where we use the trace estimates and Young's inequality to conclude the last inequality. Integrating the equality above in time over the time interval $(0,t)$ and choosing $\delta > 0$ small enough, by the Gronwall inequality we obtain that
%\begin{equation}
%\|\bfv(t)\|^2_{L^2(\rO)} + \int_0^t \|\nabla \bfv\|^2_{L^2(\rO)} d\tilde{t} \le C \Big[\|u_0\|^2_{L^2(\rO)} + \int_0^t \big[\|\bfh\|^2_{H^2(\Gamma)} + 1\big] d\tilde{t} \Big] \,. \label{bfv_L2H1_est}
%\end{equation}

\subsection{The estimate for $\bfv_t$ in $L^2(0,\rT;L^2(\rO))$}\label{sec:vt_ve_indep_est}
As suggested in Section \ref{sec:Galerkin}, estimating $\bfv_t\in L^2(0,\rT;L^2(\rO))$ requires that we use $\bfv_t$ as a test function in (\ref{weak_v}). Since $\bfv_t$ does not belongs to $H^1(\rO)$, it cannot be used as a test function in (\ref{weak_v}). To resolve this issue, we adopt the approach of difference quotient. We note that since $\bfv \in \C([0,\rT];H^1(\rO))$, $D_{\Delta t} \bfv(t) \equiv \smallexp{$\displaystyle{} \frac{\bfv(t+\Delta t) - \bfv(t)}{\Delta t}$}$ can be used as a test function in (\ref{weak_v}) for all $t \in [0,\rT]$. By doing so, for $\tilde{t} = t$ and $\tilde{t} = t+\Delta t$ for almost all $t\in (0,\rT)$ and $\Delta t>0$,
\begin{align}
& \Big(\bfJ(\tilde{t}) \bfv_t(\tilde{t}), D_{\Delta t} \bfv(t)\Big)_{L^2(\rO)} + \int_\rO \bfJ(\tilde{t}) \Big[\bfA^j_\ell(\tilde{t}) \bfv^i_{,j}(\tilde{t}) + \bfA^j_i(\tilde{t}) \bfv^\ell_{,j}(\tilde{t})\Big] \rA^k_\ell(\tilde{t}) \big[D_{\Delta t} \bfv(t)\big]^i_{,k} \, dx \nonumber\\
&\qquad + \ve^2 \int_\Gamma (\bfJ(\tilde{t}) \bfA^j_i(\tilde{t}) \bfv^i(\tilde{t}))' \Big[\bfJ(\tilde{t}) \bfA^j_k(\tilde{t}) D_{\Delta t} \bfv^k(t)\Big]' dS \label{difference_quotient}\\
&\quad = \int_\Gamma \hspace{-2pt}\L_\ve(\bfh(\tilde{t})) \bfJ(\tilde{t}) \bfA^j_i(\tilde{t}) \rN_j D_{\Delta t} \bfv^i(t) dS + \Big(\bfq(\tilde{t}), \bfJ(\tilde{t}) \bfA^j_i(\tilde{t}) \big[D_{\Delta t} \bfv(t)\big]^i_{,j} \Big)_{L^2(\rO)} \nonumber \\
&\qquad + \Big(\bfJ(\tilde{t}) \bfA^\ell_j(\tilde{t}) (\Psi^j_t(\tilde{t}) \hspace{-1pt}-\hspace{-1pt} \bfv^j(\tilde{t})) \bfv^i_{,\ell}(\tilde{t}), D_{\Delta t} v^i(t) \Big)_{L^2(\rO)}\,. \nonumber
\end{align}
%\begin{equation}\label{difference_quotient}
%\begin{array}{l}
%\displaystyle{} \Big(\bfJ(\tilde{t}) \bfv_t(\tilde{t}), D_{\Delta t} \bfv(t)\Big)_{L^2(\rO)} + \int_\rO \bfJ(\tilde{t}) \Big[\bfA^j_\ell(\tilde{t}) \bfv^i_{,j}(\tilde{t}) + \bfA^j_i(\tilde{t}) \bfv^\ell_{,j}(\tilde{t})\Big] \rA^k_\ell(\tilde{t}) \big[D_{\Delta t} \bfv(t)\big]^i_{,k} \, dx \vspace{.1cm}\\
%\displaystyle{} \qquad + \ve^2 \int_\Gamma (\bfJ(\tilde{t}) \bfA^j_i(\tilde{t}) \bfv^i(\tilde{t}))' \Big[\bfJ(\tilde{t}) \bfA^j_k(\tilde{t}) D_{\Delta t} \bfv^k(t)\Big]' dS \vspace{.1cm}\\
%\displaystyle{} \quad = \int_\Gamma \hspace{-2pt}\L_\ve(\bfh(\tilde{t})) \bfJ(\tilde{t}) \bfA^j_i(\tilde{t}) \rN_j D_{\Delta t} \bfv^i(t) dS \hspace{-2pt}+\hspace{-2pt} \Big(\bfJ \bfA^\ell_j (\Psi^j_t \hspace{-1pt}-\hspace{-1pt} \bfv^j) \bfv^i_{,\ell}, D_{\Delta t} v^i(t) \Big)_{L^2(\rO)} \vspace{.1cm}\\
%\displaystyle{} \qquad + \Big(\bfq(\tilde{t}), \bfJ(\tilde{t}) \bfA^j_i(\tilde{t}) \big[D_{\Delta t} \bfv(t)\big]^i_{,j} \Big)_{L^2(\rO)}\,.
%\end{array}
%\end{equation}
Summing (\ref{difference_quotient}) over $\tilde{t} = t$ and $\tilde{t} = t+\Delta t$, integrating the sum over the time interval $(a,b) \subseteq [0,\rT]$, and then passing $\Delta t \to 0$ will then result in the desired estimates.

To be more precise, we note that
\begin{align*}
& \int_\rO \hspace{-2pt}\bfJ(t\hspace{-1pt}+\hspace{-1pt}\Delta t) \Big[\bfA^j_\ell(t\hspace{-1pt}+\hspace{-1pt}\Delta t) \bfv^i_{,j}(t\hspace{-1pt}+\hspace{-1pt}\Delta t) \hspace{-1pt}+\hspace{-1pt} \bfA^j_i(t\hspace{-1pt}+\hspace{-1pt}\Delta t) \bfv^\ell_{,j}(t\hspace{-1pt}+\hspace{-1pt}\Delta t)\Big] \rA^k_\ell(t\hspace{-1pt}+\hspace{-1pt}\Delta t) \big[D_{\Delta t} \bfv(t) \big]^i_{,k} \, dx \\
&\qquad + \int_\rO \bfJ(t) \Big[\bfA^j_\ell(t) \bfv^i_{,j}(t) + \bfA^j_i(t) \bfv^\ell_{,j}(t)\Big] \rA^k_\ell(t) \big[D_{\Delta t} \bfv(t)\big]^i_{,k} \, dx \\
&\qquad\qquad = \frac{1}{\Delta t} \int_\rO \bfJ(\tilde{t}) \Big[\bfA^j_\ell(\tilde{t}) \bfv^i_{,j}(\tilde{t}) + \bfA^j_i(\tilde{t}) \bfv^\ell_{,j}(\tilde{t})\Big] \rA^k_\ell(\tilde{t}) \bfv^i_{,k} (\tilde{t}) dx \Big|_{\,\tilde{t}=t}^{\,\tilde{t}=t+\Delta t} \\
&\qquad\qquad\quad - \int_\rO \frac{(\bfJ \bfA^k_\ell \bfA^j_\ell)(t+\Delta) - (\bfJ \bfA^k_\ell \bfA^j_\ell)(t)}{\Delta t} \bfv^i_{,j}(t) \bfv^i_{,k}(t+\Delta t)\, dx
\end{align*}
and due to the fact that $\bfA^j_i \bfv^i_{,j} = 0$\,,
\begin{align*}
& \big(\bfq(t+\Delta t), \bfJ(t+\Delta t) \bfA^j_i(t+\Delta t) D_{\Delta t} \bfv(t) \big)_{L^2(\rO)} \hspace{-1pt}+ \big(\bfq(t), \bfJ(t) \bfA^j_i(t) \big[D_{\Delta t} \bfv(t)\big]^i_{,j} \big)_{L^2(\rO)} \\
%&\qquad = - \frac{1}{\Delta t} \Big(\bfq(t+\Delta t), \bfJ(t+\Delta t) \bfA^j_i(t+\Delta t) \bfv^i_{,j}(t)\Big)_{L^2(\rO)} \\
%&\qquad\quad + \frac{1}{\Delta t} \Big(\bfq(t), \bfJ(t) \bfA^j_i(t) \bfv^i_{,j}(t+\Delta t) \Big)_{L^2(\rO)} \\
&\qquad = - \Big(\bfq(t+\Delta t), \frac{\bfJ(t+\Delta t) \bfA^j_i(t+\Delta t) - \bfJ(t) \bfA^j_i(t)}{\Delta t}\, \bfv^i_{,j}(t) \Big)_{L^2(\rO)} \\
&\qquad\quad - \Big(\bfq(t), \frac{\bfJ(t+\Delta t) \bfA^j_i(t+\Delta t) - \bfJ(t) \bfA^j_i(t)}{\Delta t}\, \bfv^i_{,j}(t+\Delta t) \Big)_{L^2(\rO)} \,.
\end{align*}
By the fact that
\begin{alignat*}{2}
f(\cdot+\Delta t) &\to f(\cdot) \qquad&&\text{in}\quad L^2(a,b;X) \quad \text{if \ $f\in L^2(0,\rT;X)$}\,, \\
\frac{f(\cdot+\Delta t) - f(\cdot)}{\Delta t} &\rightharpoonup f_t(\cdot) \qquad&&\text{in}\quad L^2(a,b;X) \quad\text{if \ $f_t\in L^2(0,\rT;X)$}\,,
\end{alignat*}
if $[a,b]\subseteq (0,\rT)$, we obtain that
\begin{align*}
& \lim_{\Delta t\to 0} \int_a^b \Big[\int_\rO \bfJ(t \hspace{-1pt}+\hspace{-1pt} \Delta t) \big(\bfA^j_\ell(t \hspace{-1pt}+\hspace{-1pt} \Delta t) \bfv^i_{,j}(t \hspace{-1pt}+\hspace{-1pt} \Delta t) \hspace{-1pt}+\hspace{-1pt} \bfA^j_i(t \hspace{-1pt}+\hspace{-1pt} \Delta t) \bfv^\ell_{,j}(t \hspace{-1pt}+\hspace{-1pt} \Delta t)\big) \times \\
&\hspace{56pt}\times \hspace{-1pt} \rA^k_\ell(t \hspace{-1pt}+\hspace{-1pt} \Delta t) \big[D_{\Delta t} \bfv(t)\big]^i_{,k} \\
&\hspace{56pt} +\hspace{-1pt} \bfJ(t) \big(\bfA^j_\ell(t) \bfv^i_{,j}(t) \hspace{-1pt}+\hspace{-1pt} \bfA^j_i(t) \bfv^\ell_{,j}(t)\big) \rA^k_\ell(t) \big[D_{\Delta t} \bfv(t)\big]^i_{,k} \, dx\Big] dt \\
&\qquad = \int_a^b \frac{1}{2} \frac{d}{dt} \big\|\sqrt{\bfJ}\,\big[(\nabla \bfv) \bfA \hspace{-1pt}+\hspace{-1pt} \bfA^\rT (\nabla \bfv)^{\rT}\big]\big\|^2_{L^2(\rO)} dt - \int_a^b \int_\rO (\bfJ \bfA^k_\ell \bfA^j_\ell)_t \bfv^i_{,j} \bfv^i_{,k} dx dt
\end{align*}
and
\begin{align*}
& \lim_{\Delta t \to 0} \int_a^b \Big[\hspace{-1pt}\Big(\bfq(t\hspace{-1pt}+\hspace{-1pt}\Delta t), \bfJ(t\hspace{-1pt}+\hspace{-1pt}\Delta t) \bfA^j_i(t\hspace{-1pt}+\hspace{-1pt}\Delta t) \big[D_{\Delta t} \bfv(t)\big]^i_{,j}\Big)_{L^2(\rO)} \\
&\hspace{37pt} + \hspace{-2pt}\Big(\bfq(t), \bfJ(t) \bfA^j_i(t) \big[D_{\Delta t} \bfv(t)\big]^i_{,j}\Big)_{L^2(\rO)}\Big] dt \hspace{-1pt}=\hspace{-1pt} -\hspace{-1pt} 2 \int_a^b \hspace{-2pt} \big(\bfq, (\bfJ \rA^j_i)_t \bfv^i_{,j}\big)_{L^2(\rO)} dt\,.
\end{align*}
Moreover,
\begin{align*}
&\lim_{\Delta t \to 0} \int_a^b \int_\Gamma \Big[\big(\L_\ve(\bfh) \bfJ \bfA^j_i\big) (t+\Delta t) + \big(\L_\ve(\bfh) \bfJ \bfA^j_i\big)(t)\Big] \rN_j D_{\Delta t} \bfv^i(t) dx dt \\
%&\qquad = 2 \int_a^b \big\langle \L_\ve(\bfh), \bfJ \bfA^j_i \rN_j \bfv_t \big\rangle dt \\
&\qquad = 2 \int_a^b \big\langle \L_\ve(\bfh), \big(\bfJ \bfA^j_i \rN_j \bfv)_t \big\rangle dt - 2 \int_a^b \int_\Gamma \L_\ve(\bfh) \big(\bfJ \bfA^j_i)_t \rN_j \bfv\, dS dt \\
&\qquad \le 2 \int_a^b \big\langle \L_\ve(\bfh), \big(\bfJ \bfA^j_i \rN_j \bfv)_t \big\rangle dt + C \int_a^b \big(\|\bfh\|_{H^2(\Gamma)} + 1\big) \|\bfh_t\|_{W^{1,4}(\Gamma)} \|\bfv\|_{L^4(\Gamma)} dt
\end{align*}
and by Young's inequality,
\begin{align*}
&\lim_{\Delta t \to 0} \int_a^b \int_\Gamma \Big[\big(\bfJ \bfA^j_i \Delta_0 (\bfJ \bfA^j_\ell \bfv^\ell)\big)(t+\Delta t) + \big(\bfJ \bfA^j_i \Delta_0 (\bfJ \bfA^j_\ell \bfv^\ell)\big)(t)\Big] D_{\Delta t} \bfv^i(t) dS dt \\
&\quad = 2 \int_a^b \int_\Gamma (\bfJ \bfA^j_i \bfv^i_t) \Delta_0 (\bfJ \bfA^j_\ell \bfv^\ell) dS = 2 \int_a^b \int_\Gamma \big[(\bfJ \bfA^j_i \bfv^i)_t - (\bfJ \bfA^j_i)_t \bfv^i \big]\Delta_0 (\bfJ \bfA^j_\ell \bfv^\ell) dS \\
&\quad \le 2 \int_a^b \int_\Gamma \big[(\bfJ \bfA^j_i \bfv^i)_t \Delta_0 (\bfJ \bfA^j_\ell \bfv^\ell) dS dt + \frac{C_{\delta_1}}{\ve^2} \int_a^b \|\nabla \Psi_t\|^2_{L^4(\Gamma)} \|\bfv\|^2_{L^4(\Gamma)} dt \\
&\qquad + \delta_1 \ve^2 \int_a^b \|\Delta_0 (\bfJ \bfA \bfv)\|^2_{L^2(\Gamma)} dt \,.
\end{align*}
As a consequence, summing (\ref{difference_quotient}) over $\tilde{t} = t$ and $\tilde{t} = t+\Delta t$ and then integrating in $t$ over the time interval $(a,b)$ and then passing $\Delta t \to 0$, we find that
\begin{align}
& \frac{1}{2} \big\|\sqrt{\bfJ}\big[\nabla \bfv \bfA + \bfA^\rT (\nabla \bfv)^{\rT}\big]\big\|^2_{L^2(\rO)} \Big|_{t=a}^{t=b} + \int_a^b \|\sqrt{\bfJ} \bfv_t\|^2_{L^2(\rO)} dt \nonumber\\
&\qquad \le 2 \int_a^b \big\langle \L_\ve(\bfh), \big(\bfJ \bfA^j_i \rN_j \bfv^i)_t \big\rangle dt + 2 \ve^2 \int_a^b \int_\Gamma \big[(\bfJ \bfA^j_i \bfv^i)_t \Delta_0 (\bfJ \bfA^j_\ell \bfv^\ell) dS dt \nonumber\\
&\qquad\quad + C \int_a^b \big(\|\bfh\|_{H^2(\Gamma)} + 1\big) \|\bfh_t\|_{W^{1,4}(\Gamma)} \|\bfv\|_{L^4(\Gamma)} dt \label{bvt_est_temp_ineq}\\
&\qquad\quad + C \int_a^b \|\nabla \Psi_t\|_{L^2(\rO)} \Big[\|\nabla \bfv\|^2_{L^4(\rO)} + \|\bfq\|_{L^4(\rO)} \|\nabla \bfv\|_{L^4(\rO)} \Big] dt \nonumber\\
&\qquad\quad + C_{\delta_1} \int_a^b \|\nabla \Psi_t\|^2_{L^4(\Gamma)} \|\bfv\|^2_{L^4(\Gamma)} dt + \delta_1 \ve^4 \int_a^b \|\Delta_0 (\bfJ \bfA \bfv)\|^2_{L^2(\Gamma)} dt\,. \nonumber
\end{align}
We note that
\begin{align*}
\|\nabla \Psi_t\|_{L^4(\Gamma)} &\le C \|\bfh_t\|_{H^{1.25}(\Gamma)} \le C \big[\|\bfJ \bfA\|_{H^{1.25}(\Gamma)} \|\bfv\|_{H^{0.7}(\Gamma)} + \|\bfJ \bfA\|_{H^{0.7}(\Gamma)} \|\bfv\|_{H^{1.25}(\Gamma)} \big] \\
&\le C \big[\|\bfh\|_{H^{2.25}(\Gamma)} \|\bfv\|_{H^{1.2}(\rO)} + \|\bfv\|_{H^{1.75}(\rO)} \big]\,,
\end{align*}
by %the boundedness of $\bfv\in L^\infty(0,\rT_\ve;L^2(\rO))$ and
Young's inequality we obtain that
\begin{align*}
\|\nabla \Psi_t\|^2_{L^4(\Gamma)} \|\bfv\|^2_{L^4(\Gamma)} &\le C \big[\|\bfh\|^2_{H^{2.25}(\Gamma)} \|\bfv\|^2_{H^{1.2}(\rO)} \hspace{-2pt}+\hspace{-1pt} \|\bfv\|^2_{H^{1.75}(\rO)} \big] \|\bfv\|^2_{H^{0.75}(\rO)} \\
&\le C \big[\|\bfh\|^{11/8}_{H^{1.7}(\Gamma)} \|\bfh\|^{5/8}_{H^{2.5}(\Gamma)} \|\bfv\|^{18/5}_{H^1(\rO)} \|\bfv\|^{2/5}_{H^2(\rO)} \hspace{-2pt}+\hspace{-1pt} \|\bfv\|^{5/2}_{H^1(\rO)} \|\bfv\|^{3/2}_{H^2(\rO)} \big] \\
&\le C_\delta \big[\|\bfv\|^4_{H^1(\rO)} \hspace{-2pt}+\hspace{-1pt} \|\bfv\|^{10}_{H^1(\rO)}\big] \hspace{-2pt}+\hspace{-1pt} \delta \big[\|\bfv\|^2_{H^2(\rO)} \hspace{-2pt}+\hspace{-1pt} \|\bfh\|^2_{H^{2.5}(\Gamma)} \big] \,.
\end{align*}
Since (\ref{bvt_est_temp_ineq}) holds for all $0<a<b<\rT$ and $\bfv \in \C([0,\rT];H^1(\rO))$, passing $a\to 0$ and $b\to t$ for some $t \le \rT$, by (\ref{hH25_estimate}) estimate (\ref{bvt_est_temp_ineq}) implies that
\begin{align*}
& \frac{1}{2} \big\|\sqrt{\bfJ}\big[\nabla \bfv \bfA \hspace{-2pt}+\hspace{-1pt} \bfA^\rT (\nabla \bfv)^{\rT}\big]\big\|^2_{L^2(\rO)} \hspace{-2pt}+\hspace{-1pt} \int_0^t \|\sqrt{\bfJ} \bfv_t\|^2_{L^2(\rO)} d\tilde{t} \le C \big[1 \hspace{-2pt}+\hspace{-1pt} \|u_0\|^2_{H^1(\rO)} \hspace{-2pt}+\hspace{-1pt} \|h_0\|^2_{H^{1.5}(\Gamma)}\big] \\
&\qquad \hspace{-2pt}+\hspace{-1pt} 2 \int_0^t \big\langle \L_\ve(\bfh), \big(\bfJ \bfA^j_i \rN_j \bfv)_t \big\rangle\, d\tilde{t} \hspace{-2pt}+\hspace{-1pt} 2 \ve^2 \int_0^t \int_\Gamma \big[(\bfJ \bfA^j_i \bfv^i)_t \Delta_0 (\bfJ \bfA^j_\ell \bfv^\ell) dS d\tilde{t} \\
&\qquad \hspace{-2pt}+\hspace{-1pt} C_{\delta,\delta_1} \int_0^t \Big[\big(1 \hspace{-2pt}+\hspace{-1pt} \|\bfh\|^2_{H^2(\Gamma)}\big) \|\bfv\|^2_{H^1(\rO)} \hspace{-2pt}+\hspace{-1pt} \|\bfv\|^4_{H^1(\rO)} \hspace{-2pt}+\hspace{-1pt} \|\bfv\|^{10}_{H^1(\rO)} \Big]d\tilde{t} \\
&\qquad \hspace{-2pt}+\hspace{-1pt} \delta \int_0^t \Big[\|\bfv\|^2_{H^2(\rO)} \hspace{-2pt}+\hspace{-1pt} \|\bfq\|^2_{H^1(\rO)} \Big] d\tilde{t} \hspace{-2pt}+\hspace{-1pt} \delta_1 \ve^4 \int_0^t \|\Delta_0 (\bfJ \bfA \bfv)\|^2_{L^2(\Gamma)} d\tilde{t} \,.
\end{align*}
Integrating by parts in time,
\begin{align*}
& \int_0^t \hspace{-2pt}\big\langle \L_\ve(\bfh), \big(\bfJ \bfA^j_i \rN_j \bfv)_t \big\rangle\, d\tilde{t} = \big\langle \L_\ve(\bfh), \bfJ \bfA^j_i \rN_j \bfv \big\rangle\Big|_{\,\tilde{t}=0}^{\,\tilde{t}=t} - \int_0^t \hspace{-2pt}\big\langle \big(\L_\ve(\bfh)\big)_t, \bfJ \bfA^j_i \rN_j \bfv \big\rangle\, d\tilde{t} \\
%&\quad \le C \big(\|\bfh(t)\|_{H^2(\Gamma)} + 1\big) \|\bfv(t)\|_{L^2(\Gamma)} + C \int_0^t \|\bfh_t\|_{H^1(\Gamma)} \|\bfv\|_{L^2(\Gamma)} d\tilde{t} \\
%&\qquad + \int_0^t \Big[\frac{\bfh^\pprime}{\big[(1 + \rb_0 \bfh_{\ve\ve})^2 + \bfh^{\prime\hspace{1pt}2}_{\ve\ve}\big]^{1/2}}\Big]_t \bfh^\pprime_t d\tilde{t} \\
&\quad \le C_{\delta,\delta_2} \Big[1 \hspace{-1pt}+\hspace{-1pt} \|\bfv\|^2_{L^2(\rO)} \hspace{-2pt}+\hspace{-1pt} \int_0^t \hspace{-2pt}\|\bfv\|^2_{H^1(\rO)} d\tilde{t} \Big] \hspace{-1pt}+\hspace{-1pt} \delta \Big[\hspace{-1pt}\|\bfh\|^2_{H^2(\Gamma)} \hspace{-2pt}+\hspace{-2pt} \int_0^t \hspace{-2pt}\|\bfv\|^2_{H^2(\rO)} d\tilde{t}\Big] \hspace{-1.5pt}+\hspace{-1pt} \delta_2 \|\bfv\|^2_{H^1(\rO)},
\end{align*}
while on the other hand,
\begin{align*}
\int_0^t \hspace{-1pt}\int_\Gamma \hspace{-1pt}\big[(\bfJ \bfA^j_i \bfv^i)_t \Delta_0 (\bfJ \bfA^j_\ell \bfv^\ell) dS d\tilde{t} \hspace{-1pt}=\hspace{-1pt} - \frac{1}{2} \int_0^t \hspace{-1pt}\frac{d}{dt} \big\|(\bfJ \bfA \bfv)'\big\|^2_{L^2(\Gamma)} d\tilde{t} \hspace{-1pt}=\hspace{-1pt} - \frac{1}{2} \big\|(\bfJ \bfA \bfv)'\big\|^2_{L^2(\Gamma)}\Big|_{\,\tilde{t}=0}^{\,\tilde{t} = t}\,.
\end{align*}
Moreover,
\begin{align*}
\|\Delta_0 (\bfJ \bfA \bfv)\|_{L^2(\Gamma)} &\le C \Big[\|\bfh_\ve\|_{H^3(\Gamma)} \|\bfv\|_{L^\infty(\Gamma)} + \|\bfh_\ve\|_{H^{2.25}(\Gamma)} \|\bfv\|_{H^{1.25}(\Gamma)} + \|\bfv\|_{H^2(\Gamma)} \Big] \\
%&\le C \Big[\frac{1}{\ve} \|\bfh\|_{H^2(\Gamma)} \|\bfv\|^{1/2}_{H^{0.25}(\Gamma)} \|\bfv\|^{1/2}_{H^{0.75}(\Gamma)} + \frac{1}{\ve^{1/2}} \|\bfv\|_{H^2(\rO)} + \|\bfv\|_{H^2(\Gamma)} \Big] \\
%&\le C \Big[\frac{1}{\ve} \|\bfh\|_{H^2(\Gamma)} \|\bfv\|^{1/2}_{H^{0.75}(\rO)} \|\bfv\|^{1/2}_{H^{1.25}(\rO)} + \frac{1}{\ve^{1/2}} \|\bfv\|_{H^2(\rO)} + \|\bfv\|_{H^2(\Gamma)} \Big] \\
&\le C \Big[\frac{1}{\ve} \|\bfh\|_{H^2(\Gamma)} \|\bfv\|^{1/2}_{L^2(\rO)} \|\bfv\|^{1/2}_{H^2(\rO)} + \frac{1}{\ve^{1/2}} \|\bfv\|_{H^2(\rO)} + \|\bfv\|_{H^2(\Gamma)} \Big]
\end{align*}
and
\begin{align*}
\ve^2 \|(\bfJ \bfA \bfv)'\|^2_{L^2(\Gamma)} &\ge \frac{\ve^2}{c} \|\bfv^\pprime\|^2_{L^2(\Gamma)} - C \ve^2 \|\bfh_\ve\|^2_{H^{2.25}(\Gamma)} \|\bfv\|^2_{H^{0.25}(\Gamma)} \\
&\ge \frac{\ve^2}{c} \|\bfv^\pprime\|^2_{L^2(\Gamma)} - C_{\delta_2} \|\bfv\|^2_{L^2(\rO)} - \delta_2 \|\bfv\|^2_{H^1(\rO)} \,.
\end{align*}
Therefore, by $\ve^2 \|u_{0\ve}\|^2_{H^1(\Gamma)} \le C \|u_0\|^2_{H^1(\rO)}$, if $\ve < 1$,
\begin{equation}\label{bfvt_LinfL2_est_temp}
\begin{array}{l}
\displaystyle{} \big\|\Def \bfv(t)\big\|^2_{L^2(\rO)} \hspace{-2pt}+\hspace{-1pt} \ve^2 \big\|\bfv^\pprime(t)\big\|^2_{L^2(\Gamma)} \hspace{-2pt}+\hspace{-1pt} \int_0^t \|\bfv_t\|^2_{L^2(\rO)} d\tilde{t} \vspace{.1cm}\\
\displaystyle{} \quad \le C \Big[1 \hspace{-2pt}+\hspace{-1pt} \|u_0\|^2_{H^1(\rO)} \hspace{-2pt}+\hspace{-1pt} \|h_0\|^2_{H^{1.5}(\Gamma)} \Big] \hspace{-2pt}+\hspace{-1pt} C_{\delta,\delta_1} \int_0^t \big(1 \hspace{-2pt}+\hspace{-1pt} \|\bfh\|^4_{H^2(\Gamma)} \big) \|\bfv\|^2_{H^1(\rO)} d\tilde{t} \vspace{.1cm}\\
\displaystyle{} \qquad \hspace{-2pt}+\hspace{-1pt} C_{\delta,\delta_2} \big(1 \hspace{-2pt}+\hspace{-1pt} \|\bfv\|^2_{L^2(\rO)}\big) \hspace{-2pt}+\hspace{-1pt} \delta_2 \|\bfv(t)\|^2_{H^1(\rO)} \hspace{-2pt}+\hspace{-1pt} C_{\delta,\delta_1} \int_0^t \big[\|\bfv\|^4_{H^1(\rO)} \hspace{-2pt}+\hspace{-1pt} \|\bfv\|^{10}_{H^1(\rO)} \big] d\tilde{t} \vspace{.1cm}\\
\displaystyle{} \qquad \hspace{-2pt}+\hspace{-1pt} \delta \int_0^t \Big[\|\bfv\|^2_{H^2(\rO)} \hspace{-2pt}+\hspace{-1pt} \|\bfq\|^2_{H^1(\rO)} \Big] d\tilde{t} \hspace{-2pt}+\hspace{-1pt} \delta_1 \Big[\|\bfh(t)\|^2_{H^2(\Gamma)} \hspace{-2pt}+\hspace{-1pt} \ve^4 \int_0^t \|\bfv\|^2_{H^2(\Gamma)} d\tilde{t}\Big]\,.
\end{array}
\end{equation}
On the other hand, the use of $\bfv$ as a test function in (\ref{weak_v}) implies that
\begin{align*}
& \frac{1}{2} \frac{d}{dt} \|\sqrt{\bfJ} \bfv\|^2_{L^2(\rO)} + \frac{1}{2} \big\|\sqrt{\bfJ}\big[\nabla \bfv \bfA + \bfA^\rT (\nabla \bfv)^{\rT}\big]\big\|^2_{L^2(\rO)} \le C \big(\|\bfh\|_{H^2(\Gamma)} + 1\big) \|\bfv\|_{L^2(\Gamma)} \\
%& \qquad \le C \big(\|\bfh\|_{H^2(\Gamma)} + 1\big) \|\bfv\|_{H^{0.25}(\Gamma)} \\
& \qquad\quad \le C_{\delta_3} \big(1 + \|\bfv\|^2_{L^2(\rO)} \big) + \delta_3 \|\nabla \bfv\|^2_{L^2(\rO)} + C \|\bfh\|^2_{H^2(\Gamma)} \,,
\end{align*}
where we use the trace estimates and Young's inequality to conclude the last inequality. Integrating the equality above in time over the time interval $(0,t)$ and choosing $\delta_3 > 0$ small enough, by the Gronwall inequality we obtain that
\begin{equation}
\|\bfv(t)\|^2_{L^2(\rO)} + \int_0^t \|\nabla \bfv\|^2_{L^2(\rO)} d\tilde{t} \le C \Big[\|u_0\|^2_{L^2(\rO)} + \int_0^t \big(\|\bfh\|^2_{H^2(\Gamma)} + 1\big) d\tilde{t} \Big] \,. \label{bfv_L2H1_est}
\end{equation}
Combining estimates (\ref{bfvt_LinfL2_est_temp}) and (\ref{bfv_L2H1_est}), by Korn's inequality and choosing $\delta_2$ small enough we conclude that
\begin{align}
& \big\|\bfv(t)\big\|^2_{H^1(\rO)} + \ve^2 \big\|\bfv(t)\big\|^2_{H^1(\Gamma)} + \int_0^t \|\bfv_t\|^2_{L^2(\rO)} d\tilde{t} \le C_\delta \big[1 + \|u_0\|^2_{H^1(\rO)} + \|h_0\|^2_{H^{1.5}(\Gamma)} \big] \nonumber\\
&\qquad\quad + C_{\delta,\delta_1} \int_0^t \Big[\big(\|\bfh\|^4_{H^2(\Gamma)} + 1\big) \big(1+\|\bfv\|^2_{H^1(\rO)}\big) + \|\bfv\|^{10}_{H^1(\rO)} \Big] d\tilde{t} \label{bfvt_L2L2_est}\\
&\qquad\quad + \delta \int_0^t \Big[\|\bfv\|^2_{H^2(\rO)} + \|\bfq\|^2_{H^1(\rO)} \Big] d\tilde{t} + \delta_1 \Big[\|\bfh(t)\|^2_{H^2(\Gamma)} + \ve^4 \int_0^t \|\bfv\|^2_{H^2(\Gamma)} d\tilde{t} \Big]\,. \nonumber
\end{align}

\subsection{The estimate for $\bfv \in L^2(0,\rT;H^{1.5}(\Gamma))$}\label{sec:v_ve_indep_est}
%To apply Corollary \ref{cor:elliptic_regularity}, we need to obtain an $\ve$-independent estimate of $\bfw$ in the Sobolev space $L^2(0,\rT;H^{1.5}(\Gamma))$.
%
Let $\{\chi_m\}_{m=1}^K$ be cut-off functions supported near the boundary $\Gamma$ so that there exists $\theta_m: B(0,r_m) \to \supp(\chi_m)$ satisfying
\begin{enumerate}
\item $\theta_m(B_+(0,r_m)) = \supp(\chi_m) \cap \rO$\,, where $B_+(0,r_m) \equiv B(0,r_m)\cap \{y_2 > 0\}$,
\item $\theta_m(B(0,r_m)\cap \{y_2 = 0\}) = \supp(\chi_m) \cap \Gamma$\,;
\item $\det(\nabla \theta_m) = 1$.
\end{enumerate}
For a fixed $m$, define $\wfa = (\nabla \theta_m)^{-1}$, $\widetilde{\chi} = \chi_m\circ\theta_m$ in $\supp(\widetilde{\chi}_m)$, $(\widetilde{\bfw}, \widetilde{q}) = (\bfw, \bfq)\circ \theta_m$ in $B_+(0,r_m)$. Our goal is to use
\begin{equation}\label{test_function}
\varphi = \big[\Lambda_\eps (\widetilde{\chi}^2 \Lambda_\eps \wfv_{,1})_{,1}\big] \circ \theta_m^{-1}\,,
\end{equation}
as a test function in (\ref{weak_v}) (which is legitimate since $\varphi \in L^2(0,\rT;H^1(\rO))$, and then pass $\eps \to 0$.

\subsubsection{The estimate of $(\bfJ \bfv_t, \varphi)_{L^2(\rO)}$}
Let $\wfJ = \bfJ \circ \theta_m$. Then
%Making a change of variable,
\begin{align*}
& (\bfJ \bfv_t, \varphi)_{L^2(\rO)} = \big(\wfJ \wfv_t,\Lambda_\eps (\widetilde{\chi}^2 (\Lambda_\eps \wfv_{,1})_{,1}\big)_{L^2(B_+(0,r_m))} \\
&\qquad = \big(\Lambda_\eps (\wfJ \wfv_t), (\widetilde{\chi}^2 \Lambda_\eps \wfv_{,1})_{,1}\big)_{L^2(B_+(0,r_m))} \\
&\qquad = \big(\wfJ \Lambda_\eps \wfv_t, (\widetilde{\chi}^2 \Lambda_\eps \wfv_{,1})_{,1}\big)_{L^2(B_+(0,r_m))} + \big(\comm{\Lambda_\eps}{\wfJ} \wfv_t, (\widetilde{\chi}^2 \Lambda_\eps \wfv_{,1})_{,1}\big)_{L^2(B_+(0,r_m))}\\
&\qquad = - \big(\wfJ (\widetilde{\chi} \Lambda_\eps \wfv_{,1})_t, \widetilde{\chi} \Lambda_\eps \wfv_{,1}\big)_{L^2(B_+(0,r_m))} - \big(\wfJ_{,1} \Lambda_\eps \wfv_t, \widetilde{\chi}^2 \Lambda_\eps \wfv_{,1}\big)_{L^2(B_+(0,r_m))} \\
&\qquad\quad + \big(\comm{\Lambda_\eps}{\wfJ} \wfv_t, (\widetilde{\chi}^2 \Lambda_\eps \wfv_{,1})_{,1}\big)_{L^2(B_+(0,r_m))} \,.
\end{align*}
By (\ref{comm_est_temp1}) and the properties of convolution,
\begin{align*}
& \big(\comm{\Lambda_\eps}{\wfJ} \wfv_t, (\widetilde{\chi}^2 \Lambda_\eps \wfv_{,1})_{,1}\big)_{L^2(B_+(0,r_m))} \\
&\qquad \ge - C \eps \|\nabla \bfJ\|_{L^\infty(\rO)} \|\bfv_t\|_{L^2(\rO)} \|\bfv\|_{H^2(\rO)} \ge - \frac{C \eps}{\ve} \|\bfv_t\|_{L^2(\rO)} \|\bfv\|_{H^2(\rO)}\,;
\end{align*}
thus
\begin{align*}
& \big(\Lambda_\eps (\wfJ \wfv_{t,1}), \widetilde{\chi}^2 \Lambda_\eps \wfv_{,1}\big)_{L^2(B_+(0,r_m))} \ge - \frac{1}{2} \frac{d}{dt} \|\sqrt{\wfJ} \widetilde{\chi} \Lambda_\eps \wfv_{,1}\|^2_{L^2(B_+(0,r_m))} \\
&\qquad -\hspace{-1pt} \frac{1}{2} \big(\wfJ_t \widetilde{\chi} \Lambda_\eps \wfv_{,1}, \widetilde{\chi} \Lambda_\eps \wfv_{,1}\big)_{L^2(B_+(0,r_m))} \hspace{-2pt}-\hspace{-1pt} \frac{C \eps}{\ve} \|\bfv_t\|_{L^2(\rO)} \|\bfv\|_{H^2(\rO)} \\
&\quad \ge - \frac{1}{2} \frac{d}{dt} \|\sqrt{\wfJ} \widetilde{\chi} \Lambda_\eps \wfv_{,1}\|^2_{L^2(B_+(0,r_m))} \hspace{-2pt}-\hspace{-1pt} C \|\bfJ_t\|_{L^2(\rO)} \|\nabla \bfv\|^2_{L^4(\rO)} \hspace{-2pt}-\hspace{-1pt} \frac{C \eps}{\ve} \|\bfv_t\|_{L^2(\rO)} \|\bfv\|_{H^2(\rO)} \,.
\end{align*}
Therefore, since $\Lambda_\eps \bfv \to \bfv$ in $C([0,\rT];H^1(\rO))$, by interpolation we conclude that
\begin{equation}\label{vt_term_est}
\begin{array}{l}
\displaystyle{} \lim_{\eps\to 0} \int_0^t \big(\Lambda_\eps (\wfJ \wfv_{t,1}), \widetilde{\chi}^2 \Lambda_\eps \wfv_{,1}\big)_{L^2(B_+(0,r_m))} d\tilde{t} \ge - \frac{1}{2} \|\sqrt{\wfJ} \widetilde{\chi} \wfv_{,1}\|^2_{L^2(B_+(0,r_m))} \vspace{.1cm}\\
\displaystyle{}\qquad - C \|u_0\|^2_{H^1(\rO)} - C_\delta \int_0^t \|\bfv\|^4_{H^1(\rO)} d\tilde{t} - \delta \int_0^t \|\bfv\|^2_{H^2(\rO)} d\tilde{t}\,.
\end{array}
\end{equation}

\subsubsection{The estimate of \smallexp{$\displaystyle{}\int_\rO$}$\bfJ (\bfA^j_\ell \bfv^i_{,j} + \bfA^j_i \bfv^\ell_{,j}) \bfA^k_\ell \varphi^i_{,k} dx$}
Let $\wfA = \bfA \circ \theta_m$. Then
\begin{align*}
& \int_\rO \hspace{-2pt}\bfJ (\bfA^j_\ell \bfv^i_{,j} \hspace{-2pt}+\hspace{-1pt} \bfA^j_i \bfv^\ell_{,j}) \bfA^k_\ell \varphi^i_{,k} dx = \int_{B_+(0,r_m)} \hspace{-2pt}\wfJ \big(\wfA^j_\ell \wfa^r_j \wfv^i_{,r} \hspace{-2pt}+\hspace{-1pt} \wfA^j_i \wfa^r_j \wfv^\ell_{,r}\big) \wfA^k_\ell \wfa^s_k \Lambda_\eps (\widetilde{\chi}^2 \Lambda_\eps \wfv^i_{,1})_{,1s} dy \\
&\qquad\quad = - \int_{B_+(0,r_m)} \hspace{-2pt}\wfJ \big(\wfA^j_\ell \wfa^r_j \wfv^i_{,1r} \hspace{-2pt}+\hspace{-1pt} \wfA^j_i \wfa^r_j \wfv^\ell_{,1r}\big) \wfA^k_\ell \wfa^s_k \Lambda_\eps (\widetilde{\chi}^2 \Lambda_\eps \wfv^i_{,1s})\, dy + \R_1
\end{align*}
with some error term $\R_1$ satisfying
$$
\R_1 \ge - C \Big[\|\nabla^2 \Psi\|_{L^{2^+}(\rO)} \|\nabla \bfv\|_{L^{\infty^-}(\rO)} + \|\nabla \bfv\|_{L^2(\rO)} \Big] \|\nabla \bfv\|_{H^1(\rO)} \,,
$$
where $2^+$ is a number close to but greater than $2$ so that the Sobolev embedding $H^{0.2}(\rO) \contsubset L^{2^+}(\rO)$ holds, and $\infty^-$ is the corresponding \Holder's conjugate satisfying \smallexp{$\displaystyle{}\frac{1}{2^+} + \frac{1}{\infty^-} = \frac{1}{2}$}. Using the notation of commutators,
\begin{align*}
& \int_{B_+(0,r_m)} \hspace{-2pt}\wfJ \big(\wfA^j_\ell \wfa^r_j \wfv^i_{,1r} \hspace{-2pt}+\hspace{-1pt} \wfA^j_i \wfa^r_j \wfv^\ell_{,1r}\big) \wfA^k_\ell \wfa^s_k \Lambda_\eps (\widetilde{\chi}^2 \Lambda_\eps \wfv^i_{,1s})\, dy \\
&\quad = \int_{B_+(0,r_m)} \hspace{-2pt}\wfJ \big[\wfA^j_\ell \wfa^r_j \Lambda_\eps (\widetilde{\chi} \wfv^i_{,1r}) \hspace{-2pt}+\hspace{-1pt} \wfA^j_i \wfa^r_j (\widetilde{\chi} \Lambda_\eps \wfv^\ell_{,1r}) \big] \wfA^k_\ell \wfa^s_k  (\widetilde{\chi} \Lambda_\eps \wfv^i_{,1s})\, dy \\
&\qquad + \int_{B_+(0,r_m)} \hspace{-2pt} \Big[\big(\comm{\Lambda_\eps}{\wfJ \wfA^j_\ell \wfA^k_\ell \wfa^r_j \wfa^s_k} \wfv^\ell_{,1r}\big) + \big(\comm{\Lambda_\eps}{\wfJ \wfA^j_i \wfA^k_\ell \wfa^r_j \wfa^s_k} \wfv^\ell_{,1r}\big) \Big] \widetilde{\chi}^2 \Lambda_\eps \wfv^i_{,1s} dy\,;
\end{align*}
thus by (\ref{comm_est_temp1}),
\begin{align*}
& \int_{B_+(0,r_m)} \hspace{-2pt}\wfJ \wfa^r_j \big(\wfA^j_\ell \wfv^i_{1,r} \hspace{-2pt}+\hspace{-1pt} \wfA^j_i \wfv^\ell_{,1r}\big) \wfA^k_\ell \wfa^s_k \Lambda_\eps (\widetilde{\chi}^2 \Lambda_\eps \wfv^i_{,1s})\, dy \\
&\quad \le \int_{B_+(0,r_m)} \hspace{-2pt}\wfJ \big[\wfA^j_\ell \wfa^r_j \Lambda_\eps (\widetilde{\chi} \wfv^i_{,1r}) \hspace{-2pt}+\hspace{-1pt} \wfA^j_i \wfa^r_j (\widetilde{\chi} \Lambda_\eps \wfv^\ell_{,1r}) \big] \wfA^k_\ell \wfa^s_k  (\widetilde{\chi} \Lambda_\eps \wfv^i_{,1s})\, dy \\
&\qquad + C \eps \|\bfJ \bfA^j_\ell \bfA^k_\ell\|_{W^{1,\infty}(\rO)} \|\nabla^2 \bfv\|^2_{L^2(\rO)} \\
&\quad \le \int_{B_+(0,r_m)} \hspace{-2pt}\wfJ \big[\wfA^j_\ell \wfa^r_j \Lambda_\eps (\widetilde{\chi} \wfv^i_{,1r}) \hspace{-2pt}+\hspace{-1pt} \wfA^j_i \wfa^r_j (\widetilde{\chi} \Lambda_\eps \wfv^\ell_{,1r}) \big] \wfA^k_\ell \wfa^s_k  (\widetilde{\chi} \Lambda_\eps \wfv^i_{,1s})\, dy \hspace{-1pt}+\hspace{-1.5pt} \frac{C \eps}{\ve} \|\nabla^2 \bfv\|^2_{L^2(\rO)}.
\end{align*}
Therefore, since $\Lambda_\eps (\widetilde{\chi} \wfv^i_{,1r}) \to \widetilde{\chi} \wfv^i_{,1r}$ in $L^2(0,\rT;L^2(B_+(0,r_m)))$, by interpolation and Young's inequality we conclude that
\begin{align}
& \lim_{\eps\to 0} \int_0^t \int_\rO \hspace{-2pt}\bfJ (\bfA^j_\ell \bfv^i_{,j} \hspace{-2pt}+\hspace{-1pt} \bfA^j_i \bfv^\ell_{,j}) \bfA^k_\ell \varphi^i_{,k} dx d\tilde{t} \nonumber\\
&\quad \ge \int_0^t \hspace{-2pt}\int_{B_+(0,r_m)} \hspace{-2pt}\wfJ \big[\wfA^j_\ell \wfa^r_j \Lambda_\eps (\widetilde{\chi} \wfv^i_{,1r}) \hspace{-2pt}+\hspace{-1pt} \wfA^j_i \wfa^r_j (\widetilde{\chi} \Lambda_\eps \wfv^\ell_{,1r}) \big] \wfA^k_\ell \wfa^s_k  (\widetilde{\chi} \Lambda_\eps \wfv^i_{,1s})\, dy d\tilde{t} \nonumber\\
&\qquad - C \int_0^t \Big[\|\nabla^2 \Psi\|_{L^{2^+}(\rO)} \|\nabla \bfv\|_{L^{\infty^-}(\rO)} + \|\nabla \bfv\|_{L^2(\rO)} \Big] \|\nabla \bfv\|_{H^1(\rO)} d\tilde{t} \nonumber\\
&\quad \ge \frac{1}{2} \int_0^t \hspace{-2pt}\big\|\Def (\widetilde{\chi} \wfv_{,1})\big\|^2_{B_+(0,r_m)} d\tilde{t} \hspace{-1pt}-\hspace{-1pt} C_\delta \int_0^t \hspace{-2pt}\|\bfv\|^2_{H^1(\rO)} d\tilde{t} \hspace{-1pt}-\hspace{-1pt} (C \varsigma \hspace{-1pt}+\hspace{-1pt} \delta) \int_0^t \hspace{-2pt}\|\bfv\|^2_{H^2(\rO)} d\tilde{t} \,. \label{laplacian_term_est}
\end{align}

\subsubsection{The estimate of $(\bfq, \bfJ \bfA^j_i \varphi^i_{,j})_{L^2(\rO)}$}
Let $\widetilde{\bfq} = \bfq \circ \theta_m$. Making a change of variable and integrating by parts in $y_1$, we obtain that
\begin{align*}
& (\bfq, \bfJ \bfA^j_i \varphi^i_{,j})_{L^2(\rO)} = \big(\widetilde{\bfq}, \wfJ \wfA^j_i \wfa_j^\ell \Lambda_\eps (\widetilde{\chi}^2 \Lambda_\eps \wfv^i_{,1})_{,1\ell}\big)_{L^2(B_+(0,r_m))} \\
&\qquad\quad = - \big(\widetilde{\bfq}, \wfJ \wfA^j_i (\wfa_j^\ell)_{,1} \Lambda_\eps (\widetilde{\chi}^2 \Lambda_\eps \wfv^i_{,1})_{,\ell}\big)_{L^2(B_+(0,r_m))} \\
&\qquad\qquad - \big(\widetilde{\bfq}, (\wfJ \wfA^j_i)_{,1} \wfa_j^\ell \Lambda_\eps (\widetilde{\chi}^2 \Lambda_\eps \wfv^i_{,1})_{,\ell}\big)_{L^2(B_+(0,r_m))} \\
&\qquad\qquad - \big(\widetilde{\bfq}_{,1}, \wfJ \wfA^j_i \wfa_j^\ell \Lambda_\eps (\widetilde{\chi}^2 \Lambda_\eps \wfv^i_{,1})_{,\ell}\big)_{L^2(B_+(0,r_m))} \,.
\end{align*}
For the first term, we can integrate by parts in $y_1$ again and obtain that
\begin{align*}
& - \big(\widetilde{\bfq}, \wfJ \wfA^j_i (\wfa_j^\ell)_{,1} \Lambda_\eps (\widetilde{\chi}^2 \Lambda_\eps \wfv^i_{,1})_{,\ell}\big)_{L^2(B_+(0,r_m))} \\
&\qquad \le C \Big[\|\bfq\|_{L^{\infty^-}(\rO)} \|\nabla^2 \Psi\|_{L^{2^+}(\rO)} + \|\bfq\|_{H^1(\rO)} \Big] \|\bfv\|_{H^1(\rO)} \le C \|\bfq\|_{H^1(\rO)} \|\bfv\|_{H^1(\rO)} \,.
\end{align*}
Similarly, it is easy to see that the second term satisfy that
\begin{align*}
& - \big(\widetilde{\bfq}, (\wfJ \wfA^j_i)_{,1} \wfa_j^\ell \Lambda_\eps (\widetilde{\chi}^2 \Lambda_\eps \wfv^i_{,1})_{,\ell}\big)_{L^2(B_+(0,r_m))} \\
&\qquad\quad \le C \|\widetilde{\bfq}\|_{L^{\infty^-}(\rO)} \|\nabla^2 \Psi\|_{L^{2^+}(\rO)} \|\bfv\|_{H^2(\rO)} \le C \varsigma \|\bfq\|_{H^1(\rO)} \|\bfv\|_{H^2(\rO)}\,.
\end{align*}
For the last term, since $0 = (\bfA^j_i \bfv^i_{,j})\circ \theta_m = \wfA^j_i \wfa^\ell_j \wfv^i_{,\ell}$, we find that
$$
\wfJ \wfA^j_i \wfa^\ell_j \wfv^i_{,1\ell} = \big(\wfJ \wfA^j_i \wfa^\ell_j \wfv^i_{,\ell}\big)_{,1} - \big(\wfJ \wfA^j_i \wfa^\ell_j)_{,1} \wfv^i_{,\ell} = - \big(\wfJ \wfA^j_i \wfa^\ell_j)_{,1} \wfv^i_{,\ell} \,;
$$
thus by the convergence of $\Lambda_\eps \big(\widetilde{\chi}^2 \Lambda_\eps \wfv^i_{,1}\big)_{,\ell} \to \big(\widetilde{\chi}^2 \wfv^i_{,1}\big)_{,\ell}$ in $L^2(0,\rT;L^2(B_+(0,r_m)))$,
\begin{align*}
& - \lim_{\eps\to 0} \int_0^t \big(\widetilde{\bfq}_{,1}, \wfJ \wfA^j_i \wfa_j^\ell \Lambda_\eps (\widetilde{\chi}^2 \Lambda_\eps \wfv^i_{,1})_{,\ell}\big)_{L^2(B_+(0,r_m))} d\tilde{t} \\
&\qquad = - \int_0^t \big(\widetilde{\bfq}_{,1}, \wfJ \wfA^j_i \wfa_j^\ell (\widetilde{\chi}^2 \wfv^i_{,1})_{,\ell}\big)_{L^2(B_+(0,r_m))} d\tilde{t} \\
&\qquad \le C \Big[\|\bfq\|_{L^{\infty^-}(\rO)} \|\nabla^2 \Psi\|_{L^{2^+}(\rO)} + \|q\|_{H^1(\rO)} \Big] \|\bfv\|_{H^1(\rO)} \le C \|\bfq\|_{H^1(\rO)} \|\bfv\|_{H^1(\rO)}\,.
\end{align*}
As a consequence, by Young's inequality we conclude that
\begin{equation}\label{pressure_term_est}
\begin{array}{l}
\displaystyle{} \lim_{\eps\to 0} \int_0^t (\bfq, \bfJ \bfA^j_i \varphi^i_{,j})_{L^2(\rO)} dt \vspace{.1cm}\\
\displaystyle{}\qquad\qquad \le C_\delta \int_0^t \|\bfv\|^2_{H^1(\rO)} d\tilde{t} + (C \varsigma + \delta) \int_0^t \Big[\|\bfv\|^2_{H^2(\rO)} + \|\bfq\|^2_{H^1(\rO)}\Big] d\tilde{t}\,.
\end{array}
\end{equation}

\subsubsection{The estimate of \smallexp{$\displaystyle{} \big(\bfJ(\tilde{t}) \bfA^\ell_j(\tilde{t}) (\Psi^j_t(\tilde{t}) - \bfv^j(\tilde{t})) \bfv^i_{,\ell}(\tilde{t}), \varphi^i \big)_{L^2(\rO)}$}}
Since
$$
\|\Psi_t\|_{L^6(\rO)} \le C \|\bfh_t\|_{H^{1/6}(\Gamma)} \le C \|\bfv\|_{H^{1/6}(\Gamma)} \le C \|\bfv\|_{H^{2/3}(\rO)},
$$
by interpolation and Young's inequality it is easy to see that
\begin{align}
%\lim_{\epsilon \to 0}
& \big(\bfJ(\tilde{t}) \bfA^\ell_j(\tilde{t}) (\Psi^j_t(\tilde{t}) - \bfv^j(\tilde{t})) \bfv^i_{,\ell}(\tilde{t}), \varphi^i \big)_{L^2(\rO)} \nonumber\\
&\qquad \le C \big(\|\Psi_t\|_{L^6(\rO)} + \|\bfv\|_{L^6(\rO)}\big) \|\bfv\|_{W^{1,3}(\rO)} \|\bfv\|_{H^2(\rO)} \nonumber \\
%&\qquad \le C \|\bfv\|^{1/2}_{L^2(\rO)} \|\bfv\|_{H^1(\rO)} \|\bfv\|^{3/2}_{H^2(\rO)} \nonumber\\
&\qquad \le C_\delta \|\bfv\|^2_{L^2(\rO)} \|\bfv\|^4_{H^1(\rO)} + \delta \|\bfv\|^2_{H^2(\rO)} \,. \label{convection_term_est}
\end{align}

\subsubsection{The estimate of \smallexp{$\displaystyle{}\int_\Gamma \L_\ve(\bfh) (\bfJ \bfA^j_i \varphi^i \rN_j) dS $}}\label{sec:boundary_term}
Let $\bfg = (1 \hspace{-1pt}+\hspace{-1pt} \rb_0 \bfh_{\ve\ve})^2 + \bfh^{\prime\hspace{1pt}2}_{\ve\ve}$. %First of all, since
%\begin{align*}
%\frac{1}{2} \frac{d}{dt} \|\bfh - \bfh_0\|^2_{H^{1.75}(\Gamma)} &= (\bfh_t, \bfh - \bfh_0)_{H^{1.75}(\Gamma)} \le C \|\bfh_t\|_{H^{1.5}(\Gamma)} \|\bfh - \bfh_0\|_{H^2(\Gamma)} \\
%&\le C \big[\|h_0\|_{H^2(\rO)} + \rM\big] \|\bfh_t\|_{H^{1.5}(\Gamma)}\,,
%\end{align*}
%we find that
%\begin{align*}
%\|\bfh(t) - \bfh_0\|^2_{H^{1.75}(\Gamma)} \le C \big[\|h_0\|_{H^2(\rO)} + \rM\big] \int_0^t \|\bfh_t\|_{H^{1.5}(\Gamma)} \le C \sqrt{t} \big[\|h_0\|_{H^2(\rO)} + \rM\big] \rM\,.
%\end{align*}
%Therefore, if $0\le t \le \rT_\rM$,
%\begin{align*}
%\|\bfh(t)\|^2_{H^{1.75}(\Gamma)} \le C \Big[\|h_0\|^2_{H^{1.75}(\Gamma)} + \sqrt{t} \big[\|h_0\|_{H^2(\rO)} + \rM\big] \rM \Big] \le C (\varsigma^2 + \varsigma\big) \le C \varsigma
%\end{align*}
%which further suggests that
We note that since $\bfh$ satisfies (\ref{smallness_of_bfh}),
\begin{align*}
\|\bfg(t) - 1\|_{H^{0.7}(\Gamma)} \le C \varsigma \qquad\Forall t\in [0,\rT_{\rM_1}]\,.
\end{align*}
By writting
$
\L_\ve(\bfh) = \smallexp{$\displaystyle{}\frac{\bfh''}{1 \hspace{-1pt}+\hspace{-1pt} \rb_0 \bfh_{\ve\ve}}$} + \R\,,
$
where by (\ref{duality_ineq}) $\R \equiv \L_\ve(\bfh) - \smallexp{$\displaystyle{}\frac{\bfh''}{1 \hspace{-1pt}+\hspace{-1pt} \rb_0 \bfh_{\ve\ve}}$}$ satisfies
\begin{align*}
& \|\R\|_{H^{0.5}(\Gamma)} \le C \Big[\Big\|\smallexp{$\displaystyle{}\frac{(1 \hspace{-1pt}+\hspace{-1pt} \rb_0 \bfh)^2}{\bfg^{3/2}}$} - 1\Big\|_{H^{0.625}(\Gamma)} \|\bfh\|_{H^{2.5}(\Gamma)} + 1 \Big] \le C \Big[\sqrt{\varsigma}\, \|\bfh\|_{H^{2.5}(\Gamma)} + 1\Big]\,,
\end{align*}
we find that
\begin{align*}
& \int_\Gamma \L_\ve(\bfh) (\bfJ \bfA^j_i \varphi^i \rN_j)\, dS \\
&\qquad = \int_\Gamma \bfh'' \frac{\bfJ \bfA^\rT \rN}{1 \hspace{-1pt}+\hspace{-1pt} \rb_0 \bfh_{\ve\ve}}\, \varphi^i dS + \int_\Gamma \Big[\L_\ve(\bfh) - \frac{\bfh''}{1 \hspace{-1pt}+\hspace{-1pt} \rb_0 \bfh_{\ve\ve}}\Big] (\bfJ \bfA^\rT \rN)\cdot \varphi dS \\
&\qquad \ge \int_\Gamma \bfh'' \frac{\bfJ \bfA^\rT \rN}{1 \hspace{-1pt}+\hspace{-1pt} \rb_0 \bfh_{\ve\ve}}\, \varphi^i dS - C \Big[\sqrt{\varsigma}\, \|\bfh\|_{H^{2.5}(\Gamma)} + 1 \Big] \|\bfv\|_{H^{1.5}(\Gamma)} \,.
\end{align*}
On the other hand, since $\varphi \to (\chi^2 \wfv_{,1})_{,1}$ in $L^2(0,\rT;L^2(B(0,r_m)\cap \{y_2 = 0\}))$,
\begin{align*}
& \lim_{\eps \to 0} \int_0^t \int_\Gamma \frac{\bfh''}{1 \hspace{-1pt}+\hspace{-1pt} \rb_0\bfh_{\ve\ve}} (\bfJ \bfA^j_i \varphi^i \rN_j)\, dS d\tilde{t} \\
&\quad \ge \int_0^t \int_\Gamma \widetilde{\chi}^2 \bfh'' \Big[\bfh''_t \hspace{-1pt}-\hspace{-1pt} \Big(\frac{- \bfh'_{\ve\ve} (\rX'\circ\rX^{-1})}{1 \hspace{-1pt}+\hspace{-1pt} \rb_0 \bfh_{\ve\ve}} \hspace{-1pt}+\hspace{-1pt} \rN \Big)'' \bfv \hspace{-1pt}-\hspace{-1pt} 2 \Big(\frac{- \bfh'_{\ve\ve} (\rX'\circ\rX^{-1})}{1 \hspace{-1pt}+\hspace{-1pt} \rb_0 \bfh_{\ve\ve}} \hspace{-1pt}+\hspace{-1pt} \rN \Big)' \bfv^\pprime \Big] dS d\tilde{t} \\
&\qquad - C \int_0^t \|\bfh''\|_{L^2(\Gamma)} \|\bfv^\pprime\|_{L^2(\Gamma)} d\tilde{t} \\
&\quad \ge \frac{1}{2} \big[\|\widetilde{\chi} \bfh''\|^2_{L^2(\Gamma)} \hspace{-1pt}-\hspace{-1pt} \|h_0\|^2_{H^2(\Gamma)}\big] \hspace{-1pt}-\hspace{-1pt} \int_0^t \int_\Gamma \frac{(\rX^\pprime \hspace{-2pt}\circ\hspace{-1pt} \rX^{-1})\cdot \bfv}{1 \hspace{-1pt}+\hspace{-1pt} \rb_0 \bfh_{\ve\ve}} \bfh'''_{\ve\ve} \bfh'' dS d\tilde{t} \\
&\qquad \hspace{-1pt}-\hspace{-1pt} C \int_0^t \big(\|\bfh''\|^2_{L^4(\Gamma)} \hspace{-1pt}+\hspace{-1pt} \|\bfh''\|_{L^2(\Gamma)} \big) \|\bfv\|_{H^1(\Gamma)} d\tilde{t} \,.
\end{align*}
By commutator estimate (\ref{comm_est_temp1}),
\begin{align*}
& \int_\Gamma \frac{(\rX^\pprime \hspace{-2pt}\circ\hspace{-1pt} \rX^{-1})\cdot \bfv}{1 \hspace{-1pt}+\hspace{-1pt} \rb_0 \bfh_{\ve\ve}} \bfh'''_{\ve\ve} \bfh'' dS \\
&\qquad = - \int_\Gamma \Big(\bigcomm{\eta_\ve\star}{\frac{(\rX^\pprime \hspace{-2pt}\circ\hspace{-1pt} \rX^{-1})\cdot \bfv}{1 \hspace{-1pt}+\hspace{-1pt} \rb_0 \bfh_{\ve\ve}}} \bfh'''_\ve \Big) \bfh'' dS - \int_\Gamma \frac{(\rX^\pprime \hspace{-2pt}\circ\hspace{-1pt} \rX^{-1})\cdot \bfv}{1 \hspace{-1pt}+\hspace{-1pt} \rb_0 \bfh_{\ve\ve}} \bfh'''_\ve \bfh''_\ve\, dS \\
&\qquad \le \frac{1}{2} \int_\Gamma \Big[\frac{(\rX^\pprime \hspace{-2pt}\circ\hspace{-1pt} \rX^{-1})\cdot \bfv}{1 \hspace{-1pt}+\hspace{-1pt} \rb_0 \bfh_{\ve\ve}}\Big]' |\bfh''_\ve|^2 dS + C \ve \|\bfv\|_{W^{1,\infty}(\Gamma)} \|\bfh'''_\ve\|_{L^2(\Gamma)} \|\bfh''\|_{L^2(\Gamma)}\,;
\end{align*}
thus by interpolation and Young's inequality,
\begin{align*}
& \int_\Gamma \hspace{-2pt}\frac{(\rX^\pprime \hspace{-2pt}\circ\hspace{-1pt} \rX^{-1})\cdot \bfv}{1 \hspace{-1pt}+\hspace{-1pt} \rb_0 \bfh_{\ve\ve}} \bfh'''_{\ve\ve} \bfh'' dS \\
&\quad \le C \|\bfv\|_{H^1(\Gamma)} \|\bfh''\|^2_{L^4(\Gamma)} + C \sqrt{\ve} \|\bfv\|^{1/2}_{H^1(\Gamma)} \|\bfv\|^{1/2}_{H^2(\Gamma)} \|\bfh\|_{H^{2.5}(\Gamma)} \|\bfh\|_{H^2(\Gamma)} \\
&\quad \le C \|\bfv\|_{H^1(\Gamma)} \|\bfh\|_{H^2(\Gamma)} \|\bfh\|_{H^{2.5}(\Gamma)} + C_{\delta_1} \|\bfv\|^{2/3}_{H^1(\Gamma)} \|\bfh\|^{4/3}_{H^{2.5}(\Gamma)} \|\bfh\|^{4/3}_{H^2(\Gamma)} + \delta_1 \ve^2 \|\bfv\|^2_{H^2(\Gamma)} \\
&\quad \le C_\delta \Big[\|\bfh\|^2_{H^2(\Gamma)} + \|\bfh\|^4_{H^2(\Gamma)} \Big] \|\bfv\|^2_{H^1(\Gamma)} + \delta \Big[\|\bfh\|^2_{H^{2.5}(\Gamma)} + \ve^2 \|\bfv\|^2_{H^2(\Gamma)} \Big] \\
&\quad \le C_{\delta,\delta_2} \big(1+\|\bfh\|^8_{H^2(\Gamma)}\big) \|\bfv\|^2_{H^1(\rO)} + \delta \Big[\|\bfv\|^2_{H^2(\rO)} + \|\bfh\|^2_{H^{2.5}(\Gamma)}\Big] + \delta_2 \ve^2 \|\bfv\|^2_{H^2(\Gamma)} \,.
\end{align*}
Therefore, we obtain that
\begin{align}
& \lim_{\eps \to 0} \int_0^t \int_\Gamma \L_\ve(\bfh) (\bfJ \bfA^j_i \varphi^i)\, dS d\tilde{t} \nonumber\\
&\quad \ge \frac{1}{2} \|\widetilde{\chi} \bfh''(t)\|^2_{L^2(\Gamma)} - C \|h_0\|^2_{H^2(\Gamma)} - C_{\delta,\delta_2} \int_0^t \big(1+\|\bfh\|^8_{H^2(\Gamma)}\big) \|\bfv\|^2_{H^1(\rO)} d\tilde{t} \label{boundary_term_est}\\
&\qquad - (C \sqrt{\varsigma} + \delta) \int_0^t \Big[\|\bfv\|^2_{H^2(\rO)} + \|\bfh\|^2_{H^{2.5}(\Gamma)} \Big] d\tilde{t} - \delta_2 \ve^2 \int_0^t \|\bfv\|^2_{H^2(\Gamma)} d\tilde{t} \nonumber\,.
\end{align}

\subsubsection{The estimate of $\ve^2 \big(\Delta_0 (\bfJ \bfA^j_r \bfv^r), \bfJ \bfA^j_s \varphi^s\big)_{L^2(\Gamma)}$}
Let $|\theta_{m,1}|^2 = \widetilde{g}$. Then
\begin{align*}
\Delta_0 (\bfJ \bfA^j_r \bfv^r) \circ \theta_m = \frac{1}{\sqrt{\widetilde{g}}} \Big(\frac{1}{\sqrt{\widetilde{g}}} (\wfJ \wfA^j_r \wfv^r)_{,1}\Big)_{,1} \qquad\text{on}\quad B(0,r_m)\cap \{y_2 = 0\}\,.
\end{align*}
Therefore, by $\varphi \to (\widetilde{\chi}^2 \wfv_{,1})_{,1}$ in $L^2(0,\rT;L^2(\Gamma))$,
\begin{align*}
& \lim_{\eps\to 0} \ve^2 \int_0^t \big(\Delta_0 (\bfJ \bfA^j_r \bfv^r), \bfJ \bfA^j_s \varphi^s \big)_{L^2(\Gamma)} d\tilde{t} \\
&\qquad = \ve^2 \int_0^t \int_{B(0,r_m)\cap \{y_2=0\}} \Big(\frac{1}{\sqrt{\widetilde{g}}} (\wfJ \wfA^j_r \wfv^r)_{,1}\Big)_{,1} \wfJ \wfA^j_s (\widetilde{\chi}^2 \wfv^s_{,1})_{,1} dy_1 d\tilde{t} \\
&\qquad \ge \ve^2 \int_0^t \int_{B(0,r_m)\cap \{y_2=0\}} \frac{ \wfJ^2 \wfA^j_r \wfA^j_s}{\sqrt{\widetilde{g}}}\, (\widetilde{\chi} \wfv^r_{,11}) (\widetilde{\chi} \wfv^s_{,11})\, dy_1 d\tilde{t} \\
&\qquad\quad - C \ve^2 \int_0^t \Big[\|\bfJ \bfA\|_{H^2(\Gamma)} \|\bfv\|_{L^\infty(\Gamma)} + \|\bfJ \bfA\|_{W^{1,\infty}(\Gamma)} \|\bfv\|_{H^1(\Gamma)} \Big] \|\bfv\|_{H^2(\Gamma)} d\tilde{t}\,.
\end{align*}
By interpolation and the properties of convolution,
\begin{align*}
\|\bfJ \bfA\|_{H^2(\Gamma)} \le \frac{C}{\ve} \|\bfh\|_{H^2(\Gamma)},\quad
\|\bfJ \bfA\|_{W^{1,\infty}(\Gamma)} \le C \|\bfJ \bfA\|^{1/2}_{H^1(\Gamma)} \|\bfJ \bfA\|^{1/2}_{H^2(\Gamma)} \le \frac{C}{\sqrt{\ve}} \|\bfh\|_{H^2(\Gamma)};
\end{align*}
thus by interpolation, the trace estimate, and Young's inequality we conclude that
\begin{align}
& \lim_{\eps\to 0} \ve^2 \int_0^t \big(\Delta_0 (\bfJ \bfA^j_r \bfv^r), \bfJ \bfA^j_s \varphi^s \big)_{L^2(\Gamma)} d\tilde{t} \nonumber\\
&\quad \ge \ve^2 \int_0^t \int_{B(0,r_m)\cap \{y_2=0\}} \frac{1}{\sqrt{\widetilde{g}}}\, |\widetilde{\chi} \wfJ \wfA^j_r \wfv^r_{,11}|^2 \, dy_1 d\tilde{t} \nonumber\\
%&\qquad - C \ve \int_0^t \|\bfh\|_{H^2(\Gamma)} \|\bfv\|_{L^\infty(\Gamma)} \|\bfv\|_{H^2(\Gamma)} d\tilde{t} - C \ve^{3/2} \int_0^t \|\bfh\|_{H^2(\Gamma)} \|\bfv\|_{H^1(\Gamma)} \|\bfv\|_{H^2(\Gamma)} d\tilde{t} \nonumber\\
%&\quad \ge \ve^2 (1 - C \varsigma) \int_0^t \big\|\widetilde{g}^{-1/4} \widetilde{\chi} \wfv_{,11}\big\|^2_{L^2(B(0,r_m)\cap \{y_2=0\})} d\tilde{t} \nonumber\\
&\qquad - C \ve \int_0^t \|\bfh\|_{H^2(\Gamma)} \|\bfv\|^{6/7}_{H^{0.75}(\rO)} \|\bfv\|^{8/7}_{H^2(\Gamma)} d\tilde{t} - C \ve^{3/2} \int_0^t \|\bfh\|_{H^2(\Gamma)} \|\bfv\|^{2/3}_{H^1(\rO)} \|\bfv\|^{4/3}_{H^2(\Gamma)} d\tilde{t} \nonumber\\
&\quad \ge \ve^2 (1 - C \varsigma) \int_0^t \big\|\widetilde{g}^{-1/4} \widetilde{\chi} \wfv_{,11}\big\|^2_{L^2(B(0,r_m)\cap \{y_2=0\})} d\tilde{t} - \delta \int_0^t \|\bfv\|^2_{H^2(\rO)} d\tilde{t} \label{boundary_term_est1} \\
&\qquad - C_{\delta,\delta_2} \int_0^t \big(1+\|\bfh\|^3_{H^2(\Gamma)}\big) \|\bfv\|^2_{H^1(\rO)} d\tilde{t} - \delta_2 \ve^2 \int_0^t \|\bfv\|^2_{H^2(\Gamma)} d\tilde{t} \,. \nonumber
\end{align}

\subsubsection{The combination of {\rm(\ref{vt_term_est})}-{\rm(\ref{boundary_term_est1})}}
Summing (\ref{vt_term_est})%, (\ref{laplacian_term_est}), (\ref{pressure_term_est}), (\ref{convection_term_est}), (\ref{boundary_term_est}) and
-(\ref{boundary_term_est1}) over $m=1,\cdots,K$, by the Korn's inequality, the trace estimate, estimate (\ref{hH25_estimate}) and choosing $\delta_2 > 0$ small enough we find that
\begin{align}
& \|\sqrt{\wfJ} \widetilde{\chi} \wfv_{,1}\|^2_{L^2(B_+(0,r_m))} \hspace{-1pt}+\hspace{-1pt} \|\bfh(t)\|^2_{H^2(\Gamma)} \hspace{-1pt}+\hspace{-1pt} \int_0^t \Big[\|\bfv\|^2_{H^{1.5}(\Gamma)} \hspace{-1pt}+\hspace{-1pt} \ve^2 \|\bfv\|^2_{H^2(\Gamma)} \Big] d\tilde{t} \nonumber\\
&\quad \le C \big(1 \hspace{-1pt}+\hspace{-1pt} \|u_0\|^2_{H^1(\rO)} \hspace{-1pt}+\hspace{-1pt} \|h_0\|^2_{H^2(\Gamma)}\big) \hspace{-1pt}+\hspace{-1pt} (C \sqrt{\varsigma} \hspace{-1pt}+\hspace{-1pt} \delta) \int_0^t \Big[\|\bfv\|^2_{H^2(\rO)} \hspace{-1pt}+\hspace{-1pt} \|\bfq\|^2_{H^1(\rO)} \Big] d\tilde{t} \label{bfv_trace_estimate_temp}\\
&\qquad \hspace{-1pt}+\hspace{-1pt} C_\delta \int_0^t \big(1+ \|\bfv\|^2_{H^1(\rO)} \hspace{-1pt}+\hspace{-1pt} \|\bfh\|^8_{H^2(\Gamma)}\big) \|\bfv\|^2_{H^1(\rO)} d\tilde{t} \,. \nonumber
\end{align}

\subsection{The implication of the Stokes regularity}
First, we combine (\ref{bfvt_L2L2_est}) and (\ref{bfv_trace_estimate_temp}) as well as choose $\delta_1 > 0$ small enough to conclude that
\begin{align}
& \|\bfv(t)\|^2_{H^1(\rO)} + \|\bfh(t)\|^2_{H^2(\Gamma)} + \int_0^t \Big[\|\bfv_t\|^2_{L^2(\rO)} + \|\bfv\|^2_{H^{1.5}(\Gamma)} + \ve^2 \|\bfv\|^2_{H^2(\Gamma)} \Big] d\tilde{t} \nonumber\\
&\qquad \le C_\delta \big[1 + \|u_0\|^2_{H^1(\rO)} + \|h_0\|^2_{H^2(\Gamma)} \big] + \delta \int_0^t \big[\|\bfv\|^2_{H^2(\rO)} + \|\bfq\|^2_{H^1(\rO)} \big] d\tilde{t} \label{bfv_trace_est}\\
&\qquad\quad + C_\delta \int_0^t \Big[\big(\|\bfh\|^4_{H^2(\Gamma)} + 1\big) \big(1+\|\bfv\|^2_{H^1(\rO)}\big) + \|\bfv\|^{10}_{H^1(\rO)} \Big] d\tilde{t} \,. \nonumber
\end{align}
Next, we rewrite (\ref{NSreg1}a,b) as
\begin{alignat*}{2}
- \Delta \bfv + \nabla \bfq &= f &&\text{in}\quad\rO\,,\\
\div \bfv &= g \qquad&&\text{in}\quad\rO\,,
\end{alignat*}
where $f$ and $g$ are given by
\begin{align*}
f^i &= \bfA^\ell_j (\Psi^j_t \hspace{-1pt}-\hspace{-1pt} \bfv^j) \bfv^i_{,\ell} - \bfv_t^i + (\delta^k_j - \bfA^k_\ell \bfA^j_\ell) \bfv^i_{,jk} + (\delta^k_\ell \delta^j_i - \bfA^k_\ell \bfA^j_i) \bfv^\ell_{,jk} \\
&\quad - \bfA^k_\ell \big(\bfA^j_{\ell,k} \bfv^i_{,j} + \bfA^j_{i,k} \bfv^\ell_{,j}\big) + (\delta^j_i - \bfA^j_i) \bfq_{,j} \,,\\
g &= (\delta^j_i - \bfJ \bfA^j_i) \bfv^i_{,j}\,.
\end{align*}
Applying the regularity theory for the Stokes equation, by interpolation and Young's inequality we obtain that
\begin{align*}
& \|\bfv\|^2_{H^2(\rO)} + \|\bfq\|^2_{H^1(\rO)} \le C \Big[\|f\|^2_{L^2(\rO)} + \|g\|^2_{H^1(\rO)} + \|\bfv\|^2_{H^{1.5}(\Gamma)}\Big] \\
&\qquad \le C \Big[\|\bfv_t\|^2_{L^2(\rO)} + \|\bfv\|^2_{H^{1.5}(\Gamma)} + \|\nabla \bfA\|^2_{L^4(\rO)} \|\nabla \bfv\|^2_{L^4(\rO)} \Big] \\
&\qquad\quad + C \big(\|\Psi_t\|_{L^4(\rO)} + \|\bfv\|_{L^4(\rO)} \big) \|\nabla \bfv\|_{L^4(\rO)} + C \varsigma \Big[\|\bfv\|^2_{H^2(\rO)} + \|\bfq\|^2_{H^1(\rO)} \Big] \\
&\qquad \le C \Big[\|\bfv_t\|^2_{L^2(\rO)} + \|\bfv\|^2_{H^{1.5}(\Gamma)} \Big] + C_\delta \big(\|\bfv\|^2_{H^1(\rO)} + \|\bfv\|^4_{H^1(\rO)}\big) \\
&\qquad\quad + (C \varsigma + \delta) \Big[\|\bfv\|^2_{H^2(\rO)} + \|\bfq\|^2_{H^1(\rO)} + \|\bfh\|^2_{H^{2.5}(\Gamma)} \Big]\,;
\end{align*}
thus by (\ref{hH25_estimate}) and (\ref{bfv_trace_est}) with $\varsigma \ll 1$ and $\delta>0$ small enough,
\begin{equation}\label{bfvq_H2H1_est}
\begin{array}{l}
\displaystyle{} \int_0^t \Big[\|\bfv\|^2_{H^2(\rO)} + \|\bfq\|^2_{H^1(\rO)}\Big] d\tilde{t} \le C \big[1 + \|u_0\|^2_{H^1(\rO)} + \|h_0\|^2_{H^2(\Gamma)} \big] \vspace{.1cm}\\
\displaystyle{}\qquad\quad + C \int_0^t \Big[\big(\|\bfh\|^4_{H^2(\Gamma)} + 1\big) \big(1+\|\bfv\|^2_{H^1(\rO)}\big) + \|\bfv\|^{10}_{H^1(\rO)} \Big] d\tilde{t}\,.
\end{array}
\end{equation}
Let $E(t)$ be the energy function defined by
$$
E(t) = \sup_{s\in[0,t]} \Big[\|\bfv(s)\|^2_{H^1(\rO)} + \|\bfh(s)\|^2_{H^2(\Gamma)}\Big] + \int_0^t \|\bfv\|^2_{H^2(\rO)} d\tilde{t}
$$
Then $E(t)$ is continuous in $t$, and (\ref{bfv_trace_est}) and (\ref{bfvq_H2H1_est}) together imply that
$$
E(t) \le \rC \big[1 + \|u_0\|^2_{H^1(\rO)} + \|h_0\|^2_{H^2(\Gamma)} \big] + t \P\big(E(t)\big)
$$
for some polynomial $\P$. Therefore, there exists $\rT_* > 0$ such that
\begin{equation}\label{vhq_estimate}
E(t) \le 2 \rC\, \Big[1 + \|u_0\|^2_{H^1(\rO)} + \|h_0\|^2_{H^2(\Gamma)} \Big] \qquad\Forall t\in [0,\rT_*]
\end{equation}
whenever $E(t)$ exists.

\section{The existence of a solution to the problem}\label{sec:time_continuation}
\subsection{The continuation argument}
Let $\rM_1 > 0$ be given by
$$
\rM^2_1 = \max\Big\{2 \rC\, \Big[1 + \|u_0\|^2_{H^1(\rO)} + \|h_0\|^2_{H^2(\Gamma)}\Big], \rM^2\Big\}\,,
$$
and $\rT \equiv \min\{\rT_{\rM_1}, \rT_*\}$, where we recall that $\rT_{\rM_1}$ is chosen (in Proposition \ref{prop:small_h}) so that
$$
\|h(t)\|_{H^{1.7}(\Gamma)} < \varsigma \qquad\Forall t\in [0,\rT_{\rM_1}]\,.
$$
In Section \ref{sec:construction}, we have shown that for any given $(w_0,h_0) \in H^1_\div (\rO)\times H^2(\Gamma)$ with $\|h_0\|_{H^{1.7}(\Gamma)} < \varsigma$, there exists a solution $(w^\ve, q^\ve, h^\ve)$ to equation (\ref{NSreg}) in the time interval $[0,\rT_\ve]$ for some $\rT_\ve$ depending on the smoothing parameter $\ve$.
In principle, $\rT_\ve \to 0$ as $\ve \to 0$. Nevertheless, estimate (\ref{vhq_estimate}) implies that if $\rT_\ve \le \rT$,
$$
w^\ve(\rT_\ve) \in H^1_\div(\rO) \quad\text{and}\quad h^\ve(\rT_\ve) \in H^2(\Gamma)\,;
$$
thus if $\rT_\ve < \rT$, we can use $\big(w^\ve(\rT_\ve), h^\ve(\rT_\ve)\big)$ as a new set of initial data and extend the time interval of existence (using the updated initial data) to $[0, \rT'_\ve]$ for some $\rT'_\ve > \rT_\ve$. Estimate (\ref{vhq_estimate}) still holds in the time interval $[0,\rT_\ve']$; thus if $\rT'_\ve < \rT$, we can keep the process of solving (\ref{NSreg}) using new initial data. This implies that solution $(w^\ve,q^\ve,h^\ve)$ exists in the time interval $[0, \rT]$.

\subsection{The existence of a solution to equation (\ref{NSnew})}
Since $(w^\ve,q^\ve,h^\ve)$ exists in a time interval independent of $\ve$, and satisfies estimate (\ref{vhq_estimate}) with $\ve$-independent upper bound, we can pass $\ve \to 0$ and obtain that
\begin{alignat*}{2}
(w^\ve, w^\ve_t) &\rightharpoonup w &&\text{in}\quad L^2(0,\rT;H^2(\rO))\times L^2(0,\rT;L^2(\rO))\,, \\
q^\ve &\rightharpoonup q &&\text{in}\quad L^2(0,\rT;H^1(\rO))\,, \\
(h^\ve,h^\ve_t) &\rightharpoonup (h,h_t)\qquad &&\text{in}\quad L^2(0,\rT;H^{2.5}(\Gamma))\times L^2(0,\rT;H^{1.5}(\Gamma))
\end{alignat*}
for some $(w,q,h) \in \V(\rT)\times \Q(\rT)\times \H(\rT)$. In particular, there exists $\ve_j$ such that
$$
h^{\ve_j} \to h \qquad\text{in}\quad C([0,\rT];H^{1.75}(\Gamma))
$$
which further implies that
\begin{alignat*}{2}
\psi^{\ve_j} &\to \psi \qquad&&\text{in}\quad C([0,T];H^{2.25}(\rO))\,, \\
\rA^{\ve_j} &\to \rA &&\text{in}\quad C([0,T];H^{1.25}(\rO)) \contsubset C([0,T];C(\rO))\,,
\end{alignat*}
where $\psi^{\ve_j}$ is the ALE map corresponding to $h^{\ve_j}$ and $\rA^{\ve_j} = (\nabla \psi^{\ve_j})^{-1}$. Since $(w^{\ve_j},q^{\ve_j},h^{\ve_j})$ satisfies
\begin{align*}
& \big(w^{\ve_j}_t, \varphi\big)_{L^2(\rO)} + \rB_{\psi^{\ve_j}}(w^{\ve_j},\varphi) + \ve_j \big(w^{\ve_j\prime}, \varphi^\pprime\big)_{L^2(\Gamma)} + \big(q^{\ve_j},\div \varphi)_{L^2(\rO)} \\
&\qquad\quad = (\rF, \varphi)_{L^2(\rO)} + \int_\Gamma \L_{\ve_j}(h^{\ve_j}) (\varphi\cdot \rN)\, dS \qquad\Forall \varphi\in H^1(\rO), \text{ a.e. $t\in [0,\rT]$,}
\end{align*}
integrating the equality above in $t$ over the time interval $(a,b) \subseteq (0,\rT)$ and then passing $j \to \infty$, we conclude that $(w,q,h)$ satisfies
\begin{align*}
& \int_a^b \Big[\big(w_t, \varphi\big)_{L^2(\rO)} + \rB_{\psi}(w,\varphi) + \big(q,\div \varphi)_{L^2(\rO)}\Big] dt \\
&\qquad\quad = \int_a^b (\rF, \varphi)_{L^2(\rO)} dt + \int_a^b \int_\Gamma \L(h)(\varphi\cdot \rN) dS dt \qquad\Forall \varphi\in H^1(\rO)\,,
\end{align*}
and the Lebesgue differentiation theorem further implies that $(w,q,\varphi)$ satisfies
\begin{align*}
& \big(w_t, \varphi\big)_{L^2(\rO)} + \rB_{\psi}(w,\varphi) + \big(q,\div \varphi)_{L^2(\rO)} \\
&\qquad\quad = (\rF, \varphi)_{L^2(\rO)} + \int_\Gamma \L(h)(\varphi\cdot \rN)\, dS \qquad\Forall \varphi\in H^1(\rO), \text{ a.e. $t\in [0,\rT]$.}
\end{align*}
Moreover, since $\big(w^\ve(0),h^\ve(0)\big) = (w_0, h_0)$ for all $\ve>0$, we must have $\big(w(0),h(0)\big) = (w_0, h_0)$. Therefore, we establish the existence of a solution $(w,q,h)$ to equation (\ref{NSnew}) which is the main result of this paper.

\appendix
\newpage
\section{Proof of Theorem \ref{thm:elliptic_regularity}}\label{app:Stokes}
\noindent {\bf Theorem \ref{thm:elliptic_regularity}.}
{\sl Let $a^{jk}_{rs}$ be a $(2,2)$-tensor such that $a^{jk}_{rs} = a^{kj}_{rs} = a^{jk}_{sr}$, and satisfy}
$$
\|a^{jk}_{rs} - \lambda_1 \delta^j_k \delta^r_s - \lambda_2 \delta^k_r \delta^j_s\|_{L^\infty(\rO)} \ll 1 \eqno{\rm(\ref{smallness})}
$$
for some positive constants $\lambda_1$ and $\lambda_2$.
\begin{enumerate}
\item {\sl Suppose that $(w,q)\in \rV \times L^2(\rO)$ is a weak solution to the following elliptic equation
  $$
  - \big[a^{jk}_{rs} w^r_{,j}\big]_{,k} + q_{,s} = f^s \hspace{58pt}\text{in}\quad\rO\,, \eqno{\rm(\ref{elliptic}a)}
  $$
  $$
  \hspace{54pt}\div w = 0 \hspace{64pt}\text{in}\quad \rO\,, \eqno{\rm(\ref{elliptic}b)}
  $$
  $$
  \hspace{8pt}a^{jk}_{rs} w^r_{,j} \rN_k - q\rN_s = \ve \Delta_0 w^s + g^s \hspace{18pt}\text{on}\quad\Gamma\,; \eqno{\rm(\ref{elliptic}c)}
  $$
  that is, $(w,q)$ satisfies the variational formulation
  $$
  \begin{array}{l}
   \displaystyle{} \big(a^{jk}_{rs} w^r_{,j}, \varphi^s_{,k})_{L^2(\rO)} - (q,\div \varphi)_{L^2(\rO)} + \ve (w^\pprime, \varphi^\pprime)_{L^2(\Gamma)} \vspace{.2cm}\\
  \displaystyle{} \hspace{60pt} = (f, \varphi)_{L^2(\rO)} + (g,\varphi)_{L^2(\Gamma)} \qquad\qquad\Forall \varphi \in \rV\,.
  \end{array}
  \eqno{\rm(\ref{elliptic_weak})}
  $$
  Then $(w,q) \in H^2(\rO) \times H^1(\rO)$, and there are constants $C$ and $C_\ve$ such that}
  $$
  \begin{array}{l}
   \displaystyle{} \hspace{35pt}\|w\|^2_{H^2(\rO)} \hspace{-1pt}+\hspace{-1pt} \ve \|w\|^2_{H^2(\Gamma)} \hspace{-1pt}+\hspace{-1pt} \|q\|^2_{H^1(\rO)} \hspace{-1pt}\le\hspace{-1pt} C_\ve \|g\|^2_{L^2(\Gamma)} \hspace{-1pt}+\hspace{-1pt} C \big(1 \hspace{-1pt}+\hspace{-1pt} \|a\|^2_{L^\infty(\rO)}\big) \times \vspace{.2cm}\\
   \displaystyle{} \hspace{60pt} \times \Big[\big(1 \hspace{-1pt}+\hspace{-1pt} \|a\|^2_{W^{1,\infty}(\rO)} \big)\|w\|^2_{H^1(\rO)} \hspace{-1pt}+\hspace{-1pt} \|f\|^2_{L^2(\rO)} \hspace{-1pt}+\hspace{-1pt} \|g\|^2_{H^{-0.5}(\Gamma)} \Big] \,.
  \end{array}
  \eqno{\rm(\ref{elliptic_est})}
  $$
\item {\sl Suppose that $a \in W^{1,4}(\rO)$, and $(w,q)\in H^1_0(\rO) \times L^2 (\rO)$ is a weak solution to the following elliptic equation
  $$
  - \big[a^{jk}_{rs} w^r_{,j}\big]_{,k} + q_{,s} = f^s \hspace{16pt}\text{in}\quad\rO\,, \eqno{\rm(\ref{elliptic_Dirichlet}a)}
  $$
  $$
  \hspace{54pt}\div w = 0 \hspace{20pt}\text{in}\quad \rO\,, \eqno{\rm(\ref{elliptic_Dirichlet}b)}
  $$
  $$
  \hspace{68pt}w = 0 \hspace{20pt}\text{on}\quad\Gamma\,; \eqno{\rm(\ref{elliptic_Dirichlet}c)}
  $$
  that is,
  $$
   \ \ \big(a^{jk}_{rs} w^r_{,j}, \varphi^s_{,k})_{L^2(\rO)} - (q,\div \varphi)_{L^2(\rO)} = (f, \varphi)_{L^2(\rO)} \quad\ \Forall \varphi \in H^1_0(\rO) \,. \eqno{\rm(\ref{elliptic_weak_Dirichlet})}
  $$
  Then $(w,q) \in H^2(\rO) \times H^1(\rO)$ satisfies}
  $$
   \qquad\qquad\ \ \|w\|^2_{H^2(\rO)} + \|q\|^2_{H^1(\rO)} \le C \Big[1 + \|f\|^2_{L^2(\rO)} + \|\nabla a\|^4_{L^4(\rO)} \|w\|^2_{H^1(\rO)} \Big]\,. \eqno{\rm(\ref{elliptic_est2})}
  $$
\end{enumerate}
\begin{proof}
For simplicity, we assume $\lambda_1 = 1$ and $\lambda_2 = 0$, while the additional term $\lambda_2 \delta^k_r \delta^j_s$ with $\lambda_2>0$ can be treated in exactly the same way and with the help of the Korn inequality
$$
c_1 \|u\|_{H^1(\rO)} \le \|u\|_{L^2(\rO)} + \|\Def u\|_{L^2(\rO)} \le C_1 \|u\|_{H^1(\rO)}\,.
$$

\noindent {\bf Part 1:} We prove (1) first.

\noindent {\bf Basic energy estimates:}
Letting $\varphi = w$ in (\ref{elliptic_weak}) we obtain that
\begin{equation}\label{basic_elliptic_est}
\|\nabla w\|^2_{L^2(\rO)} + \ve \|w^\pprime\|^2_{L^2(\Gamma)} \le C \Big[\|w\|^2_{L^2(\rO)} + \|f\|^2_{L^2(\rO)} + \|g\|^2_{H^{-0.5}(\Gamma)} \Big]
\end{equation}
and the Lagrange multiplier lemma (with $a^j_i = \delta^j_i$) suggests that
\begin{align}
\|q\|_{L^2(\rO)} &\le C \Big[\|\nabla w\|_{L^2(\rO)} + \ve \|w^\pprime\|_{L^2(\Gamma)} + \|f\|_{L^2(\rO)} + \|g\|_{H^{-0.5}(\rO)} \Big] \nonumber\\
&\le C \Big[\|w\|_{L^2(\rO)} + \|f\|_{L^2(\rO)} + \|g\|_{H^{-0.5}(\rO)} \Big]\,. \label{basic_elliptic_est_q}
\end{align}

\noindent {\bf Estimates in the interior:}
Let $\chi$ be a smooth cut-off function with $\supp(\chi) \cptsubset \rO$, and $\eta_\eps$ be a family of mollifiers. Since
\begin{align}
& \big((a^{jk}_{rs} w^r_{,j}), \eta_\eps * [\chi^2 (\eta_\eps*w^s)_{,\ell}]_{,\ell k}\big)_{L^2(\rO)} = \big(\eta_\eps *(a^{jk}_{rs} w^r_{,j})_{,\ell},[\chi^2 (\eta_\eps*w^s)_{,\ell}]_{,k}\big)_{L^2(\rO)} \nonumber\\
&\qquad = \big(a^{jk}_{rs} (\eta_\eps* w^r_{,j})_{,\ell}, [\chi^2 (\eta_\eps*w^s_{,k})_{,\ell}] + 2 [\chi \chi_{,k} (\eta_\eps*w^s_{,\ell})]\big)_{L^2(\rO)} \nonumber\\
&\qquad\quad + \big((\comm{\eta_\eps*}{a^{jk}_{rs}} w^r_{,j})_{,\ell}, [\chi^2 (\eta_\eps*w^s_{,k})_{,\ell}] + 2 [\chi \chi_{,k} (\eta_\eps*w^s_{,\ell})] \big)_{L^2(\rO)} \nonumber\\
&\qquad \ge \big(1 - \|a - \id\otimes \id\|_{L^\infty(\rO)}\big) \|\chi \nabla (\eta_\eps* \nabla w)\|^2_{L^2(\rO)} \label{testing_eq1} \\
&\qquad\quad - C\|a\|_{L^\infty(\rO)} \|\nabla w\|_{L^2(\rO)} \|\chi \nabla (\eta_\eps* \nabla w)\|_{L^2(\rO)} \nonumber\\
&\qquad\quad - C \big\|(\comm{\eta_\eps*}{a^{jk}_{rs}} w^r_{,j})_{,\ell}\big\|_{L^2(\rO)} \Big[\|\chi \nabla (\eta_\eps* \nabla w)\|_{L^2(\rO)} + \|\nabla w\|_{L^2(\rO)}\Big] \nonumber
\end{align}
and by $\div w = 0$,
\begin{align}
& - (q, \eta_\eps * [\chi^2 (\eta_\eps*w^i)_{,\ell}]_{,\ell i}\big)_{L^2(\rO)} = - 2 (q, \eta_\eps * [\chi \chi_{,i} (\eta_\eps*w^i_{,\ell})]_{,\ell}\big)_{L^2(\rO)} \nonumber\\
& \qquad\quad \le C \|q\|_{L^2(\rO)} \Big[\|\nabla w\|_{L^2(\rO)} + \|\chi \nabla (\eta_\eps*\nabla)\|_{L^2(\rO)} \Big]\,, \label{testing_eq2}
\end{align}
the use of $\eta_\eps * [\chi^2 (\eta_\eps*w)_{,\ell}]_{,\ell}$ as a test function in (\ref{elliptic_weak}) suggests that
%\begin{align*}
%& \big(a^{jk}_{rs} w^r_{,j}, \varphi^s_{,k})_{L^2(\rO)} - (q,\div \varphi)_{L^2(\rO)} = (f, \varphi)_{L^2(\rO)} + (g,\varphi)_{L^2(\Gamma)} \\
%& \big(\eta_\eps * (a^{jk}_{rs} w^r_{,j})_{,\ell}, [\chi^2 (\eta_\eps*w^s)_{,\ell}]_{,k}\big)_{L^2(\rO)} - (q, \eta_\eps * [\chi^2 (\eta_\eps*w)_{,\ell}]_{,\ell i}\big)_{L^2(\rO)}
%\end{align*}
\begin{align}
& \big(1 - \|a - \id\otimes \id\|_{L^\infty(\rO)}\big) \|\chi \nabla (\eta_\eps* \nabla w)\|^2_{L^2(\rO)} \nonumber\\
&\qquad\le C \|a\|_{L^\infty(\rO)} \|\nabla w\|_{L^2(\rO)} \|\chi \nabla (\eta_\eps* \nabla w)\|_{L^2(\rO)} \nonumber\\
&\qquad\quad + C \big\|(\comm{\eta_\eps*}{a^{jk}_{rs}} w^r_{,j})_{,\ell}\big\|_{L^2(\rO)} \Big[\|\chi \nabla (\eta_\eps* \nabla w)\|_{L^2(\rO)} + \|\nabla w\|_{L^2(\rO)}\Big] \label{ineq_temp1}\\
&\qquad\quad + C \|q\|_{L^2(\rO)} \Big[\|\nabla w\|_{L^2(\rO)} + \|\chi \nabla (\eta_\eps*\nabla)\|_{L^2(\rO)} \Big] \nonumber \\
&\qquad\quad + C \|f\|_{L^2(\rO)} \|\chi \nabla (\eta_\eps*\nabla w)\|_{L^2(\rO)}\,. \nonumber
\end{align}
Similar to (\ref{comm_est_temp2}),
$$
\big\|(\comm{\eta_\eps*}{a^{jk}_{rs}} w^r_{,j})_{,\ell}\big\|_{L^2(\rO)} \le C \|\nabla a\|_{L^\infty(\rO)} \|\nabla w\|_{L^2(\rO)}\,;
$$
thus applying Young's inequality to inequality (\ref{ineq_temp1}) we find that
\begin{align*}
& \|\chi \nabla (\eta_\eps* \nabla w)\|^2_{L^2(\rO)} \le C \Big[\|a\|_{W^{1,\infty}(\rO)} \hspace{-1pt}+\hspace{-1pt} \|a\|^2_{W^{1,\infty}(\rO)}\Big] \|\nabla w\|^2_{L^2(\rO)} \hspace{-1pt}+\hspace{-1pt} C \|f\|^2_{L^2(\rO)} \\
&\qquad\quad +\hspace{-1pt} C \|q\|_{L^2(\rO)} \|\nabla w\|_{L^2(\rO)} \\
&\qquad \le C \Big[ \big(1 \hspace{-1pt}+\hspace{-1pt} \|a\|^2_{W^{1,\infty}(\rO)} \big) \|\nabla w\|^2_{L^2(\rO)} \hspace{-1pt}+\hspace{-1pt} \|w\|^2_{L^2(\rO)} \hspace{-1pt}+\hspace{-1pt} \|f\|^2_{L^2(\rO)} \hspace{-1pt}+\hspace{-1pt} \|g\|^2_{H^{-0.5}(\Gamma)}\Big]\,.
\end{align*}
By (\ref{basic_elliptic_est}), the right-hand side of the estimate above is independent of $\eps$, we can pass $\eps \to 0$ and find that $w\in H^2_{\text{loc}}(\rO)$ satisfying
\begin{equation}\label{w_H2_loc_est}
\begin{array}{l}
\displaystyle{} \|\chi \nabla^2 w\|^2_{L^2(\rO)} \vspace{.2cm}\\
\displaystyle{} \quad\ \ \le C \Big[ \big(1 \hspace{-1pt}+\hspace{-1pt} \|a\|^2_{W^{1,\infty}(\rO)} \big) \|\nabla w\|^2_{L^2(\rO)} \hspace{-1pt}+\hspace{-1pt} \|w\|^2_{L^2(\rO)} \hspace{-1pt}+\hspace{-1pt} \|f\|^2_{L^2(\rO)} \hspace{-1pt}+\hspace{-1pt} \|g\|^2_{H^{-0.5}(\Gamma)}\Big]
\end{array}
\end{equation}
for some constant $C$ depending on $\nabla \chi$.

Let $\psi\in H^2(\rO)$ be given, and $\varphi = \chi \nabla \psi$ be a test function in (\ref{elliptic_weak}). Then
\begin{align*}
& (\chi q,\Delta \psi)_{L^2(\rO)} = - (a^{jk}_{rs} w^r_{,j}, (\chi \psi_{,s})_{,k})_{L^2(\rO)} + (\chi f - q \nabla \chi, \nabla \psi)_{L^2(\rO)} \\
&\quad = \big((a^{jk}_{rs} \chi w^r_{,j})_{,k}, \psi_{,s}\big)_{L^2(\rO)} - (a^{jk}_{rs} w^r_{,j} \chi_{,k}, \psi_{,s})_{L^2(\rO)} + (\chi f - q \nabla \chi, \nabla \psi)_{L^2(\rO)} \,;
\end{align*}
thus $\chi q$ is a distributional solution to
\begin{alignat*}{2}
\Delta (\chi q) &= f_1 \equiv - \big[a^{jk}_{rs} w^r_{,j} \nabla \chi\big]_{,s} + \big[(a^{jk}_{rs} \chi w^r_{,j})_{,k}\big]_{,s} - \div (\chi f - q \nabla \chi) \quad\ &&\text{in}\quad \rO\,,\\
\chi q &= 0 &&\text{on}\quad\Gamma\,.
\end{alignat*}
Since $f_1 \in H^{-1}(\rO)$ satisfies that
\begin{align*}
&\|f_1\|_{H^{-1}(\rO)} \le C \Big[\|a^{jk}_{rs} w^r_{,j} \nabla \chi\|_{L^2(\rO)} + \|(a^{jk}_{rs} \chi w^r_{,j})_{,k}\|_{L^2(\rO)} + \|\chi f - q \nabla \chi\|_{L^2(\rO)} \Big] \\
&\quad \le C \Big[\|a\|_{W^{1,\infty}(\rO)} \|\nabla w\|_{L^2(\rO)} + \|a\|_{L^\infty(\rO)} \|\chi \nabla^2 w\|_{L^2(\rO)} + \|f\|_{L^2(\rO)} + \|q\|_{L^2(\rO)} \Big]\,,
\end{align*}
we find that $q \in H^1_{\text{loc}}(\rO)$, and
\begin{equation}\label{q_H1_loc_est}
\begin{array}{l}
\displaystyle{} \|\chi q\|_{H^1(\rO)} \le C \|a\|_{L^\infty(\rO)} \big(1+\|a\|_{W^{1,\infty}(\rO)}\big) \|\nabla w\|_{L^2(\rO)} \vspace{.2cm}\\
\displaystyle{} \hspace{55pt} + C \big(1+ \|a\|_{L^\infty(\rO)} \big) \Big[\|w\|_{L^2(\rO)} + \|f\|_{L^2(\rO)} + \|g\|_{H^{-0.5}(\Gamma)} \Big]
\end{array}
\end{equation}
for some constant $C$ depending on $\nabla \chi$.

\noindent {\bf Estimates near the boundary:}
Now we focus on the boundary regularity. Let $\{\chi_m\}_{m=1}^K$ be cut-off functions supported near the boundary $\Gamma$ %so that there exists $\theta_m: B(0,r_m) \to \supp(\chi_m)$ satisfying
%\begin{enumerate}
%\item $\theta_m(B_+(0,r_m)) = \supp(\chi_m) \cap \rO$\,, where $B_+(0,r_m) \equiv B(0,r_m)\cap \{y_2 > 0\}$,
%\item $\theta_m(B(0,r_m)\cap \{y_2 = 0\}) = \supp(\chi_m) \cap \Gamma$\,;
%\item $\det(\nabla \theta_m) = 1$.
%\end{enumerate}
introduced in Section \ref{sec:v_ve_indep_est}. For a fixed $m$, define $\wfa = (\nabla \theta_m)^{-1}$, $\widetilde{\chi} = \chi_m\circ\theta_m$ in $\supp(\widetilde{\chi}_m)$, $(\widetilde{\bfw}, \widetilde{q}) = (w, q)\circ \theta_m$ in $B_+(0,r_m)$, and let
$$
\varphi^s = \big[\theta^s_{m,\ell} \big(\Lambda_\eps (\widetilde{\chi}^2 (\Lambda_\eps (\wfa^\ell_j \, \widetilde{\bfw}^j)_{,1}))_{,1}\big)\big] \circ \theta_m^{-1}
$$
be a test function in (\ref{elliptic_weak}), where if $x = \theta_m(y)$, then ${}_{,1}$ denotes the partial derivative with respect to $y_1$. Since
\begin{align*}
\varphi^s \circ\theta_m %&= \theta^s_{m,\ell} \big(\Lambda_\eps (\widetilde{\chi}^2 (\Lambda_\eps (\wfa^\ell_j \, \widetilde{\bfw}^j)_{,1}))_{,1}\big) \\
&= \theta^s_{m,\ell} \big(\Lambda_\eps (\widetilde{\chi}^2 (\Lambda_\eps (\wfa^\ell_{j,1} \, \widetilde{\bfw}^j)_{,1}))\big) + 2 \theta^s_{m,\ell} \big(\Lambda_\eps (\widetilde{\chi} \widetilde{\chi}_{,1} (\Lambda_\eps (\wfa^\ell_j \, \widetilde{\bfw}^j)_{,1}))\big) \\
&\quad + \theta^s_{m,\ell} \big(\Lambda_\eps (\widetilde{\chi}^2 (\comm{\Lambda_\eps}{\wfa^\ell_j} \, \widetilde{\bfw}^j_{,1})_{,1})\big) + \theta^s_{m,\ell} \big(\Lambda_\eps (\widetilde{\chi}^2 \wfa^\ell_{j,1} (\Lambda_\eps \widetilde{\bfw}^j_{,1})) \big) \\
&\quad + \theta^s_{m,\ell} \big(\Lambda_\eps (\widetilde{\chi}^2 \wfa^\ell_j (\Lambda_\eps \widetilde{\bfw}^j)_{,11}) \big)\,,
\end{align*}
similar to (\ref{testing_eq1}) we have
\begin{align*}
& \big(a^{jk}_{rs} w^r_{,j}, \varphi^s_{,k})_{L^2(\rO)} \\
&\qquad = \big((a^{jk}_{rs}\circ\theta_m)\, \wfa^\ell_j \widetilde{\bfw}^r_{,\ell}, \wfa^i_k \big[\theta^s_{m,\ell} \big(\Lambda_\eps (\widetilde{\chi}^2 (\Lambda_\eps (\wfa^\ell_p\, \widetilde{\bfw}^p)_{,1}))_{,1}\big)\big]_{,i}\big)_{L^2(B_+(0,r_m))} \\
&\qquad \ge \big(\lambda - \|\wfa\|^2_{L^\infty(B_+(0,r_m))} \|a - \id\otimes \id\|_{L^\infty(\rO)}\big) \|\widetilde{\chi} (\Lambda_\eps \nabla\widetilde{\bfw})_{,1}\|^2_{L^2(B_+(0,r_m))} \\
&\qquad\quad - C \|a\|_{W^{1,\infty}(\rO)} \|\nabla w\|_{L^2(\rO)} \Big[\|\widetilde{\chi} (\Lambda_\eps \nabla\widetilde{\bfw})_{,1}\|_{L^2(B_+(0,r_m))} + \|w\|_{H^1(\rO)} \Big]\,,
\end{align*}
in which we use the property that $\wfa^\ell_j \wfa^i_j \xi_\ell \xi_i \ge \lambda |\xi|^2$ for some positive constant $\lambda$. Similar to (\ref{testing_eq2}), by $\wfa^k_i \theta^i_{m,\ell} = \delta^k_\ell$ and $\wfa^j_i \widetilde{\bfw}^i_{,j} = (\div w)\circ\theta_m = 0$ in $B_+(0,r_m)$,
\begin{align*}
- (q, \div \varphi)_{L^2(\rO)} &= - \big(\widetilde{q}, \wfa^k_i \big[\theta^i_{m,\ell} \big(\Lambda_\eps (\widetilde{\chi}^2 (\Lambda_\eps (\wfa^\ell_j \, \widetilde{\bfw}^j)_{,1}))_{,1}\big)\big]_{,k}\big)_{L^2(B_+(0,r_m))} \\
&= - \big(\widetilde{q}, \big[\Lambda_\eps (\widetilde{\chi}^2 (\Lambda_\eps (\wfa^\ell_j \, \widetilde{\bfw}^j)_{,1}))_{,1}\big]_{,k}\big)_{L^2(B_+(0,r_m))} \\
&\ge\hspace{-1pt} -\, C \|q\|_{L^2(\rO)} \Big[\|w\|_{H^1(\rO)} + \|\widetilde{\chi} (\Lambda_\eps \nabla \widetilde{\bfw})_{,1}\|_{L^2(B_+(0,r_m))} \Big]\,.
\end{align*}
Moreover, for the two integrals over the boundary, we have
\begin{align*}
& (w^\pprime, \varphi^\pprime)_{L^2(\Gamma)} \ge \|\widetilde{\chi} (\Lambda_\eps \widetilde{\bfw})_{,11}\|^2_{L^2(B(0,r_m)\cap \{y_2 = 0\})} \\
&\quad - C \|\widetilde{\chi} (\Lambda_\eps \widetilde{\bfw})_{,11}\|_{L^2(B(0,r_m)\cap \{y_2 = 0\})} \Big[\|\widetilde{\chi} (\Lambda_\eps \widetilde{\bfw})_{,1}\|_{L^2(B(0,r_m)\cap \{y_2 = 0\})} + \|w\|_{L^2(\Gamma)}\Big]
\end{align*}
and
\begin{align*}
\big(g, \varphi)_{L^2(\Gamma)} \le %\left\{\begin{array}{ll}
%C \|g\|_{L^2(\Gamma)} \Big[\|\widetilde{\chi} (\Lambda_\eps \widetilde{\bfw})_{,1}\|_{L^2(B(0,r_m)\cap \{y_2 > 0\})} + \|w\|_{H^1(\Gamma)}\Big], \vspace{.2cm}\\
%C \|g\|_{H^{0.5}(\Gamma)} \Big[\|\widetilde{\chi} (\eta_\eps* \widetilde{\bfw})_{,1}\|_{H^{0.5}(B(0,r_m)\cap \{y_2 > 0\})} + \|w\|_{H^1(\Gamma)} \Big];
%\end{array}
%\right.
C \|g\|_{L^2(\Gamma)} \Big[\|\widetilde{\chi} (\Lambda_\eps \widetilde{\bfw})_{,11}\|_{L^2(B(0,r_m)\cap \{y_2 = 0\})} + \|w\|_{H^1(\Gamma)}\Big];
\end{align*}
thus the use of $\varphi^s = \big[\theta^s_{m,\ell} \big(\Lambda_\eps (\widetilde{\chi}^2 (\Lambda_\eps (\wfa^\ell_j \, \widetilde{\bfw}^j)_{,1}))_{,1}\big)\big] \circ \theta_m^{-1}$ as a test function in (\ref{elliptic_weak}) suggests that
\begin{align*}
& \frac{\lambda}{2} \|\widetilde{\chi} (\Lambda_\eps \nabla \widetilde{\bfw})_{,1}\|^2_{L^2(B_+(0,r_m))} + \frac{\ve}{2} \|\widetilde{\chi} (\Lambda_\eps \widetilde{\bfw})_{,11}\|^2_{L^2(B(0,r_m)\cap \{y_2 > 0\})} \\
&\qquad \le C \big(1 + \|a\|^2_{W^{1,\infty}(\rO)} \big)\|w\|^2_{H^1(\rO)} + C \Big[\|q\|^2_{L^2(\rO)} + \|f\|^2_{L^2(\rO)}\Big] \\
&\qquad\quad + C_\ve \|g\|^2_{L^2(\Gamma)} + \ve \|w\|^2_{H^1(\Gamma)} \\
&\qquad \le C \Big[\big(1 + \|a\|^2_{W^{1,\infty}(\rO)} \big)\|w\|^2_{H^1(\rO)} + \|f\|^2_{L^2(\rO)} + \|g\|^2_{H^{-0.5}(\Gamma)} \Big] + C_\ve \|g\|^2_{L^2(\Gamma)}\,.
\end{align*}
We note here that $C_\ve = \O(\ve^{-1})$.

Sum the estimates above over all $m$. Since the right-hand side of the estimate above is independent of $\eps$, letting $\U = \bigcup\limits_{m=1}^K \big\{\chi_m > 1/2\big\}$, we conclude that
\begin{equation}\label{w_H2_bdy_est}
\begin{array}{l}
\displaystyle{} \lambda \|\bp w\|^2_{H^1(\U)} \hspace{-1pt}+\hspace{-1pt} \ve \|w\|^2_{H^2(\Gamma)} \vspace{.2cm}\\
\displaystyle{} \quad\ \le C \Big[\big(1 \hspace{-1pt}+\hspace{-1pt} \|a\|^2_{W^{1,\infty}(\rO)} \big)\|w\|^2_{H^1(\rO)} \hspace{-1pt}+\hspace{-1pt} \|f\|^2_{L^2(\rO)} \hspace{-1pt}+\hspace{-1pt} \|g\|^2_{H^{-0.5}(\Gamma)} \Big] \hspace{-1pt}+\hspace{-1pt} C_\ve \|g\|^2_{L^2(\Gamma)}\,.
\end{array}
\end{equation}

Let $\varphi^i = \big[\theta^i_{m,\ell}\, \widetilde{\chi} \Lambda_\eps(\wfa^\ell_j \psi^j)_{,1}\big]\circ \theta_m^{-1}$ be a test function in (\ref{elliptic_weak}). First of all we find that
\begin{align*}
(q,\div \varphi)_{L^2(\rO)} %&= \big(\widetilde{q}, \wfa^k_i \big[\theta^i_{m,\ell}\, \widetilde{\chi} \Lambda_\eps(\wfa^\ell_j \psi^j)_{,1}\big]_{,k}\big)_{L^2(B_+(0,r_m))} \\
&= - \big((a^{jk}_{rs}\circ\theta_m) \wfa^p_j \widetilde{\bfw}^r_{,p} , \wfa^q_k \big[\theta^s_{m,\ell}\, \widetilde{\chi} \Lambda_\eps(\wfa^\ell_i \psi^i)_{,1}\big]_{,q} \big)_{L^2(B_+(0,r_m))} \\
&\quad + \ve (\widetilde{\bfw}^{\pprime\prime}, \theta^i_{m,\ell}\, \widetilde{\chi} \Lambda_\eps (\wfa^\ell_j \psi^j)_{,1}\big)_{L^2(B_+(0,r_m)\cap \{y_2 = 0\})} \\
&\quad + \big(f\circ\theta_m, \theta^i_{m,\ell}\, \widetilde{\chi} \Lambda_\eps(\wfa^\ell_j \psi^j)_{,1}\big)_{L^2(B_+(0,r_m))} \\
&\quad + \big(g\circ\theta_m, \theta^i_{m,\ell}\, \widetilde{\chi} \Lambda_\eps(\wfa^\ell_j \psi^j)_{,1}\big)_{L^2(B_+(0,r_m)\cap \{y_2=0\})} \,.
\end{align*}
By $\wfa^k_i \theta^i_{m,\ell} = \delta^k_\ell$ and the Piola identity $\wfa^k_{j,k} = 0$,
\begin{align*}
& (q,\div \varphi)_{L^2(\rO)} = \big(\widetilde{q}, \wfa^k_i \big[\theta^i_{m,\ell}\, \widetilde{\chi} \Lambda_\eps(\wfa^\ell_j \psi^j)_{,1}\big]_{,k}\big)_{L^2(B_+(0,r_m))} \\
&\quad = \big(\widetilde{q}, \wfa^k_i \big(\theta^i_{m,\ell}\, \widetilde{\chi}\big)_{,k} \Lambda_\eps (\wfa^\ell_j \psi^j)_{,1} + \widetilde{\chi} \Lambda_\eps(\wfa^k_j \psi^j_{,k})_{,1}\big)_{L^2(B_+(0,r_m))} \\
&\quad = \big(\Lambda_\eps(\wfa^k_i (\theta^i_{m,\ell}\, \widetilde{\chi})_{,k} \widetilde{q}), (\wfa^\ell_j \psi^j)_{,1} \big)_{L^2(B_+(0,r_m))} \hspace{-2pt}- \big(\Lambda_\eps(\widetilde{\chi} \widetilde{q})_{,1}, \wfa^k_j \psi^j_{,k}\big)_{L^2(B_+(0,r_m))} \,.
\end{align*}
Therefore, if we define a bounded linear function $T: \rV \to \bbR$ by
\begin{align*}
T(\varphi) &= \big(\Lambda_\eps(\wfa^k_i (\theta^i_{m,\ell}\, \widetilde{\chi})_{,k} \widetilde{q}), (\wfa^k_j \psi^j)_{,1}\big)_{L^2(B_+(0,r_m))} \\
&\quad - \big(\Lambda_\eps\big[\theta^s_{m,\ell}\, \widetilde{\chi} \wfa^q_k \big((a^{jk}_{rs}\circ\theta_m) \wfa^p_j \widetilde{\bfw}^r_{,p}\big)_{,q}\big], (\wfa^\ell_i \psi^i)_{,1} \big)_{L^2(B_+(0,r_m))} \\
&\quad - \ve (\Lambda_\eps\big[\theta^i_{m,\ell}\, \widetilde{\chi} \widetilde{\bfw}^{\pprime\prime}\big], (\wfa^\ell_j \psi^j)_{,1} \big)_{L^2(B_+(0,r_m)\cap \{y_2 = 0\})} \\
&\quad - \big(\Lambda_\eps\big[\theta^i_{m,\ell}\, \widetilde{\chi} (f\circ\theta_m)\big], (\wfa^\ell_j \psi^j)_{,1}\big)_{L^2(B_+(0,r_m))} \\
&\quad - \big(\Lambda_\eps\big[\theta^i_{m,\ell}\, \widetilde{\chi}(g\circ\theta_m)\big], (\wfa^\ell_j \psi^j)_{,1}\big)_{L^2(B_+(0,r_m)\cap \{y_2=0\})}\,,
\end{align*}
by the Lagrange multiplier lemma (with $a^j_i = \wfa^j_i$) we find that $\Lambda_\eps (\widetilde{\chi}\widetilde{q})_{,1}$ is the Lagrange multiplier associated to $T$, and
\begin{align*}
& \|\Lambda_\eps (\widetilde{\chi}\widetilde{q})_{,1}\|_{L^2(B_+(0,r_m))}
%\\
%&\quad \le C \Big[\|\wfa^k_i (\theta^i_{m,\ell}\, \widetilde{\chi})_{,k} \widetilde{q}\|_{L^2(B_+(0,r_m))} + \big\|\theta^s_{m,\ell}\, \widetilde{\chi} \wfa^q_k \big[(a^{jk}_{rs}\circ\theta_m) \wfa^p_j \widetilde{\bfw}^r_{,p}\big]_{,q}\big\|_{L^2(B_+(0,r_m))} \\
%&\qquad\quad + \ve \|\theta^i_{m,\ell}\, \widetilde{\chi} \widetilde{\bfw}^{\pprime\prime}\|_{L^2(B(0,r_m)\cap \{y_2=0\})} + \|\theta^i_{m,\ell}\, \widetilde{\chi} (f\circ\theta_m)\|_{L^2(B_+(0,r_m))} \\
%&\qquad\quad + \|\theta^i_{m,\ell}\, \widetilde{\chi}(g\circ\theta_m)\|_{L^2(B(0,r_m)\cap \{y_2=0\}} \Big] \\
%&\quad \le C \Big[\|q\|_{L^2(\rO)} + \|a\|_{L^\infty(\rO)} \|\chi \nabla^2 w\|_{L^2(\rO)} + \|\nabla a\|_{L^\infty(\rO)} \|\nabla w\|_{L^2(\rO)} + \ve \|w\|_{H^2(\Gamma)} \\
%&\qquad\quad + \|f\|_{L^2(\rO)} + \|g\|_{L^2(\Gamma)} \Big]
\le C \Big[\|q\|_{L^2(\rO)} + \|a\|_{L^\infty(\rO)} \|\chi \nabla^2 w\|_{L^2(\rO)} \\
&\qquad\quad + \|\nabla a\|_{L^\infty(\rO)} \|\nabla w\|_{L^2(\rO)} + \ve \|w\|_{H^2(\Gamma)} + \|f\|_{L^2(\rO)} + \|g\|_{L^2(\Gamma)} \Big] \\
%&\qquad \le C \big(1+\|a\|_{L^\infty(\rO)} \big) \big(1+ \|a\|_{W^{1,\infty}(\rO)} \big) \Big[\|w\|_{L^2(\rO)} + \|f\|_{L^2(\rO)} + \|g\|_{L^2(\Gamma)} \Big]\,. %+ C \|a\|_{L^\infty(\rO)}\big(1 + \|a\|_{W^{1,\infty}(\rO)} \big) \|\nabla w\|_{L^2(\rO)}
&\qquad \le C \big(1 + \|a\|_{L^\infty(\rO)}\big) \big(1 + \|a\|_{W^\infty(\rO)}\big) \|w\|_{H^1(\rO)} \\
&\qquad\quad + C \big(1 + \|a\|_{L^\infty(\rO)}\big) \Big[\|f\|_{L^2(\rO)} + \|g\|_{H^{-0.5}(\Gamma)} \Big] + C \|g\|_{L^2(\Gamma)}\,.
\end{align*}
Passing $\eps$ to zero, we obtain that
\begin{equation}\label{q_H1_bdy_est}
\begin{array}{l}
\displaystyle{} \|(\widetilde{\chi}\widetilde{q})_{,1}\|_{L^2(B_+(0,r_m))} \le C \big(1 + \|a\|_{L^\infty(\rO)} \big) \big(1 + \|a\|_{W^{1,\infty}(\rO)} \big) \|w\|_{H^1(\rO)} \vspace{.2cm}\\
\displaystyle{} \qquad\qquad + C \big(1 + \|a\|_{L^\infty(\rO)} \big) \Big[\|f\|_{L^2(\rO)} + \|g\|_{H^{-0.5}(\Gamma)} \Big] + C \|g\|_{L^2(\Gamma)} \,.
\end{array}
\end{equation}

\noindent {\bf Estimates of normal derivatives of $w$ and $q$:}
Since $(w,q)\in H^2_{\text{loc}}(\rO) \times H^1_{\text{loc}}(\rO)$, (\ref{elliptic}) holds almost everywhere. Making the change of variable $x = \theta_m(y)$, (\ref{elliptic}a,b) are transformed to
\begin{subequations}\label{elliptic_temp}
\begin{alignat}{2}
- \big[\wfa^\ell_k (a^{jk}_{rs}\circ\theta_m) \wfa^i_j \widetilde{\bfw}^r_{,i}\big]_{,\ell} + \wfa^k_s \widetilde{q}_{,k} &= (f\circ\theta_m) \quad\ &&\text{in}\quad B_+(0,r_m)\,,\\
\wfa^j_i \widetilde{\bfw}^i_{,j} &= 0 &&\text{in}\quad B_+(0,r_m)\,.
\end{alignat}
\end{subequations}
Differentiating (\ref{elliptic_temp}b) with respect to $y_2$, after rearrangement we find that
\begin{subequations}\label{rearrangement}
\begin{alignat}{2}
- (a^{jk}_{rs}\circ\theta_m) \wfa^2_j \wfa^2_k \widetilde{\bfw}^r_{,22} + \wfa^2_s \widetilde{q}_{,2} &= f^s_2 &&\text{in}\quad B_+(0,r_m)\,,\\
\wfa^2_i \widetilde{\bfw}^i_{,22} &= f_3 \qquad&&\text{in}\quad B_+(0,r_m)\,,
\end{alignat}
\end{subequations}
where $f_2$ and $f_3$ are given by
\begin{align*}
f^s_2 &\equiv (f\circ\theta_m) + \big[\wfa^\ell_k (a^{jk}_{rs}\circ\theta_m) \wfa^i_j\big]_{,\ell} \widetilde{\bfw}^r_{,i} + (a^{jk}_{rs}\circ\theta_m) \wfa^1_j \wfa^1_k \widetilde{\bfw}^r_{,11} \\
&\quad + 2 (a^{jk}_{rs}\circ\theta_m) \wfa^1_j \wfa^2_k \widetilde{\bfw}^r_{,12} - \wfa^1_s \widetilde{q}_{,1} \,,\\
f_3 &\equiv - \wfa^2_{i,2} \widetilde{\bfw}^i_{,2} - \wfa^1_i \widetilde{\bfw}^i_{,12} - \wfa^1_{i,2} \widetilde{\bfw}^i_{,1}\,,
\end{align*}
and satisfy
\begin{align*}
& \|\widetilde{\chi} f_2\|_{L^2(B_+(0,r_m))} + \|\widetilde{\chi} f_3\|_{L^2(B_+(0,r_m))} \le C \Big[\|f\|_{L^2(\rO)} + \|\widetilde{\chi} \widetilde{q}^\pprime\|_{L^2(B_+(0,r_m))} \\
&\qquad\quad + \big(1+ \|a\|_{W^{1,\infty}(\rO)}\big) \|w\|_{H^1(\rO)} + \big(1+\|a\|_{L^\infty(\rO)}\big) \|\widetilde{\chi} \widetilde{\bfw}^\pprime\|_{H^1(B_+(0,r_m))} \Big] \\
&\qquad \le C \big(1 + \|a\|_{L^\infty(\rO)}\big) \big(1 + \|a\|_{W^\infty(\rO)}\big) \|w\|_{H^1(\rO)} \\
&\qquad\quad + C \big(1 + \|a\|_{L^\infty(\rO)}\big) \Big[\|f\|_{L^2(\rO)} + \|g\|_{H^{-0.5}(\Gamma)} \Big] + C_\ve \|g\|_{L^2(\Gamma)} \,.
\end{align*}
Write (\ref{rearrangement}) in the matrix form
$$
\left[\begin{array}{ccc}
- (a^{jk}_{11}\circ\theta_m) \wfa^2_j \wfa^2_k & - (a^{jk}_{21}\circ\theta_m) \wfa^2_j \wfa^2_k & \wfa^2_1 \vspace{.2cm} \\
- (a^{jk}_{12}\circ\theta_m) \wfa^2_j \wfa^2_k & - (a^{jk}_{22}\circ\theta_m) \wfa^2_j \wfa^2_k & \wfa^2_2 \vspace{.2cm} \\
\wfa^2_1 & \wfa^2_2 & 0
\end{array}
\right]
\left[\begin{array}{c}
\widetilde{\chi} \widetilde{\bfw}^1_{,22} \vspace{.2cm}\\
\widetilde{\chi} \widetilde{\bfw}^2_{,22} \vspace{.2cm}\\
\widetilde{\chi} \widetilde{q}_{,2}
\end{array}
\right] = \left[\begin{array}{c}
\widetilde{\chi} f^1_2 \vspace{.2cm}\\
\widetilde{\chi} f^2_2 \vspace{.2cm}\\
\widetilde{\chi} f_3
\end{array}
\right]\,.
$$
Since $a^{jk}_{rs} \approx \delta^j_k \delta^r_s$ and $\det(\wfa) = \det(\nabla \theta_m)^{-1} = 1$,
$$
\det\left(\left[\begin{array}{ccc}
- (a^{jk}_{11}\circ\theta_m) \wfa^2_j \wfa^2_k & - (a^{jk}_{21}\circ\theta_m) \wfa^2_j \wfa^2_k & \wfa^2_1 \vspace{.2cm} \\
- (a^{jk}_{12}\circ\theta_m) \wfa^2_j \wfa^2_k & - (a^{jk}_{22}\circ\theta_m) \wfa^2_j \wfa^2_k & \wfa^2_2 \vspace{.2cm} \\
\wfa^2_1 & \wfa^2_2 & 0
\end{array}
\right]\right) \approx \big[(\wfa^2_1)^2 + (\wfa^2_2)^2\big]^2 \ne 0\,.
$$
Therefore, (\ref{rearrangement}) is solvable, and
\begin{align}
& \|\widetilde{\chi} \widetilde{\bfw}_{,22}\|_{L^2(B_+(0,r_m))} + \|\widetilde{\chi} \widetilde{q}_{,2}\|_{L^2(B_+(0,r_m))} \nonumber\\
&\qquad\quad \le C \Big[\|\widetilde{\chi} f_2\|_{L^2(B_+(0,r_m))} + \|\widetilde{\chi} f_3\|_{L^2(B_+(0,r_m))} \Big] \nonumber\\
&\qquad\quad \le C \big(1 + \|a\|_{L^\infty(\rO)}\big) \big(1 + \|a\|_{W^\infty(\rO)}\big) \|w\|_{H^1(\rO)} \label{wq_bdy_est}\\
&\qquad\qquad + C \big(1 + \|a\|_{L^\infty(\rO)}\big) \Big[\|f\|_{L^2(\rO)} + \|g\|_{H^{-0.5}(\Gamma)} \Big] + C_\ve \|g\|_{L^2(\Gamma)}\,. \nonumber
\end{align}
Estimate (\ref{elliptic_est}) then follows from the combination of (\ref{w_H2_loc_est}), (\ref{q_H1_loc_est}), (\ref{w_H2_bdy_est}), (\ref{q_H1_bdy_est}), and (\ref{wq_bdy_est}).

\noindent {\bf Part 2:} Next, we assume that $a\in W^{1,4}(\rO)$ and $(w,q) \in H^1_0(\rO) \times L^2(\rO)$ is a weak solution to (\ref{elliptic_Dirichlet}). Let $\ve$ be a smoothing parameter and $a^\ve \in W^{1,\infty}(\rO)$ be a sequence converging to $a$ in $W^{1,4}(\rO)$, and $a^\ve$ still satisfies the requirement (\ref{smallness}) by the following construction: let $\rE:W^{1,4}(\rO) \to W^{1,4}(\bbR^2)$ be an extension map, and $a^\ve$ is defined by
$$
a^\ve = \eta_\ve * (\rE a)\,.
$$
Let $w^\ve$ be the weak solution to
\begin{subequations}\label{elliptic_part2_temp}
\begin{alignat}{2}
- \big[(a^\ve)^{jk}_{rs} w^{\ve r}_{,j}\big]_{,k} + q^\ve_{,s} &= f^s &&\text{in}\quad\rO\,,\\
\div w^\ve &= 0 &&\text{in}\quad\rO\,,\\
w^\ve &= 0 \qquad&&\text{on}\quad\Gamma\,;
\end{alignat}
\end{subequations}
that is,
$$
\big((a^\ve)^{jk}_{rs} w^{\ve r}_{,j}, \varphi^s_{,k}\big)_{L^2(\rO)} = (f,\varphi)_{L^2(\rO)} \qquad\Forall \varphi \in H^1_{0,\div}(\rO) \equiv H^1_0(\rO) \cap H^1_\div(\rO)\,.
$$
The existence of a weak solution is guaranteed by the Lax-Milgram theorem, and $w^\ve$ satisfies the basic energy estimate
\begin{equation}\label{basic_est_part2}
\|w^\ve\|_{H^1(\rO)} \le C \|f\|_{L^2(\rO)}\,.
\end{equation}
(A different version of) the Lagrange multiplier lemma then suggests that there exists a unique $q\in L^2(\rO)$ such that
\begin{equation}\label{weak_part2}
\big((a^\ve)^{jk}_{rs} w^{\ve r}_{,j}, \varphi^s_{,k}\big)_{L^2(\rO)} + (q^\ve, \div \varphi)_{L^2(\rO)} = (f,\varphi)_{L^2(\rO)} \qquad\Forall \varphi \in H^1_0(\rO)
\end{equation}
and
\begin{equation}\label{basic_est_part2_q}
\|q^\ve\|_{L^2(\rO)} \le C \Big[\|\nabla w^\ve\|_{L^2(\rO)} + \|f\|_{L^2(\rO)}\Big] \le C \|f\|_{L^2(\rO)}\,.
\end{equation}
The argument in part 1 then can applied again (in the case $\ve=0$ and $g=0$) to show that $(w^\ve,q^\ve)\in H^2(\rO)\times H^1(\rO)$.

Since $(w^\ve,q^\ve)$ is a strong solution, we can perform the estimates as those in Part 1 in a slightly different fashion. To illustrate the idea, we focus on the interior estimates. The same as part 1, we use $\eta_\eps * [\chi^2 (\eta_\eps*w^{\ve s}_{,\ell})]_{,\ell}$ as a test function in (\ref{elliptic_weak_Dirichlet}), here we emphasize that here the convolution parameter is $\eps$ instead of $\ve$. This time we integrate by parts in $x_\ell$ first and then move the convolution around next (in part 1 we move the convolution around first then integrate by parts for the next step since the solution does not belong to $H^2(\rO)$ yet) to obtain that
\begin{align*}
& - \big((a^\ve)^{jk}_{rs} w^{\ve r}_{,j}, \eta_\eps * [\chi^2 (\eta_\eps*w^{\ve s}_{,\ell})]_{,\ell k}\big)_{L^2(\rO)} \\
&\qquad\quad = \big(\big[(a^\ve)^{jk}_{rs} w^{\ve r}_{,j}\big]_{,\ell}, \eta_\eps * [\chi^2 (\eta_\eps*w^{\ve s}_{,\ell})]_{,k}\big)_{L^2(\rO)} \\
&\qquad\quad = \big((a^\ve)^{jk}_{rs} \chi (\eta_\eps * w^{\ve r}_{,\ell j}), \chi (\eta_\eps*w^{\ve s}_{,\ell k})\big)_{L^2(\rO)} \\
&\qquad\qquad + \big(\comm{\eta_\eps*}{(a^\ve)^{jk}_{rs}} w^{\ve r}_{,\ell j}, [\chi^2 (\eta_\eps*w^{\ve s})_{,\ell}]_{,k}\big)_{L^2(\rO)} \\
&\qquad\qquad + 2 \big((a^\ve)^{jk}_{rs} w^{\ve r}_{,\ell j}, \eta_\eps * [\chi \chi_{,k} (\eta_\eps*w^{\ve s}_{,\ell})]\big)_{L^2(\rO)} \\
&\qquad\qquad + \big((a^\ve)^{jk}_{rs,\ell} w^{\ve r}_{,j}, \eta_\eps * [\chi^2 (\eta_\eps*w^{\ve s})_{,\ell}]_{,k}\big)_{L^2(\rO)} \,.
\end{align*}
By (\ref{comm_est_temp1}) and the basic energy estimate (\ref{basic_est_part2}),
\begin{align*}
& \|\chi (\eta_\eps* \nabla^2 w^\ve)\|^2_{L^2(\rO)} \le \big(\big[(a^\ve)^{jk}_{rs} w^{\ve r}_{,j}\big]_{,k}, \eta_\eps * [\chi^2 (\eta_\eps*w^{\ve s}_{,\ell})]_{,\ell}\big)_{L^2(\rO)} \\
&\qquad\quad + \|a - \id\otimes \id\|_{L^\infty(\rO)} \|\chi (\eta_\eps* \nabla^2 w^\ve)\|^2_{L^2(\rO)} \\
&\qquad\quad + C \eps \|a^\ve\|_{W^{1,\infty}(\rO)} \|\chi \nabla^2 w^\ve\|_{L^2(\rO)} \Big[\|f\|_{L^2(\rO)} + \|\chi (\eta_\eps * \nabla^2 w^\ve)\|_{L^2(\rO)}\Big] \\
&\qquad\quad + C \|a^\ve\|_{L^\infty(\rO)} \|\chi \nabla^2 w^\ve\|_{L^2(\rO)} \|f\|_{L^2(\rO)} \\
&\qquad\quad + C \|\nabla a^\ve\|_{L^4(\rO)} \|\chi \nabla w^\ve\|_{L^4(\rO)} \Big[\|f\|_{L^2(\rO)} + \|\chi (\eta_\eps * \nabla^2 w^\ve)\|_{L^2(\rO)}\Big] \,.
\end{align*}
Since $w^\ve \in H^2(\rO)$ independent of $\eps$, by Young's inequality and passing $\eps \to 0$ we find that
\begin{align}
& \|\chi \nabla^2 w^\ve\|^2_{L^2(\rO)} \le \big(\big[(a^\ve)^{jk}_{rs} w^{\ve r}_{,j}\big]_{,k}, \eta_\eps * [\chi^2 (\eta_\eps*w^{\ve s}_{,\ell})]_{,\ell}\big)_{L^2(\rO)} \nonumber\\
&\qquad\quad + \big(\|a - \id\otimes \id\|_{L^\infty(\rO)} + \delta \big) \|\chi \nabla^2 w^\ve\|^2_{L^2(\rO)} \label{part2_temp_ineq1}\\
&\qquad\quad + C_\delta \Big[\big(1+\|a^\ve\|^2_{L^\infty(\rO)}\big) \|f\|^2_{L^2(\rO)} + \|\nabla a^\ve\|^2_{L^4(\rO)} \|\chi \nabla w^\ve\|^2_{L^4(\rO)} \Big] \,. \nonumber
\end{align}
%Therefore, Young's inequality further implies that
%\begin{align}
%&\|\chi (\eta_\eps*\nabla^2 w^\ve)\|^2_{L^2(\rO)} \le - \big((a^\ve)^{jk}_{rs} w^{\ve r}_{,j}, \eta_\eps * [\chi^2 (\eta_\eps*w^{\ve s}_{,\ell})]_{,\ell k}\big)_{L^2(\rO)} \nonumber\\
%&\qquad + \big( \|a - \id\otimes \id\|_{L^\infty(\rO)} + \delta_1\big) \|\chi (\eta_\eps*\nabla^2 w^\ve)\|^2_{L^2(\rO)} \nonumber\\
%&\qquad + C_{\delta_1} \Big[\|\nabla a^\ve\|^2_{L^4(\rO)} \|\chi\nabla w^\ve\|^2_{L^4(\rO)} + \eps^2 \|a^\ve\|^2_{W^{1,\infty}(\rO)} \|\chi \nabla^2 w^\ve\|^2_{L^2(\rO)}\Big] \label{part2_temp_ineq1}\\
%&\qquad + C_\delta \big(\|a^\ve\|^2_{L^\infty(\rO)} + \eps^2 \|a^\ve\|^2_{W^{1,\infty}(\rO)}\big) \|f\|^2_{L^2(\rO)} + \delta \|\chi \nabla^2 w^\ve\|_{L^2(\rO)} \nonumber\\
%&\qquad + C \|\nabla a^\ve\|_{L^4(\rO)} \|\chi \nabla w^\ve\|_{L^4(\rO)} \|f\|_{L^2(\rO)} \,.\nonumber
%\end{align}
By interpolation,
\begin{align*}
C_\delta \|\nabla a^\ve\|^2_{L^4(\rO)} &\|\chi \nabla w^\ve\|^2_{L^4(\rO)} \le C_\delta \|\nabla a^\ve\|^2_{L^4(\rO)} \|\chi \nabla w^\ve\|_{L^2(\rO)} \|\chi \nabla w^\ve\|_{H^1(\rO)} \\
%&\le C \|\nabla w^\ve\|_{L^2(\rO)} \Big[\|\nabla w^\ve\|_{L^2(\rO)} + \|\chi \nabla^2 w^\ve\|_{L^2(\rO)}\Big] \\
&\le C_\delta \|\nabla a^\ve\|^2_{L^4(\rO)} \|\nabla w^\ve\|_{L^2(\rO)} \Big[1 + \|\chi \nabla^2 w^\ve\|_{L^2(\rO)} \Big] \\
&\le C_\delta \Big[%\|\nabla a^\ve\|^2_{L^4(\rO)} \|\nabla w^\ve\|_{L^2(\rO)}
1 + \|\nabla a^\ve\|^4_{L^4(\rO)} \|\nabla w^\ve\|^2_{L^2(\rO)} \Big] + \delta \|\chi \nabla^2 w^\ve\|^2_{L^2(\rO)} \,;
\end{align*}
thus by choosing $\delta >0$ small enough, with the smallness assumption (\ref{smallness}) we find that (\ref{part2_temp_ineq1}) suggests that
\begin{equation}\label{part2_temp_ineq2}
\begin{array}{l}
\displaystyle{} \|\chi \nabla w^\ve\|^2_{L^2(\rO)} \le \big(\big[(a^\ve)^{jk}_{rs} w^{\ve r}_{,j}\big]_{,k}, \eta_\eps * [\chi^2 (\eta_\eps*w^{\ve s}_{,\ell})]_{,\ell}\big)_{L^2(\rO)} \vspace{.15cm}\\
\displaystyle{}\hspace{68pt} + C \Big[1 + \|f\|^2_{L^2(\rO)} + \|\nabla a^\ve\|^4_{L^4(\rO)} \|\nabla w^\ve\|^2_{L^2(\rO)} \Big]\,.
\end{array}
\end{equation}
%\begin{equation}\label{part2_temp_ineq2}
%\begin{array}{l}
%\displaystyle{} \|\chi (\eta_\eps*\nabla^2 w^\ve)\|^2_{L^2(\rO)} \le - \big((a^\ve)^{jk}_{rs} w^{\ve r}_{,j}, \eta_\eps * [\chi^2 (\eta_\eps*w^{\ve s}_{,\ell})]_{,\ell k}\big)_{L^2(\rO)} \vspace{.2cm}\\
%\displaystyle{} \quad + \big(\delta + C \eps^2 \|a^\ve\|^2_{W^{1,\infty}(\rO)} \big) \|\chi \nabla^2 w^\ve\|^2_{L^2(\rO)} \vspace{.2cm}\\
%\displaystyle{} \quad + C_\delta \Big[\|a^\ve\|^2_{L^\infty(\rO)} + \|\nabla a^\ve\|^{4/3}_{L^4(\rO)} + \|\nabla a^\ve\|^4_{L^4(\rO)} + \eps^2 \|a^\ve\|^2_{W^{1,\infty}(\rO)}\Big] \|f\|^2_{L^2(\rO)} \vspace{.2cm}\\
%\displaystyle{} \quad + C \big(\|\nabla a^\ve\|_{L^4(\rO)} + \|\nabla a^\ve\|^2_{L^4(\rO)}\big) \|f\|^2_{L^2(\rO)} \,.
%\end{array}
%\end{equation}
On the other hand, we have
\begin{align*}
& - \big(q^\ve, \div\big[\eta_\eps * [\chi^2 (\eta_\eps*w^\ve_{,\ell})]_{,\ell}\big]\big)_{L^2(\rO)} = - \big(q^\ve , \eta_\eps * [\chi^2 (\eta_\eps*w^{\ve s}_{,\ell})]_{,\ell s}\big)_{L^2(\rO)} \nonumber\\
&\qquad\qquad = - 2 \big(q^\ve, \eta_\eps * [\chi \chi_{,s} (\eta_\eps* w^{\ve s}_{,\ell})\big]_{,\ell}\big)_{L^2(\rO)} \nonumber\\
&\qquad\qquad \le C \|q^\ve\|_{L^2(\rO)} \Big[\|\nabla w\|_{L^2(\rO)} + \|\chi (\eta_\eps * \nabla^2 w^\ve)\|_{L^2(\rO)}\Big] \nonumber\\
&\qquad\qquad \le C_\delta \|f\|^2_{L^2(\rO)} + \delta \|\chi (\eta_\eps*\nabla^2 w^\ve)\|^2_{L^2(\rO)}
\end{align*}
which implies (by passing $\eps \to 0$) that
\begin{align}
(q^\ve_{,s}, (\chi^2 w^\ve_{,\ell})_{,\ell}\big)_{L^2(\rO)} \le C_\delta \|f\|^2_{L^2(\rO)} + \delta \|\chi \nabla^2 w^\ve\|^2_{L^2(\rO)}
\label{part2_temp_ineq3}
\end{align}
Moreover,
\begin{equation}\label{part2_temp_ineq4}
- \big(f, (\chi^2 w^\ve_{,\ell})_{,\ell}\big)_{L^2(\rO)} \le C_\delta \|f\|^2_{L^2(\rO)} + \delta \|\chi \nabla^2 w^\ve\|^2_{L^2(\rO)} \,.
\end{equation}
By choosing $\delta > 0$ small enough, the combination of (\ref{part2_temp_ineq2}), (\ref{part2_temp_ineq3}) and (\ref{part2_temp_ineq4}) together with interpolation then suggests that
\begin{align*}
\|\chi \nabla^2 w^\ve\|^2_{L^2(\rO)} \le C \Big[1 + \|f\|^2_{L^2(\rO)} + \|\nabla a^\ve\|^4_{L^4(\rO)} \|\nabla w^\ve\|^2_{L^2(\rO)} \Big] \,.
\end{align*}

The interior $H^1$-estimate of $q^\ve$ is done in the same way as part 1 (via a different version of the Lagrange multiplier lemma), and the estimates near the boundary can be done in the same fashion (by integrating by parts first then moving the convolution around) as the interior estimates. Moreover, the estimates of the normal derivatives of $w^\ve$ and $q^\ve$ are obtained in the same way as part 1, so we conclude that
\begin{equation}\label{wq_est_Dirichlet_temp}
\|w^\ve\|^2_{H^2(\rO)} + \|q^\ve\|^2_{H^1(\rO)} \le C \Big[1 + \|f\|^2_{L^2(\rO)} + \|\nabla a^\ve\|^4_{L^4(\rO)} \|\nabla w^\ve\|^2_{L^2(\rO)} \Big] \,.
\end{equation}
Since $a\in W^{1,4}(\rO)$, the right-hand side of the estimate above is independent of $\ve$. Therefore, there is a sequence $\ve_j$ such that $(w^{\ve_j}, q^{\ve_j})$ converges weakly to some function $(w,q) \in H^2(\rO) \times H^1(\rO)$. Moreover, since $a^\ve \to a$ in $W^{1,4}(\rO)$, the variational identity (\ref{weak_part2}) converges to (\ref{elliptic_weak_Dirichlet}); thus $(w,q)$ must be the weak solution to (\ref{elliptic_Dirichlet}). Finally, estimate (\ref{elliptic_est2}) is a direct consequence of (\ref{wq_est_Dirichlet_temp}) by passing $\ve$ to the limit.
\end{proof}

\section{Proof of Lemma \ref{lem:key_elliptic_estimate}}\label{app:elliptic}
Before proceeding to the proof of Lemma \ref{lem:key_elliptic_estimate}, we state the following simple proposition which can be proved easily by interpolation.
\begin{proposition}\label{prop:key_elliptic_est}
Let $f\in L^2(0,\rT;H^{0.5}(\Gamma))$ $h_0 \in H^{2.5}(\Gamma)$, and $h \in \H_1(\rT)$ be a strong solution to
\begin{subequations}\label{h_elliptic}
\begin{alignat}{2}
h'' + \ve^2 h''_t &= f \qquad&&\text{on}\quad \Gamma\times (0,\rT)\,, \\
h &= g &&\text{on}\quad \Gamma\times \{t=0\}\,.
\end{alignat}
\end{subequations}
Then $h\in L^2(0,\rT;H^{2.5}(\Gamma))$ and satisfies
\begin{equation}\label{h_elliptic_estimate}
\|h'\|^2_{L^2(0,\rT;H^{1.5}(\Gamma))} \le C \Big[\ve^2 \|g\|^2_{H^{2.5}(\Gamma)} + \|f\|^2_{L^2(0,\rT;H^{0.5}(\Gamma))} \Big]\,.
\end{equation}
\end{proposition}

\noindent {\bf Lemma \ref{lem:key_elliptic_estimate}}
{\sl Let $(\bfw,\bfq,\bfh) \in \W(\rT) \times \Q(\rT) \times \H_1(\rT)$ be a strong solution to {\rm(\ref{NSreg1})}. Then}
$$
\int_0^\rT \hspace{-1pt} \|\bfh(t)\|^2_{H^{2.5}(\Gamma)} dt \le C \Big[1 + \|h_0\|^2_{H^{1.5}(\Gamma)} + \|\bfv\|^2_{\V(\rT)} + \|\bfq\|^2_{\Q(\rT)} \Big]\,. \eqno{\rm(\ref{hH25_estimate})}
$$
{\sl In particular, the corresponding $\bfJ$, $\bfA$ and $\Psi$ satisfy}
$$
\begin{array}{l}
\displaystyle{} \int_0^\rT \hspace{-2pt}\Big[\|\bfA\|^2_{H^2(\rO)} + \|\bfJ\|^2_{H^2(\rO)} + \|\nabla \psi\|^2_{H^2(\rO)} \Big] dt \vspace{.1cm}\\
\displaystyle{} \hspace{50pt}\le C \Big[1 + \|h_0\|^2_{H^2(\Gamma)} + \|\bfv\|^2_{\V(\rT)} + \|\bfq\|^2_{\Q(\rT)}\Big]\,.
\end{array}
\eqno{\rm(\ref{JApsi_H2_estimate})}
$$
\begin{proof}
We note that (\ref{NSreg1}c) can be rewritten as
\begin{equation}\label{basic_elliptic}
\frac{(\opbh) \bfh''}{\bfg^{3/2}} \rN + \ve^2 \bfw^{\pprime\prime} = f \qquad\text{on}\quad\Gamma\times (0,\rT)\,,
\end{equation}
where $f$ is given by
$$
f^s = \psi^i_{,s} \big[\bfA^j_\ell \bfv^i_{,j} + \bfA^j_i \bfv^\ell_{,j} - \bfq \delta^\ell_i \big] \bfA^k_\ell \rN_k + \frac{(1 \hspace{-1pt}+\hspace{-1pt} \rb_0 \bfh_{\ve\ve})^2 + 2 \bfh^{\prime2}_{\ve\ve} + \bfh_{\ve\ve} \bfh'_{\ve\ve} \rb_0^\prime}{\bfg^{3/2}} \rN_s
$$
and satisfies
$$
\|f\|_{L^2(0,\rT;H^{0.5}(\Gamma))} \le C \Big[\|\bfv\|_{\V(\rT)} + \|\bfq\|_{\Q(\rT)} + 1 \Big]\,.
$$
We first show that $\bfh$ indeed belongs to $L^\infty(0,\rT;H^{2.5}(\Gamma))$ (with an $\ve$-dependent estimate), and then use this fact to obtain estimate (\ref{hH25_estimate}).

%To see why $\bfh \in L^\infty(0,\rT;H^{2.5}(\Gamma))$, we t
Taking the inner product of (\ref{basic_elliptic}) and $\rN$, by the Leibniz rule we obtain that
\begin{align*}
\Big[\smallexp{$\displaystyle{}\frac{(\opbh) }{\bfg^{3/2}}$}\bfh + \ve^2 (\bfw\cdot \rN)\Big]'' &= \widetilde{f}
\end{align*}
where $\widetilde{f}$ is given by
$$
\widetilde{f} \equiv f \cdot \rN + \ve^2 \bfw\cdot \rN'' + 2 \ve^2 \bfw^\pprime\cdot \rN' + \Big[\smallexp{$\displaystyle{}\frac{(\opbh) }{\bfg^{3/2}}$}\Big]'' \bfh + 2 \Big[\smallexp{$\displaystyle{}\frac{(\opbh) }{\bfg^{3/2}}$}\Big]' \bfh'\,.
$$
We note that due to the convolution, $\widetilde{f} \in L^2(0,\rT;H^{0.5}(\Gamma))$. Therefore, by elliptic regularity,
$$
\smallexp{$\displaystyle{}\frac{(\opbh) }{\bfg^{3/2}}$} \bfh + \ve^2 (\bfw\cdot \rN) = g \in L^2(0,\rT;H^{2.5}(\Gamma))\,.
$$
Since $\bfw$ satisfies $\bfh_t = \smallexp{$\displaystyle{}\frac{\bfw\cdot \rN}{\opbh}$}$, we can further rewrite the equation above as
$$
\bfg^{-3/2} \bfh + \ve^2 \bfh_t = \frac{g}{\opbh}\,.
$$
Solving the ODE above using the method of integrating factor, we find that $\bfh \in L^\infty(0,\rT;H^{2.5}(\Gamma))$.

Now we rewrite (\ref{basic_elliptic}) as
$$
\bfh'' \rN + \ve^2 \frac{\bfw^{\pprime\prime}}{\opbh} = \frac{f}{\opbh} + (1 - \bfg^{-3/2}) \bfh'' \rN
$$
which, by projecting to the normal direction, further suggests that
\begin{equation}\label{h_eq_elliptic1}
\bfh'' + \ve^2 \bfh_t'' = \widebar{f}
\end{equation}
with $\widebar{f}$ satisfying
\begin{align*}
\|\widebar{f}\|_{H^{0.5}(\Gamma)} &\le C \Big[\|f\|_{H^{0.5}(\Gamma)} \hspace{-2pt}+\hspace{-1pt} \big(\|\bfh\|_{H^{1.55}(\Gamma)} \hspace{-2pt}+\hspace{-1pt} \|\bfg \hspace{-1pt}-\hspace{-1pt} 1\|_{H^{0.55}(\Gamma)} \big) \|\bfh\|_{H^{2.5}(\Gamma)} \hspace{-2pt}+\hspace{-1pt} \ve^2 \|\bfw\|_{H^{1.5}(\Gamma)} \Big] \\
&\le C \Big[\|f\|_{H^{0.5}(\Gamma)} \hspace{-2pt}+\hspace{-1pt} \varsigma \|\bfh\|_{H^{2.5}(\Gamma)} \hspace{-2pt}+\hspace{-1pt} \|\bfh\|_{H^{1.55}(\Gamma)} \|\bfv\|_{H^{1.5}(\Gamma)} \Big] \\
&\le C \Big[\|f\|_{H^{0.5}(\Gamma)} + \|\bfv\|_{H^2(\rO)} + \varsigma \|\bfh\|_{H^{2.5}(\Gamma)} \Big]\,,
\end{align*}
where we use $\bfw = \bfJ \bfA \bfv$ and (\ref{smallness_of_bfh}) to derive the estimate above. Therefore, by Proposition \ref{prop:key_elliptic_est},
\begin{align}
&\int_0^\rT \hspace{-2pt}\|\bfh'\|^2_{H^{1.5}(\Gamma)} d\tilde{t} \hspace{-1pt}\le\hspace{-1pt} C \Big[\ve^2 \|h_{0\ve}\|^2_{H^{2.5}(\Gamma)} \hspace{-2pt}+\hspace{-1pt} \int_0^\rT \hspace{-2pt} \big[\|f\|^2_{H^{0.5}(\Gamma)} \hspace{-2pt}+\hspace{-1pt} \|\bfv\|^2_{H^2(\rO)} \hspace{-2pt}+\hspace{-1pt} \varsigma \|\bfh\|^2_{H^{2.5}(\Gamma)} \big] d\tilde{t}\Big] \nonumber\\
&\qquad\quad \le C \Big[\|h_0\|^2_{H^{1.5}(\Gamma)} \hspace{-1pt}+\hspace{-1pt} \|\bfv\|^2_{\V(\rT)} \hspace{-1pt}+\hspace{-1pt} \|\bfq\|^2_{\Q(\rT)} \hspace{-1pt}+\hspace{-1pt} 1\Big] \hspace{-1pt}+\hspace{-1pt} C \varsigma \int_0^\rT \hspace{-2pt}\|\bfh\|^2_{H^{2.5}(\Gamma)} d\tilde{t} \,. \label{h_H25_est_temp1}
\end{align}
On the other hand, the evolution equation (\ref{he_eq}) implies that
$$
\|\bfh(t)\|_{L^2(\Gamma)} \le \|h_0\|_{L^2(\Gamma)} \hspace{-1pt}+\hspace{-1pt} \int_0^t \Big\|\smallexp{$\displaystyle{}\frac{(\bfJ \bfA^\rT \rN) \cdot \bfv}{\opbh}$}\Big\|_{L^2(\Gamma)} d\tilde{t} \le \|h_0\|_{L^2(\Gamma)} \hspace{-1pt}+\hspace{-1pt} \sqrt{t} \|\bfv\|_{L^2(0,t;H^2(\rO))}\,,
$$
and (\ref{hH25_estimate}) following from the combination of (\ref{h_H25_est_temp1}) and the estimate above since $\varsigma \ll 1$.
\end{proof}

\vspace{.1in}

\noindent {\bf Acknowledgments.} AC was supported by the National Science Council (Taiwan) under grant 100-2115-M-008-009-MY3, and YL was supported by the National Science Council (Taiwan) under grant
102-2115-M-008-001.

\end{document}